\documentclass[reqno,twoside,11pt]{amsart}

\usepackage{amsmath,amsfonts,calrsfs,fullpage,amssymb,xcolor,verbatim,eucal,yfonts,mathrsfs}

\usepackage{mathtools}
\usepackage{xcolor}
\mathtoolsset{showonlyrefs}

\usepackage[normalem]{ulem}

 \usepackage{xcolor}
\newtheorem{Theorem}{Theorem}[section]
\newtheorem{Definition}{Definition}
\newtheorem{Proposition}[Theorem]{Proposition}
\newtheorem{Lemma}[Theorem]{Lemma}
\newtheorem{Corollary}[Theorem]{Corollary}
\newtheorem{Remark}[Theorem]{Remark}
\newtheorem{construction}[Theorem]{Construction}

\newtheorem{Hypothesis}{Hypothesis}
\makeatletter
\@addtoreset{equation}{section}

\makeatother

\def\a{\alpha}

\def\R{\mathbb R}
\def\N{\mathbb N}

\def\E{\mathbb E}
\def\P{\mathbb P}

\def\eps{\epsilon}
\def\ds{\displaystyle}

\def\e{\epsilon}

\newcommand{\blue}{\textcolor{blue}}

\newcommand{\oc}{\varrho}

\title{
Large deviations for long-time occupation measures of  stochastic evolution equations  with small, asymptotically  rough noise
}\date{}

\author[A. Budhiraja]{amarjit Budhiraja}\address{Department of Statistics and Operations Research\\University of North Carolina,
Chapel Hill}
\email{budhiraj@email.unc.edu}
\thanks{A. Budhiraja was partially supported by NSF Grant DMS-2506010 (2025-2028), {\em  Some Interacting Particle Systems and Applications}}

\author[S. Cerrai]{Sandra Cerrai}
\address{Department of Mathematics\\
University of Maryland, College Park\\
}
\email{cerrai@umd.edu}
\thanks{S. Cerrai was partial  supported by NSF Grant DMS-2348096 (2024-2027), {\em 
 Multiscale Analysis of Infinite-Dimensional Stochastic Systems}}

\subjclass[2010]{}


\begin{document}

\begin{abstract}
We study the long-time, small-noise behavior of a class of
dissipative stochastic evolution equations in a separable Hilbert
space, driven by a cylindrical Wiener process whose covariance
degenerates to a limiting operator in the strong operator topology.
A prototypical example is a stochastic reaction-diffusion equation
on a bounded domain with a spatially homogeneous but spectrally
regularized noise that becomes spatially rough in the limit.

We establish a large deviation principle for the family of 
occupation measures as the time horizon becomes large, the noise intensity
tends to zero, and the noise becomes increasingly rough.  
The result covers a broad class of dissipative stochastic
evolution equations in infinite dimensions, including equations
driven by asymptotically rough cylindrical noise whose covariance
need not be trace class. The finite-dimensional
setting has been studied in the recent work of Budhiraja and Zoubouloglou. The infinite-dimensional setting introduces several substantial new difficulties. Both the upper and lower bound arguments require a careful handling  of the unbounded evolution operator and inverse covariance operator: the former in establishing a key orthogonality property, and the latter in a central controllability construction. In addition, the increasing roughness of the driving noise must be balanced against its vanishing amplitude through estimates that remain uniform throughout this singular perturbation regime.
Proofs combine analytic semigroup
techniques, fractional domain space estimates and a careful treatment
of stochastic convolutions in weighted spaces.

The rate function is given by a simple explicit formula and its evaluation at a given probability measure on the state space is expressed as the average, with respect to the
measure, of the squared Cameron-Martin cost of
canceling the deterministic drift at each point.  The proof follows
the weak convergence approach based on the Bou\`{e}-Dupuis variational formula and relies on an explicit
construction of near-optimal controls, built by alternating
travel phases, which steer the process between prescribed target state
in the support of the occupation measure, and hold phases, during which the process is stabilized
near a target state and the occupation measure is shaped.
\end{abstract}

 \maketitle
 
 \tableofcontents
 
\section{Introduction}
We study a  class of semilinear stochastic evolution equations on a separable Hilbert space $H$ in a coupled asymptotic regime involving vanishing noise, increasingly long time horizons, and increasingly rough driving noise. More precisely, for each $\epsilon\in(0,1)$, we consider
\begin{equation}
 \label{eq1}
 du^\e(t)=\left(A u^\e(t)+Q\nabla U(u^\e(t))	\right)\,dt+s(\e) dw^\epsilon(t),\ \ \ \ \ u^\e(0)=u_0 \in\,H.
 \end{equation}
Here $A:D(A)\subset H\to H$ is the generator of an analytic semigroup $e^{tA}$, $Q$ is a symmetric and positive definite bounded linear operator on $H$, and $U:H\to \R$ is a Fr\'echet differentiable mapping such that $\nabla U:H\to H$ is Lipschitz continuous. Moreover $w^\epsilon(t)$, $t\geq 0$, is a cylindrical Wiener process with covariance $B_\epsilon$ such that $B^{1/2}_\epsilon\to Q^{1/2}$ strongly in $H$, as $\epsilon\to 0$. Finally $s:(0,1)\to (0,1)$ is such that
$s(\e)\to 0$, as $\e\to 0$.
 The parameter $\epsilon$ simultaneously governs the noise amplitude through $s(\epsilon)$, the observation time horizon, and the spatial roughness of the driving noise through the family of covariance operators $(B_\epsilon)_{\epsilon>0}$.
Consider the collection of occupation (empirical) measures
\begin{equation}\label{as20}\nu^\epsilon:=\epsilon\int_0^{1/\e} \delta_{u^\epsilon(t)}\,dt\end{equation}
viewed as  random probability measures on $H$. The main goal of this paper is to establish a large deviation principle for the family $(\nu^\epsilon)_{\epsilon>0}$, as $\eps \to 0$.

The assumptions on the model are collected in Hypotheses~\ref{H1} and~\ref{H2}. Hypothesis~\ref{H1} specifies the structural assumptions on the linear operators and the noise, including the dissipativity of $A$, the commutation of $A$ and $Q$,  the convergence of the covariance operators $B_\epsilon$ to $Q$, as $\epsilon\to0$,
 and a condition that suitably balances the increasing roughness of the driving noise with its vanishing amplitude. Hypothesis~\ref{H2} imposes regularity and dissipativity assumptions on the nonlinear potential $U$. Together, these assumptions ensure that 
 \eqref{eq1} is globally well posed and has suitable  global stability properties.

 Hypotheses~\ref{H1} and \ref{H2}  include spatially rough driving noises, including the case $Q=I$ and two prototypical examples to keep in mind are the following.
 
 \smallskip
 
 {\bf Example 1.} Take $\nabla U=0$ and $B_\e=Q$. Then   \eqref{eq1} becomes the following Ornstein-Uhlenbeck equation
 \begin{equation}
 \label{eq1-intro-ex1}
 du^\e(t)=A u^\e(t)\,dt+ s(\e)\,dw^Q(t),\ \ \ \ \ u^\epsilon(0)=u_0,	
 \end{equation}
 where $w^Q(t)$ is a cylindrical Wiener process with covariance $Q$.

 {\bf Example 2.} Assume  $H=L^2(D)$, for some bounded smooth domain $D\subset \R^d$, with $d\geq 1$, and take $A$ to be the realization of the Laplace operator in $L^2(D)$, endowed with Dirichlet boundary conditions. We fix $f:\R\to \R$ Lipschitz continuous and define
 \[U(h)=\int_D \int_0^{h(x)} f(t)\,dt\,dx,\ \ \ \ \ h \in\,H.\]
 Take $Q=I$ and  assume $B_\epsilon$ is a compact operator, for every $\e \in\,(0,1)$, such that $B_\epsilon\to I$, strongly, as $\epsilon\to 0$. Then, equation \eqref{eq1} becomes the following stochastic reaction-diffusion equation
\begin{equation}\label{eq1-intro-ex2}
\left\{\begin{array}{l}
\ds{\partial_t u^\epsilon(t, x) = \Delta u^\epsilon(t, x) + f(u^\epsilon(t, x)) + s(\e)\, \partial_t w^{B_\epsilon}(t, x), }\\[10pt]
\ds{u^\epsilon(t, x) = 0,\ \ \ \ x \in \partial D, \ \ \ \ \ \ 
u^\epsilon(0, x) = u_0(x).}
\end{array}\right.
\end{equation}
Here  $w^{B_\epsilon}(t)$ is a cylindrical Wiener process on $L^2(D)$, having covariance $B_\epsilon$ and $f:\mathbb{R}\to\mathbb{R}$ is some Lipschitz-continuous mapping. 

\smallskip
When $s(\epsilon)\equiv 1$, the above large deviation problem corresponds to the  classical theory of Donsker-Varadhan on empirical measures of Markov processes which under suitable stability conditions provides  a large deviation principle for the occupation measures of an ergodic Markov process, on some Polish state space $S$, over long time horizons (see \cite{DVI}-\cite{DVIII}).  Under certain conditions, the rate function for this setting is given by the variational formula
\begin{equation}\label{eq:bc954}
I(\gamma)
:= -\inf_{u\in\mathcal D(L),\; u>0}
\int_S
\frac{Lu}{u} \, d\gamma, \ \ \ \ \  \gamma \in \mathcal{P}(S),
\end{equation}
where $\mathcal{P}(S)$ denotes the space of probability measures on $S$ equipped with the topology of weak convergence, and $L$ denotes the infinitesimal generator of the Markov semigroup. 

Donsker-Varadhan theory for specific infinite-dimensional stochastic evolution equations has been investigated by a number of authors. 
 Initial works on this topic considered fixed nondegenerate noise and, using a criterion of \cite{Wu2001}, 
established occupation measure large deviation principles under conditions ensuring ergodicity and sufficient regularity of the associated Markov semigroup (e.g. strong Feller property)\cite{Gourcy2007, Gourcy2007NS}. More recently, Donsker-Varadhan type large deviation principles for occupation measures have been obtained for several classes of randomly forced dissipative PDEs, including stochastic Navier-Stokes equations, complex Ginzburg-Landau equations, and damped nonlinear wave equations. These works develop Kifer-type and multiplicative-ergodic approaches, with coupling, controllability, and, in the degenerate Fourier-space setting, Malliavin-calculus arguments, to verify the required mixing or uniform Feller properties
\cite{Jaksic2015,Jaksic2018,Martirosyan2019,Nersesyan2019,NersesyanPengXu2023}.

A key feature of the classical Donsker-Varadhan theory is that the rate function is finite only for probability measures that are absolutely continuous with respect to the invariant measure and this absolute continuity with respect to the stationary measure plays an essential role in both the form of the rate function and its analysis.
This picture changes dramatically when the noise amplitude vanishes. In that regime, the invariant measures themselves converge to singular limits, and the probability measures describing the large deviation behavior are typically singular with respect to the invariant measure of the limiting deterministic dynamics. 

The point that Donsker-Varadhan methods are not applicable for this singular noise setting was  discussed in \cite{BudhirajaZoubouloglou2024}, which considered a finite dimensional model and established a large deviation principle for  occupation measures of small-noise diffusions over long time horizons. The resulting rate function no longer involves relative entropies or variational formulas of the form in \eqref{eq:bc954}, but instead admits an explicit action representation. Specifically,~\cite{BudhirajaZoubouloglou2024}
studied  occupation measures for the finite dimensional small-noise diffusion of the form 
\[
dY^\epsilon(t)
=
-\psi(Y^\epsilon(t))\,dt
+
 s(\epsilon)
 \sigma(Y^\epsilon(t))\,dB(t), \qquad t\geq 0.
\]
Under suitable nondegeneracy, dissipativity and stability assumptions on the coefficients, they established a large deviation principle for the occupation measure of the form in \eqref{as20} (replacing $u^{\eps}$ with $Y^{\eps}$) with speed $(\epsilon s^2(\epsilon))^{-1}$ and the explicit rate function
\[
I(\gamma)
=
\frac12
\int_{\mathbb R^d}
\left|\sigma^T(y)a^{-1}(y)\psi(y)\right|^2
\,\gamma(dy), \ \ \ \ \ \  \gamma \in \mathcal{P}(\mathbb R^d),
\]
where $a=\sigma\sigma^T$.

The present work extends this program to the infinite-dimensional setting of equation~\eqref{eq1}. Under Hypotheses~\ref{H1} and~\ref{H2}, we show that the empirical occupation measures satisfy a large deviation principle on $\mathcal P(H)$ with speed $(\epsilon s^2(\epsilon))^{-1}$ and the explicit rate function
\[
I(\gamma)
=
\frac12
\int_H
\bigl|Q^{-1/2}Ax+Q^{1/2}\nabla U(x)\bigr|_H^2\,\gamma(dx), \ \ \ \ \ \gamma \in \mathcal{P}(H),
\]
with the convention that the integrand is equal to $+\infty$ whenever
$x\notin \mathcal D$, where $\mathcal D$ is the domain defined in
\eqref{as35}.

\smallskip

While the overall strategy - the weak convergence approach together with the Bou\'e-Dupuis variational representation -  is modeled on that of \cite{BudhirajaZoubouloglou2024}, the passage from finite to infinite dimensions is far from routine, and the difficulties we overcome are of three distinct kinds: the asymptotic roughness of the noise, the unboundedness of the operators $A$ and $Q^{-1}$ in the analysis of the upper bound, and the same unboundedness in the controllability construction underlying the lower bound. We discuss each in what follows.

\emph{Asymptotically rough noise.} The most distinctive feature of our setting is that the driving noise is not fixed but becomes increasingly rough as $\eps\to0$, with covariances $B_\eps$ that need not be trace class and that converge to $Q$ only in the strong operator topology. This is a genuinely new phenomenon, absent both from the finite-dimensional model of \cite{BudhirajaZoubouloglou2024} and from the earlier Donsker--Varadhan analyses of infinite-dimensional dynamics \cite{Gourcy2007,Gourcy2007NS,Jaksic2015,Jaksic2018,Martirosyan2019,Nersesyan2019,NersesyanPengXu2023}, which all treat a fixed, nondegenerate noise. Here the growing roughness must be balanced against the vanishing amplitude $s(\eps)$. Namely, the stochastic convolution loses spatial regularity as $\eps\to0$, and the entire argument must be carried out through moment and fractional-regularity estimates that remain uniform throughout this singular perturbation regime. This balance is quantified by Hypothesis~\ref{H1}(4) and, in the representative example of Remark~\ref{QAcommute}(2), by the condition $\gamma\beta<2$, which expresses that the divergence of the relevant trace norms as the regularization is removed is slow enough to be compensated by the factor $s^2(\eps)$.

\emph{Unbounded operators and the upper bound.} In the finite-dimensional case, the explicit form of the rate function is extracted from the weak limit of the controlled occupation measures through an orthogonality identity \cite[Lemma~4.4]{BudhirajaZoubouloglou2024}. In our setting this identity (Lemma~\ref{orthogonality}) is considerably more delicate, since it involves the unbounded operators $A$ and $Q^{-1}$, and It\^o's formula cannot be applied directly to test functions on $H$ built from them. We circumvent this obstruction by testing only against functions depending on finitely many modes of the eigenbasis of $A$, letting the number of modes grow only after the limit $\eps\to0$ has been taken. This mode-truncation is also what makes the effective domain of the rate function, namely $\mathcal D=\{x\in D(A):Ax\in\mathrm{Range}(Q^{1/2})\}$ in~\eqref{as35}, emerge naturally.

\emph{Unbounded operators and the lower bound.} The lower bound is based  on an explicit, recursive construction of near-optimal controls that alternate \emph{travel} phases, which rapidly steer the state between prescribed points $x_1,\dots,x_k$ in the support of the target measure, and \emph{hold} phases, during which the state is stabilized near a target point for a length of time proportional to the mass the measure places there (Subsection~\ref{lb-control-construction}). Both phases are affected by the unboundedness of the operators. The travel control is defined explicitly in the eigenbasis of $A$, and its cost is controlled by fractional norms $|(-A)^{1/2}\cdot|_H$, which forces the target points to be chosen in $\mathcal D$. The hold phase, on the other hand, requires a control that cancels the drift $Ax_*+Q\nabla U(x_*)$ at the target, which is possible only when $Ax_*\in\mathrm{Range}(Q^{1/2})$, that is, it requires working with the unbounded operator $Q^{-1/2}A$, and relies on a new stabilization estimate for the resulting shifted stochastic equation, established in Section~\ref{stab-lemma} (Lemma~\ref{stabilization}).

Finally, we emphasize a structural relaxation relative to \cite{BudhirajaZoubouloglou2024}. We do not assume that $U$ is twice differentiable, but only that $U\in C^{1,1}$, i.e.\ that $\nabla U$ is Lipschitz. This is essential in order to include the stochastic reaction-diffusion equation of Example~2, in which $\nabla U$ is a Nemytskii operator on $L^2(D)$ that is Fr\'echet differentiable only when the reaction term is affine.

\smallskip

As we already mentioned, the present work concerns the simultaneous asymptotic regime in which the noise amplitude vanishes while the observation time becomes large. This should be contrasted with the classical Freidlin-Wentzell theory, which considers the asymptotic behavior of invariant measures by first passing to the long-time limit for each fixed noise level and then studying the small-noise limit of the resulting stationary distributions. In finite dimensions, this theory is developed in the monograph of 
Freidlin and Wentzell~\cite{FreidlinWentzell}. Infinite-dimensional analogues for stochastic evolution equations have been established in a number of settings, including the papers~\cite{Sowers, CerraiRoeckner2005, BrzeCerr1998}, and many subsequent contributions.

The two asymptotic regimes capture fundamentally different phenomena. The Freidlin-Wentzell theory describes exponentially unlikely deviations of the invariant measure from its concentration near stable equilibria and is governed by quasipotentials. In contrast, the present work studies fluctuations of  occupation measures over observation windows whose lengths diverge simultaneously with the vanishing of the noise. As a consequence, the associated rate function  instead takes the explicit action form established in Theorem~\ref{main}. A related but distinct line of work concerns the connection between finite-time dynamical large deviations and stationary large deviations. In particular, Bertini, Gabrielli, and Landim~\cite{BertiniGabrielliLandim} showed that, for a class of finite dimensional diffusions, invariant measure rate function can be characterized through the asymptotic behavior of the finite-time dynamical action as the time horizon tends to infinity.

\smallskip

The remainder of the paper is organized as follows. Section~\ref{assumptions} introduces the notation, assumptions, and main result, and describes the Laplace principle formulation used in the proof. Section~\ref{varformula} presents the main variational representation for Brownian functionals used in our work. Section~\ref{a-priori} proves tightness of the occupation measures associated with a near-optimal collection of controlled processes. By studying the limits of these controlled occupation measures, Section~\ref{upperbound} establishes the Laplace principle upper bound, with the key orthogonality identity proved in Section~\ref{as30}. Section~\ref{stab-lemma}, in preparation for the proof of the large deviation lower bound, proves a stabilization estimate around an arbitrary  point in the state space.  Section~\ref{lowerbound} uses this estimate as a key ingredient to a recursive construction of a near optimal control to prove the Laplace principle lower bound. Finally, Section~\ref{compactness} proves compactness of the level sets of the rate function.

\section{Notations, assumptions  and main results}
\label{assumptions}

Throughout the paper,  $H$ is a separable Hilbert space, endowed with the scalar product $\langle\cdot,\cdot\rangle_H$ and the corresponding norm $\vert \cdot\vert_H$. We will denote by $\mathcal{L}(H)$ the space of bounded linear operators in $H$ and by  $\mathcal{L}^+_1(H)$ and $\mathcal{L}_2(H)$ the subspaces of non-negative operators having finite trace and of Hilbert-Schmidt operators, respectively. Given an arbitrary Polish space $\mathcal{X}$, we will denote by $B_b(\mathcal{X})$ the space of bounded Borel measurable mappings $G:\mathcal{X}\to \R$ and by $C_b(\mathcal{X})$ the subspace of bounded continuous functions, with
\[\|G\|_\infty:=\sup_{x \in\,\mathcal{X}} |G(x)|.\] If $d_{\mathcal{X}}$ is the distance in $\mathcal{X}$, we denote by $\text{Lip}_b(\mathcal{X})$ the subspace of Lipschitz continuous functions, with 
\[\|G\|_{\text{\tiny{Lip}}}:=\|G\|_\infty+[G]_{\text{\tiny{Lip}}}:=\|G\|_\infty+\sup_{x\neq y}\frac{|G(x)-G(y)|}{d_{\mathcal{X}}(x,y)}.\]
Finally, for every $\kappa>0$, we  denote by $\text{Lip}_{b, \kappa}(\mathcal{X})$ the subspace of functions having Lipschitz norm bounded by $\kappa$.
Throughout the paper, $c,C,c_1,c_2$, etc. denote positive constants whose values may change from line to line, even within the same display.

\smallskip

Recall the stochastic evolution equation introduced in \eqref{eq1}, given in terms of the operators $A,Q$, the nonlinearity $U$, cylindrical Brownian motion $w^{\eps}$ with covariance $B_{\eps}$, and noise intensity $s(\eps) \to 0$. 
 In what follows, we will assume the following conditions for the operators $A$ and $Q$, and for the noise $w^\e(t)$.
\begin{Hypothesis}
	\label{H1} $\;$
	\begin{enumerate}
	\item[1.]  $A:D(A)\subset H\to H$ is a self-adjoint operator with compact resolvent, that generates an analytic semigroup $e^{tA}$, $t\geq 0$. Moreover, there exists $\lambda>0$ such that
	\[\langle A x,x\rangle_H\leq -\lambda \,|x|_H^2,\ \ \ \ \ x \in\,D(A).\]
	\item[2.] The operator $Q$ belongs to $\mathcal{L}(H)$ and  $\operatorname{Ker} Q=\{0\}$. Moreover, $Q$  commutes with $A$, 
	in the sense that
\[
   Q\, e^{tA}  =  e^{tA}\, Q, \ \ \ \ \  t \geq 0,
\]
and for every $t>0$, $Q^{-1}Ae^{tA}$ is a bounded linear operator on $H$.
	 \item[3.] 	
	For every $\e \in\,(0,1)$ the process $w^\epsilon(t)$, $t\geq 0$, is a cylindrical Wiener process defined on the stochastic basis $(\Omega, \mathcal{F},  (\mathcal{F}_t )_{t\geq 0}, \mathbb{P})$, and having covariance $B_\epsilon$. The operator $B_\e$ is symmetric and non-negative, commutes with $A$,  and 	\[\lim_{\e\to 0} B_\epsilon^{1/2} x=Q^{1/2} x,\ \ \ \ \ x \in\,H.\]

	\item[4.] There exists $\alpha \in\,(0,1/2)$ such that for every
$\epsilon \in\,(0,1)$ it holds
\[\int_0^t (s\wedge 1)^{-2\alpha}e^{sA}B_\epsilon e^{sA}\,ds
  \in\,\mathcal{L}^+_1(H),\ \ \ \ \ t \geq 0.\]
Moreover,
\begin{equation}\label{as1}
\sup_{\e \in\,(0,1)} s^2(\epsilon)
  \int_0^{\infty} (s\wedge 1)^{-2\alpha}\,
  \mathrm{tr}\,\bigl(e^{sA}B_\epsilon e^{sA}\bigr)\,ds <\infty.
\end{equation}
	\end{enumerate}
	\end{Hypothesis}
	
	\begin{Remark}\label{QAcommute}
{\em {\em 1.} 
Hypotheses ~\ref{H1}.1 and ~\ref{H1}.2 
  say that the operators $A$ and $Q$ admit a common orthonormal
   eigenbasis $ (e_k )_{k \in \N}$ of $H$, namely \[A e_k = -\alpha_k e_k,\ \ \ Q e_k = q_k e_k,\] for some positive nondecreasing sequence $(\alpha_k)_{k \in \N}$ such that $\alpha_k \to \infty$ as $k \to \infty$, and a bounded
   sequence $ (q_k )_{k \in\,\mathbb{N}} \subset [0, \|Q\|_{\mathcal{L}(H)}]$.
A standard consequence is that $Q$
leaves $D(A)$ invariant and $Q A x = A Q x$ for every $x \in D(A)$;
similarly, $Q$ commutes with the fractional powers $(-A)^\gamma$ for
every $\gamma \in \R$, and with the projections $\pi_n$ onto the first
$n$ eigenspaces of $-A$, for every $n \in \N$. 
A natural class of
examples is  given by $Q = g(-A)$ for some   Borel measurable
function $g : [0, \infty) \to [0, \infty)$ such that its restriction to $\sigma(-A)$ is bounded and $1/g(x)$ grows subexponentially as $x\to \infty$, including in particular
the choices $Q = (-A)^{-\gamma}$, for $\gamma \geq 0$.

\smallskip

{\em 2.}  A typical example of an operator $B_\e \in\,\mathcal{L}(H)$ satisfying Hypotheses \ref{H1}.3 and \ref{H1}.4 is
\[B_\e h=\sum_{k=1}^\infty q_k(1+s(\epsilon)^\gamma \alpha_k)^{-\beta}\langle h,e_k\rangle_H e_k,\ \ \ \ \ \ h \in\,H,\]
for some { $\gamma>0$, $\beta \ge 0$} that have to be chosen appropriately in order to fulfill those hypotheses. For example, if we take $q_k\sim \alpha_k^{-\vartheta}$, for some $\vartheta\geq 0$, and  have $\alpha_k\sim k^\varrho$, for some $\varrho>0$, we have
{\[q_k(1+s(\epsilon)^\gamma \alpha_k)^{-\beta}{\sim } k^{-\vartheta\varrho}(1+s(\e)^{\gamma} k^{\varrho})^{-\beta},\]}
 so that
 {
\begin{multline*}
\int_0^{+\infty}(s\wedge 1)^{-2\alpha}\mathrm{tr}\,\bigl(e^{sA}B_\epsilon e^{sA}\bigr)\,ds\\
\leq c \sum_{k=1}^\infty k^{-\vartheta\varrho}(1+s(\e)^{\gamma} k^{\varrho})^{-\beta}\left(\int_0^1 s^{-2\alpha} e^{-2 k^{\varrho} s}\,ds+\int_1^{+\infty} e^{-2 k^{\varrho} s}\,ds\right)\\
\leq c \sum_{k=1}^\infty k^{-(\vartheta\varrho+\varrho-2\alpha \varrho)}(1+s(\e)^{\gamma} k^{\varrho})^{-\beta}\\
\leq
c\min\left\{\,s(\epsilon)^{-\gamma \beta}\sum_{k=1}^\infty k^{-(\vartheta\varrho+\beta \varrho+\varrho-2\alpha \varrho)}, \sum_{k=1}^\infty k^{-(\vartheta\varrho+\varrho-2\alpha \varrho)}\right\}
\end{multline*}

 In particular, \eqref{as1} holds if 
 \[\beta>\frac 1{\varrho}-\vartheta-1+2\alpha,\ \ \ \ \ \gamma \beta<2\]
 or $\beta=0$ and $\vartheta > \varrho^{-1}-1+2\alpha$.

 In case $A$ is the realization of the Laplacian on a sufficiently smooth bounded domain, endowed with Dirichlet boundary conditions, we have $\varrho=2/d$, so that the first condition above becomes
 \[\beta>\left(\frac{d}2-\vartheta -1\right)+2\alpha,\ \ \ \ \ \gamma \beta< 2.\]
When $d=1$ we can take $\vartheta=0$ (space-time white noise) and $\beta=0$, with no condition on $\gamma$ since then the condition holds for any $\alpha \in (0,1/4)$.
 On the other hand, when $d\geq 2$ and $\vartheta=0$, we need to take 
\[\beta>\frac{d-2}2+2\alpha,\ \ \ \ \ \ \gamma \beta < 2.\]
Thus the regularization exponent $\beta$ is determined by two competing requirements. The first requirement
ensures that the covariance is sufficiently regularized at high frequencies for the first part of Hypothesis~\ref{H1}(4) to hold. 
Increasing $\beta$ improves the spatial regularity of the covariance, but also causes the relevant trace norm to diverge more rapidly as the regularization is removed. The condition $\gamma\beta<2$, that ensures the second part of  Hypothesis~\ref{H1}(4), expresses that this divergence must be sufficiently slow that it is compensated by the vanishing noise intensity $s(\epsilon)^2$. 
	}}
\end{Remark}
	
	As for the nonlinearity, we will assume the following conditions.
	
	\begin{Hypothesis}\label{H2}
	 The mapping $U:H\to \R$ is continuously differentiable and $\nabla U:H\to H$ is Lipschitz continuous, with $\nabla U(0)=0$ and 
	 {
	\[[\nabla U]_{\text{{\em \tiny{Lip}}}}:=\sup_{x \neq y}\frac{|\nabla U(x)-\nabla U(y)|_H}{|x-y|_H}{<\lambda\,\|Q\|_{\mathcal{L}(H)}^{-1}.}\]
{In what follows we will denote	$\omega:= \lambda\, - [\nabla U]_{\text{{\em \tiny{Lip}}}}\|Q\|_{\mathcal{L}(H)} >0.$}}
	
\end{Hypothesis}
{We remark that, unlike in \cite{BudhirajaZoubouloglou2024}, which treats a finite-dimensional model, we do not assume that $U$ is $C^2$. Instead, we assume only that $U\in C^{1,1}$, namely that $U$ is continuously Fr\'{e}chet differentiable and that $\nabla U$ is Lipschitz. This is important in order to cover Example~2. Indeed, in that example $\nabla U$ is the Nemytskii operator $x\mapsto f(x(\cdot))$ on $L^2(D)$. For nonatomic $D$, Fr\'{e}chet differentiability of this Nemytskii operator as a map from $L^2(D)$ to $L^2(D)$ would force $f$ to be affine, thereby excluding the nonlinear reaction terms allowed here.
}

{Hypothesis \ref{H1} and \ref{H2} will be taken to hold throughout the work and may not be noted explicitly in the statement of various results.}
It is well known  that   under Hypotheses \ref{H1} and \ref{H2} (not all conditions included there are needed for this), for every  $\e \in\,(0,1)$ and every $u_0 \in\,H$ and $T>0$, equation \eqref{eq1} admits a unique adapted mild solution $u^\epsilon \in\,L^2(\Omega;C([0,T];H) )$
 (for a proof see e.g. {\cite[Theorem 7.4]{DPZ})}. {The condition $\omega>0$ guarantees that the linear dissipation dominates the nonlinear drift, yielding global exponential stability of the deterministic dynamics.}
 In what follows, we will study the validity of a large deviation principle for the family of measures $\{\nu^\epsilon\}$ introduced in \eqref{as20}.

For every $\e \in\,(0,1)$ we define 
\[v^\e(t):=u^\e(t/\e),\ \ \ \ \ \ t \in\,[0,1],\]
where $u^\e$ is the solution of equation \eqref{eq1}.
It is immediate to check that $v^\e$ is the solution of the equation
\begin{equation}
\label{eq1-main}
dv^\epsilon(t) = \frac 1\e \big(A v^\epsilon(t) + Q \nabla U(v^\epsilon(t))\big) dt + \frac{s(\e)}{\sqrt{\e}}\, dw^\epsilon(t),\ \ \ \ \ 
v^\epsilon(0) = u_0,
	\end{equation}
where $w^\epsilon(t)$ is a cylindrical Wiener process with covariance $B_\epsilon$. {We remark that the Brownian motion here is different from the one in \eqref{eq1}, given by the time and space scaling of the latter Brownian motion, but we use the same notation to simplify presentation.}
Furthermore, $\nu^\epsilon$ in \eqref{as20} can be written using this scaled process as
\begin{equation}\label{nu-epsilon}\nu^\epsilon=\int_0^1 \delta_{v^\e(t)}\,dt.\end{equation}

\subsection{Large deviations and Laplace principle}

Let $\mathcal{X}$ be a Polish space. A function $I : \mathcal{X} \to [0, \infty]$
is called a \emph{rate function} on $\mathcal{X}$ if it has compact level sets,
i.e., for every $\kappa \in (0, \infty)$, the level set
$\Phi_\kappa:=\{x \in \mathcal{X} : I(x) \leq \kappa\}$ is a compact subset of $\mathcal{X}$.
As a convention, the infimum over the empty set is taken to be $+\infty$.
 
A family $(X^\eps)_{\eps>0}$ of $\mathcal{X}$-valued random variables,
defined on a probability space $(\Omega, \mathcal{F}, \mathbb{P})$, is said to
satisfy the \emph{Laplace principle} on $\mathcal{X}$ with rate function $I$ and
speed $\alpha(\eps)$, where $\alpha(\eps) \to \infty$ as $\eps\downarrow 0$, if for
every bounded and continuous function $F : \mathcal{X} \to \mathbb{R}$,
\[
  \lim_{\eps \to 0} -\frac{1}{\alpha(\eps)} \log
  \mathbb{E} \exp\big (-\alpha(\eps) F(X^\eps)\big)
   =  \inf_{x \in \mathcal{X}}\,\bigl( F(x) + I(x) \bigr).
\]
We say that the \emph{Laplace upper} (respectively \emph{lower}) \emph{bound}
holds if the left-hand side is bounded below (respectively above) by the
right-hand side. It is a classical result (see e.g.
\cite[Theorem~1.8]{BudhirajaDupuis2019}) that, for a rate function $I$, the family
$(X^\eps)_{\e>0}$ satisfies the Laplace principle on $\mathcal{X}$ with rate function
$I$ and speed $\alpha(\eps)$ if and only if it satisfies the \emph{large deviation
principle} (LDP) on $\mathcal{X}$ with the same rate function and speed. That is,
for every Borel set $A \subset \mathcal{X}$,
\[
  -\inf_{x \in A^\circ} I(x)
   \leq 
  \liminf_{\eps \to 0} \frac{1}{\alpha(\eps)} \log \mathbb{P}(X^\eps \in A)
   \leq 
  \limsup_{\eps \to 0} \frac{1}{\alpha(\eps)} \log \mathbb{P}(X^\eps \in A)
   \leq  -\inf_{x \in \bar A} I(x),
\]
where $A^\circ$ and $\bar A$ denote, respectively, the interior and closure of $A$
in $\mathcal{X}$.
 
In the present work, the role of $\mathcal{X}$ is played by the space
$\mathcal{P}(H)$ of Borel probability measures on the separable Hilbert space $H$,
equipped with the topology of weak convergence, which makes $\mathcal{P}(H)$ a
Polish space. This topology can be
metrized using the bounded-Lipschitz distance
\[
  d_{\mathrm{BL}}(\mu, \nu)
   \coloneqq  \sup_{f \in \text{Lip}_{b, 1}(H)}
    \left| \int_H f \, d\mu - \int_H f \, d\nu \right|,
  \qquad \mu, \nu \in \mathcal{P}(H).
\]
 We establish the Laplace
principle for the family $(\nu^\eps)_{\eps \in (0,1)}$  
of $\mathcal{P}(H)$-valued random variables introduced in \eqref{nu-epsilon}, with speed $\alpha(\eps) = (\eps\, s^2(\eps))^{-1}$.
In view of the equivalence recalled above, our main result, stated in
Theorem~\ref{main} below, may equivalently be read as the assertion that
$ (\nu^\eps)_{\e \in\,(0,1)}$ satisfies a large deviation principle on $\mathcal{P}(H)$ with the
same rate function and speed.

\subsection{Formulation of main result}

We introduce the following mapping 
$I : \mathcal{P}(H) \to [0, \infty]$ defined by
\begin{equation}\label{action}
  I(\gamma)  \coloneqq  \frac{1}{2} \int_H
    \bigl| Q^{-1/2} A x \,+\, Q^{1/2} \nabla U(x) \bigr|_H^2 \, \gamma(dx),
  \qquad \gamma \in \mathcal{P}(H),
\end{equation}
with the convention that the integrand equals $+\infty$ at any point $x \in H$
{for which $x \notin D(A)$ or $x\in D(A)$
and $A x \notin \mathrm{Range}(Q^{1/2})$}. This paper is devoted to the proof of the following result.

\begin{Theorem}\label{main}
Suppose that Hypotheses~\ref{H1} and~\ref{H2} hold. Then the map $I:\mathcal{P}(H)\to [0,+\infty]$ defined in \eqref{action} is a rate function on
$\mathcal{P}(H)$. Moreover, the family $(\nu^\eps)_{\eps \in (0,1)}$ of
$\mathcal{P}(H)$-valued random variables defined in \eqref{nu-epsilon}
satisfies the Laplace principle on $\mathcal{P}(H)$ with rate function $I$ and
speed $(\eps\, s^2(\eps))^{-1}$. Equivalently, $(\nu^\eps)_{\epsilon \in\,(0,1)}$ satisfies a large
deviation principle on $\mathcal{P}(H)$ with the same rate function and speed.
\end{Theorem}

\begin{Remark}\label{rate-function}
{\em 
 Because $A$ is unbounded, the integrand
$|Q^{-1/2} A x + Q^{1/2} \nabla U(x)|_H^2$ takes the value $+\infty$ on a large subset
of $H$. The effective domain of $I$ is therefore concentrated on measures $\gamma$
supported on the set
\begin{equation} \label{as35}
  {\mathcal{D}  \coloneqq  \bigl\{ x \in D(A)  :  A x \in \mathrm{Range}(Q^{1/2})
    \ \bigr\}}.
\end{equation}
Note that, $I(\delta_0) = 0$, which is consistent with the law of large numbers
$\nu^\eps \to \delta_0$ in probability as $\eps\downarrow 0$. More generally,
$I(\gamma) < \infty$ requires $\gamma$ to be concentrated on points $x$ at which
both $A x$ lies in the range of $Q^{1/2}$ and the resulting vector has finite
$H$-norm; the interplay between $A$ and $Q$ in Hypothesis~\ref{H1}
guarantees that this set is non-trivial.

}
\end{Remark}

\subsection{About the proof of Theorem~\ref{main}}\label{proof}

In order to prove Theorem~\ref{main}, we will first show, in
Section~\ref{upperbound}, the Laplace upper bound. Namely,  for every
$F \in C_b(\mathcal{P}(H))$,
\begin{equation}\label{laplace-upper}
  \liminf_{\eps \to 0}\, -\eps\, s^2(\eps) \, \log \mathbb{E}
    \exp\!\Big( -\frac{1}{\epsilon s^2(\epsilon)}F(\nu^\epsilon) \Big)
   \ge  \inf_{\gamma \in \mathcal{P}(H)}
    \big( F(\gamma)  +  I(\gamma) \bigr).
\end{equation}
Then, in Section~\ref{lowerbound}, we will prove the complementary
Laplace lower bound: for every $F \in C_b(\mathcal{P}(H))$,
\begin{equation}\label{laplace-lower}
  \limsup_{\eps \to 0}\, -\eps\, s^2(\eps) \, \log \mathbb{E}
    \exp \Big( -\frac{1}{\epsilon s^2(\epsilon)}F(\nu^\epsilon) \Big)
   \leq  \inf_{\gamma \in \mathcal{P}(H)}
    \bigl( F(\gamma)  +  I(\gamma) \,\bigr).
\end{equation}
Finally, in Section~\ref{compactness}, we will show that the function
$I$ defined in~\eqref{action} has compact level sets in
$\mathcal{P}(H)$, and is therefore a rate function. Together, these three
results complete the proof of Theorem~\ref{main}.

\section{A variational formula}\label{varformula}

In this section, we recall a variational representation for exponential
moments of functionals of a Hilbert space valued Wiener process. The
representation is due to Bou\'e and Dupuis~\cite{BoueDupuis1998} in the
finite-dimensional case and was extended to the infinite-dimensional
setting in \cite{BudhirajaDupuis2000}; in the form used below it can be found in
\cite{BudhirajaDupuis2019}. It
will be the starting point of the proof of our main result on the large
deviation principle for the family $(\nu^\eps )_{\eps \in (0,1)}$ on
$\mathcal{P}(H)$.

Throughout this section, let $(\Omega, \mathcal{F}, (\mathcal{F}_t )_{t \in\,[0,1]}, \mathbb{P})$
be a filtered probability space, with the filtration $(\mathcal{F}_t )_{0 \leq t \leq 1}$
satisfying the usual conditions. Recall from Section~\ref{assumptions} that for
every $\eps \in (0,1)$, the process $w^\eps(t)$, $t \in [0,1]$, is a cylindrical Wiener
process on $H$ with covariance $B_\eps$, where $B_\eps$ is a symmetric,
non-negative bounded linear operator on $H$ such that
$B_\eps^{1/2} \to Q^{1/2}$ strongly in $H$ as $\eps \downarrow 0$. 
{ Furthermore, as usual, we assume that 
for $0\leq s \leq t$, $w^\eps(t)-w^\eps(s)$ is independent of $\mathcal{F}_s$.}

We denote by $(e_k )_{k \in \mathbb{N}}$ a fixed complete orthonormal system in $H$.
With this choice, the cylindrical Wiener process $w^\eps$ admits the formal
expansion
\begin{equation}\label{weps-expansion}
  w^\eps(t)  =  \sum_{k=1}^\infty B_\eps^{1/2} e_k \, \beta_k(t),
  \qquad t \in [0,1],
\end{equation}
where $(\beta_k )_{k \in \mathbb{N}}$ is a sequence of independent, standard,
real-valued $(\mathcal{F}_t )$-Brownian motions. The series in
\eqref{weps-expansion} converges in $L^2(\Omega; K)$ for any Hilbert
space $K \supset H$ into which $H$ is embedded by a Hilbert-Schmidt operator.

\medskip

We now introduce the relevant classes of admissible controls. 
We denote by $\mathcal{A}$
the collection of all $(\mathcal{F}_t )_{t \in\,[0,1]}$-progressively measurable,
$H$-valued processes $\psi(t)$, $t \in\,[0,1]$, that satisfy
\[
  \mathbb{E} \int_0^1 | \psi(s) |_H^2 \, ds  <  \infty.
\]
For $M \in (0,\infty)$, let
\begin{equation}\label{SM-def}
  \mathcal{S}_M  \coloneqq  \Bigl\{ h \in L^2(0,1;H)  : 
    \int_0^1 | h(s) |_H^2 \, ds  \leq  M \Bigr\}.
\end{equation}
The set $\mathcal{S}_M$, equipped with the weak topology inherited from $L^2(0,1;H)$,
is a compact metrizable space (cf.~\cite[Lemma~3.11]{BudhirajaDupuis2019}). We set
\[
  \mathcal{A}_{b,M}  \coloneqq  \bigl\{ \psi \in \mathcal{A}  :  \psi(\cdot, \omega) \in \mathcal{S}_M,
    \text{ for } \mathbb{P}\text{-a.e. } \omega \in \Omega \bigr\},
  \qquad
  \mathcal{A}_b\   \coloneqq\ \  \bigcup_{M = 1}^\infty \mathcal{A}_{b,M}.
\]

\medskip

 The following variational representation holds.

\begin{Theorem}\label{varrep-thm}
Let $G \in B_b(C([0,1] : H))$ and let $\eps \in (0,1)$. Then
\begin{equation}\label{varrep}
  -\log \mathbb{E} \exp\bigl( -G(w^\eps) \bigr)
   = 
  \inf_{\psi \in \mathcal{R}}
    \mathbb{E}\!\left(
      G\!\left( w^\eps + \int_0^{\,\cdot\,} B_\eps^{1/2} \psi(s) \, ds \right)
       +  \frac{1}{2} \int_0^1 |\psi(s)|_H^2 \, ds
    \right),
\end{equation}
where $\mathcal{R}$ may be taken to be either $\mathcal{A}$ or $\mathcal{A}_b$.
\end{Theorem}

\begin{proof}
The representation is the Hilbert space valued version of the Bou\'e-Dupuis
formula for finite dimensional Brownian motions; see \cite[Theorem~8.3]{BudhirajaDupuis2019} (cf.\ also
\cite[Theorem~3.6]{BudhirajaDupuis2000} for the original Hilbert space extension).
\end{proof}
 
\medskip
 
We now describe how Theorem~\ref{varrep-thm} will be applied in the sequel. Recall
the rescaled process $v^\eps$ from \eqref{eq1-main} considered on the time interval $[0,1]$,
and the associated empirical measure
\[
  \nu^\eps  =  \int_0^1 \delta_{v^\eps(t)} \, dt  \in  \mathcal{P}(H).
\]
Under Hypotheses~\ref{H1} and~\ref{H2}, equation~\eqref{eq1-main}
admits a unique mild solution $v^\eps \in L^2(\Omega; C([0,1];H))$, and the map
$w^\eps \mapsto v^\eps$ is a measurable function from $C([0,1] ; K)$ into
$C([0,1] ; H)$ (where $K$ denotes a sufficiently large extension of $H$ in which
$w^\eps$ takes continuous paths). Consequently, there exists a measurable map
$
  \mathcal{G}^\eps : C([0,1] ; K)  \longrightarrow  \mathcal{P}(H)$
such that
\[
  \nu^\eps  =  \mathcal{G}^\eps(w^\eps), \qquad \mathbb{P}\text{-a.s.}
\]
 
Fix $F \in C_b(\mathcal{P}(H))$ and apply Theorem~\ref{varrep-thm} with
\[G = G^\eps \coloneqq \frac{1}{\eps s^2(\eps)} \, F \circ \mathcal{G}^\eps.\]
 Multiplying both
sides of~\eqref{varrep} by $\eps s^2(\eps)$ and using the identity
\[
  \frac{1}{2} \, \eps s^2(\eps) \int_0^1 | \psi(s) |_H^2 \, ds
   =  \frac{1}{2} \int_0^1 \bigl| \sqrt{\eps}\, s(\eps) \, \psi(s) \bigr|_H^2 \, ds,
\]
together with the substitution $\varphi(s) = \sqrt{\eps}\, s(\eps) \, \psi(s)$, one obtains the
following representation, which will be the basis for the proof of both the Laplace
upper and lower bounds in Sections~\ref{upperbound} and~\ref{lowerbound} below
\begin{equation}\label{varrep-scaled}
 \begin{array}{l}\ds{ -\eps s^2(\eps) \log \mathbb{E} \exp \left ( - \frac{1}{\epsilon s^2(\epsilon)}F(\nu^\epsilon)\right)
  }\\[14pt]
  \ds{\quad \quad \quad =
  \inf_{\varphi \in \mathcal{R}} \mathbb{E}\left(
    \frac{1}{2} \int_0^1 | \varphi(s) |_H^2 \, ds
     +  F\Big( \mathcal{G}^\eps
     \Big(
      w^\eps + \frac{1}{\sqrt{\eps}\, s(\eps)} \int_0^{\,\cdot\,} B_\eps^{1/2} \varphi(s) \, ds
    \Big) \Big)
  \right).}\end{array}
\end{equation}
 
By a standard application of Girsanov's theorem,
the controlled process on the right side appearing inside $F$ in~\eqref{varrep-scaled} can
be identified, almost surely, with the empirical measure
\[
  \bar\nu^\eps  =  \int_0^1 \delta_{\bar v^\eps(t)} \, dt,
\]
where $\bar v^\eps$ solves the controlled equation
\begin{equation}\label{veps-controlled}
  d \bar v^\eps(t)  =  \frac{1}{\eps}\bigl( A \bar v^\eps(t) + Q \nabla U(\bar v^\eps(t)) \bigr) dt
     +  \frac{1}{\eps} B_\eps^{1/2} \varphi(t) \, dt
     +  \frac{s(\eps)}{\sqrt{\eps}} \, dw^\eps(t),
  \qquad \bar v^\eps(0) = u_0,
\end{equation}
on $[0,1]$. The representation~\eqref{varrep-scaled} then reduces the problem of
establishing the Laplace principle for $(\nu^\eps )_{\e \in\,(0,1)}$ to the asymptotic analysis, as
$\eps\downarrow 0$, of the cost functional
\[
 \inf_{\varphi \in \mathcal{R}}   \mathbb{E}\!\left(
    \frac{1}{2} \int_0^1 | \varphi(s) |_H^2 \, ds + F(\bar\nu^\eps)
  \right),
\]
where, for each admissible $\varphi \in \mathcal{R}$, $\bar\nu^\eps$ is the empirical measure
associated with the controlled equation~\eqref{veps-controlled} driven by $\varphi$.  This program is carried
out in Sections~\ref{upperbound} and~\ref{lowerbound}.

\section{Tightness of occupation measures}
\label{a-priori}

We introduce an auxiliary Hilbert
space $V_\delta$, depending on a parameter $\delta > 0$, into which $H$
embeds compactly. For every $\delta > 0$, the operator
$(-A)^{-\delta}$ is well defined on $H$ via the spectral
calculus,
\[
  (-A)^{-\delta} x  =  \sum_{k=1}^\infty \alpha_k^{-\delta}\, \langle x, e_k\rangle_H\, e_k,
  \qquad x \in H,
\]
and is a compact, self-adjoint, strictly positive operator on $H$.

\begin{Definition}\label{Ubeta}
For every $\delta > 0$, let $V_\delta$ denote the completion of $H$
with respect to the inner product
\[
  \langle x, y\rangle_{V_\delta}  \coloneqq  \bigl\langle (-A)^{-\delta} x,\, (-A)^{-\delta} y \bigr\rangle_H,
  \qquad x, y \in H,
\]
and corresponding norm $|x|_{V_\delta} = |(-A)^{-\delta} x|_H$.
\end{Definition}

Equivalently, working in the eigenbasis $(e_k )_{k \in\,\mathbb{N}}$,
\[
  V_\delta  =  \Bigl\{ x  =  \sum_{k=1}^\infty x_k e_k  : 
    |x|_{V_\delta}^2  \coloneqq  \sum_{k=1}^\infty \alpha_k^{-2\delta} x_k^2  <  \infty \Bigr\},
\]
so that $V_\delta$ plays the role of a ``negative-order space'' attached
intrinsically to the operator $A$, with no reference to a Sobolev scale
or to boundary conditions. The collection $(V_\delta )_{\delta > 0}$ forms
a  scale of separable Hilbert spaces with $H \subset V_\delta$
for every $\delta > 0$, and $V_{\delta_1} \subset V_{\delta_2}$ whenever
$\delta_1 \leq \delta_2$.

It follows from classical arguments that
for every $\delta > 0$, $V_\delta$ is a separable Hilbert space, and the inclusion $H \hookrightarrow V_\delta$ is
continuous, with
\[
  |x|_{V_\delta}  \leq  \alpha_1^{-\delta}\, |x|_H, \qquad x \in H.
\]
Moreover, { since $\alpha_k \to \infty$}, the inclusion $H \hookrightarrow V_\delta$ is compact.
Finally, for every $h \in D((-A)^\delta)$, { the map
\[
  z \mapsto \langle z,h\rangle_H, \qquad z\in H,
\]
extends uniquely to a continuous linear functional on $V_\delta$,
(still denoted by $\langle z,h\rangle_H$) and, for every
$z \in V_\delta$,
\[
  \langle z,h\rangle_H
  =
  \bigl\langle (-A)^{-\delta} z,\,
  (-A)^\delta h \bigr\rangle_H .
\]
}


\begin{Remark}\label{role-of-Ubeta}
{\em \emph{1.} The space $V_\delta$ plays a role analogous to the negative-order
Sobolev space $H^{-\delta}$ that arises in concrete reaction-diffusion
SPDEs (Example~2 of the introduction). In that setting, with $A$ the
Dirichlet Laplacian on a bounded domain $D \subset \R^d$, a direct
computation identifies $V_\delta$ with the topological dual of $D((-A)^\delta)$
relative to the $L^2(D)$-pivot, which coincides on smooth functions with
the standard $H^{-2\delta}$ Sobolev space on $D$. The abstract definition
in Definition~\ref{Ubeta} avoids any reference to such Sobolev scales,
making the construction available without further structure on $H$.

\smallskip

\emph{2.} We will use $V_\delta$  as the target
space for the control coordinate of the occupation measure: linear
functionals of the form $z \mapsto \langle z, h\rangle_H$, which arise
when one tests It\^o's formula against smooth functions $\eta$ on $H$, are
continuous on $V_\delta$ as soon as $\nabla \eta \in D((-A)^\delta)$. Weak
limit points of the controlled empirical measures in
$\mathcal{P}(H \times V_\delta)$ can therefore be paired against such
functionals to derive the orthogonality identity underlying the rate
function.  {This observation is used to argue the continuity of the map
$(y, \varphi) \mapsto \phi'_R(\eta_n(y))g_n(y,\phi)$ on $H\times V_{\delta}$ that appears in the proof of  Lemma \ref{orthogonality}.}}
\end{Remark}

\subsection{A-priori bounds}

We start from the following uniform bounds for the controlled equation~\eqref{veps-controlled}. 
\begin{Lemma}\label{moment-bound}
Let Hypotheses~\ref{H1} and~\ref{H2} hold, and let
$(\varphi^\eps )_{\eps \in (0,1)} \subset \mathcal{A}_{b,M}$ be a family of
controls.
Let $\bar v^\eps$ denote the unique mild solution of the controlled
equation~\eqref{veps-controlled} corresponding to the control $\varphi^\eps$.
Then there exists some $c>0$ such that 
\begin{equation}
\label{moment-H}
\eps\, \mathbb{E} \sup_{t \in\,[0,1]}  \bigl| \bar v^\eps(t) \bigr|_H^2+\mathbb{E} \int_0^1 \bigl| \bar v^\eps(t) \bigr|_H^2 \, dt\leq c,\ \ \ \ \ \ \ \e \in\,(0,1).
\end{equation}
Moreover, there exists $\beta \in (0, \alpha)$ such that
\begin{equation}\label{moment-fractional}
  \sup_{\eps \in (0,1)} \mathbb{E} \int_0^1
    \bigl| (-A)^\beta \bar v^\eps(t) \bigr|_H^2 \, dt
   <  \infty.
\end{equation}
\end{Lemma}

\begin{proof}
We recall from Hypothesis~\ref{H2}, the dissipativity constant
\[\omega \coloneqq \lambda - [\nabla U]_{\mathrm{Lip}} \|Q\|_{\mathcal{L}(H)} > 0.\]
The mild solution of
\eqref{veps-controlled} can be written as
\begin{equation}\label{mild-decomposition}
\begin{array}{l}\ds{  \bar v^\eps(t)  = e^{tA/\eps} u_0+ \frac{1}{\eps} \int_0^t e^{(t-s)A/\eps}\, Q\, \nabla U(\bar v^\eps(s)) \, ds}\\[10pt]
\ds{\quad \quad \quad \quad \quad +\frac{1}{\eps} \int_0^t e^{(t-s)A/\eps}\, B_\eps^{1/2}\, \varphi^\eps(s) \, ds+\frac{s(\eps)}{\sqrt{\eps}} \int_0^t e^{(t-s)A/\eps} \, dw^\eps(s)=:\sum_{i=1}^3  \Phi_i^\eps(t)+Z^\eps(t).}\end{array}
\end{equation}

\smallskip

\emph{Step 1. } We first prove \eqref{moment-H}.
We have
\begin{equation}\label{est1}
  \bigl| \Phi_1^\eps(t) \bigr|_H  \leq  e^{-\lambda t/\eps} \, |u_0|_H\leq |u_0|_H,
  \qquad \int_0^1 \bigl|\Phi_1^\eps(t)\bigr|_H^2 \, dt  \leq  \frac{\eps}{2\lambda} |u_0|_H^2.
\end{equation}

For  $\Phi^\e_2(t)$ and $\Phi^\e_3(t)$, if we define $X^\e(t):=\Phi^\e_2(t)+\Phi^\e_3(t)$, it is immediate to check that
\[
  \frac{dX^\eps(t)}{dt}  =  \frac{1}{\eps} A X^\eps(t) + \frac{1}{\eps}\left( Q \nabla U(\bar v^\eps(t))+B_\e^{1/2} \varphi^\e(t)\right), \qquad X^\eps(0) = 0.
\]
Since $\nabla U(0)=0$, due to Hypothesis \ref{H1}(1) and \ref{H2} we have
\[\begin{array}{l}
\ds{  \frac{1}{2} \frac{d}{dt} |X^\eps(t)|_H^2
   \leq  -\frac{\lambda}{\eps} |X^\eps(t)|_H^2
    + \frac 1\epsilon\,\left([\nabla U]_{\mathrm{Lip}} \|Q\|_{\mathcal{L}(H)}
    \, |\bar v^\eps(t)|_H+\|B_\e^{1/2} \|_{\mathcal{L}(H)}\,|\varphi^\e(t)|_H\right)|X^\e(t)|_H}\\[12pt]
    \ds{ \quad \quad \quad  
   \leq -\frac{\omega}\epsilon \,|X^\eps(t)|_H^2+ \frac{c}{\epsilon}\,\left( |\Phi^\e_1(t)|_H+|Z^\e(t)|_H+\|B_\e^{1/2} \|_{\mathcal{L}(H)}\,|\varphi^\e(t)|_H\right)|X^\e(t)|_H}\\[12pt]
  \ds{ \quad \quad \quad \quad \quad \quad  \leq -\frac{\omega}{2\,\e}\,|X^\eps(t)|_H^2 +\frac{c}{\epsilon}\,\left(|\Phi^\e_1(t)|^2_H+|Z^\e(t)|^2_H+\|B_\e^{1/2} \|^2_{\mathcal{L}(H)}\,|\varphi^\e(t)|^2_H\right).}
\end{array}
\]
Since $B_\eps^{1/2}$ converges strongly to $Q^{1/2}$, 
from the uniform
boundedness principle we get  
\begin{equation}\label{CB}\sup_\eps \|B_\eps^{1/2}\|_{\mathcal{L}(H)}=:c_B < \infty.\end{equation}
Therefore,  we get
\[
  \frac{d}{dt} \e\,|X^\eps(t)|_H^2+\omega\,|X^\eps(t)|_H^2
   \leq  c\, \left(|\Phi^\e_1(t)|^2_H+|Z^\e(t)|^2_H+|\varphi^\e(t)|_H^2\right),
\]
and integrating with respect to $t \in\,[0,1]$ we conclude
\begin{equation}
\label{est2}
\sup_{t \in\,[0,1]}\epsilon\,	|\Phi_2^\eps(t)+\Phi_3^\eps(t)|_H^2+\int_0^1 |\Phi_2^\eps(t)+\Phi_3^\eps(t)|_H^2\,dt\leq c\,\int_0^1 \left(|\Phi^\e_1(t)|^2_H+|Z^\e(t)|^2_H\right)\,dt+c M.
\end{equation}

Next, we estimate the stochastic convolution $Z_\e(t)$. Due to the analyticity of $A$, for every $\beta\geq 0$ and $s>0$ we have
\[\text{tr}\big({ (-A)^{2\beta}} e^{s A}B_\epsilon e^{sA}\big)\leq c\, (s\wedge 1)^{-2\beta}\,\text{tr}\big( e^{\,\frac s2 A}B_\epsilon e^{\,\frac s2 A}\big)\]
Thus, if we take $\beta\in\,[0,\alpha)$
\[\begin{array}{l}
 \ds{ \mathbb{E} \bigl|(-A)^{\beta}Z^\eps(t)\bigr|_H^2
   =  \frac{s^2(\eps)}{\eps} \int_0^t
    \mathrm{tr} \big((-A)^{2\beta} e^{sA/\eps}\, B_\eps\, e^{sA/\eps} \big) ds= s^2(\eps) \int_0^{t/\eps} 
    \mathrm{tr}\big((-A)^{2\beta} e^{rA}\, B_\eps\, e^{rA} \big) dr}\\[12pt]
    \ds{\quad \quad \quad \quad \quad \quad \quad \quad \quad \quad \quad \quad \leq c\,s^2(\eps)\int_0^\infty (r\wedge 1)^{-2\alpha}\,\mathrm{tr}\big(e^{rA}\, B_\eps\, e^{rA} \big) dr,}\end{array}
\]
so that, thanks to \eqref{as1}, for every $\beta \in\,[0,\alpha)$ 
\begin{equation}\label{as4}
\sup_{\epsilon\in\,(0,1)}\sup_{t \in\,[0,1]}\mathbb{E} \bigl|(-A)^{\beta}Z^\eps(t)\bigr|_H^2<\infty.	
\end{equation}
Moreover, by a stochastic factorization argument ({see \cite[Section 5.3]{DPZ}}), we have
\[Z^\e(t)=\frac{s(\epsilon)}{\sqrt{\epsilon}}\,\frac{\sin (\pi\alpha)}{\pi}\int_0^t (t-s)^{\alpha-1} e^{(t-s)A/\epsilon} Y_\e(s)\,ds,\]
where
\[Y_\e(s)=\int_0^s (s-r)^{-\alpha} e^{(s-r)A/\epsilon}\,dw^\e(r).\]
Hence, if we assume $p>1/\alpha$, thanks to \eqref{as1} this implies
\[\begin{array}{l}\ds{\mathbb{E}\,\sup_{t \in\,[0,1]}|Z^\e(t)|_H^{p}\leq c_p\, \frac{s^{p}(\epsilon)}{\epsilon^{p/2}} \left(\int_0^1 s^{-(1-\alpha)p/(p-1)}e^{-\frac{\lambda}\epsilon p/(p-1)s}\,ds\right)^{p-1}\int_0^1 \mathbb{E}\,|Y_\e(s)|_H^p\,ds}\\[12pt]
\ds{\quad \quad \quad \leq c_p\,\frac{s^p(\epsilon)}{\epsilon^{\,p/2+1-\alpha p}}
\int_0^1 \mathbb{E}\,|Y_\e(s)|_H^p\,ds}\\[12pt]
\ds{\quad \quad \quad \quad \quad \quad \leq c_p\,\frac{s^p(\epsilon)}{\epsilon^{\,p/2+1-\alpha p}}
\int_0^1 \left(\int_0^s r^{-2\a
}\text{tr}\big( e^{rA/\epsilon}B_\e e^{rA/\epsilon}\big)\,dr\right)^{p/2}\,ds\leq c_p \,{\frac{1}{\epsilon}}}\end{array}\]
{here the second line in the display follows from the inequality $\int_0^1 r^{-a} e^{-cr/\eps} dr \leq \tilde c \eps^{1-a}$
 for $a<1$ and $c>0$, the first inequality on the third line is from It\^{o}'s isometry, and the second inequality on the third line
 is a consequence of the observation  that, from Hypothesis \ref{H1}(4),
 \begin{align*}
	 \int_0^s r^{-2\a
	 }\text{tr}\big( e^{rA/\epsilon}B_\e e^{rA/\epsilon}\big)\,dr \leq \eps^{1-2\alpha} \int_0^{\infty} (u\wedge 1)^{-2\alpha}
	 \text{tr}\big( e^{uA}B_\e e^{uA}\big) du \leq c \eps^{1-2\alpha} s^{-2}(\eps).
	\end{align*}

We conclude
\begin{equation}
\label{as3}
\epsilon\,	\mathbb{E}\,\sup_{t \in\,[0,1]}|Z^\e(t)|_H^2 \leq c_p.
\end{equation}
}
Finally, if we put together \eqref{est1}, \eqref{est2}, \eqref{as4} with $\beta=0$, and \eqref{as3}, we obtain \eqref{moment-H}.

\smallskip

\emph{Step 2.} We prove \eqref{moment-fractional}.
In \eqref{as4}, we have already proved the estimate for the stochastic convolution $Z^\e(t)$. Thus, it remains to prove the estimate for the deterministic part.
For $\Phi_1^\eps$ and $\Phi_3^\eps$, the smoothing property of the analytic
semigroup yields, for any $\beta>0$ and $t \in (0,1)$,
\[
  \bigl| (-A)^\beta \Phi_1^\eps(t) \bigr|^2_H
     \leq  c\, \bigl( t/\eps \bigr)^{-2\beta} e^{-2\lambda t/\epsilon} \, |u_0|^2_H.
\]
Taking $\beta \in (0,1/2)$ and integrating over $t \in [0,1]$, gives 
\begin{equation}
\label{as6}
\int_0^1\bigl| (-A)^\beta \Phi_1^\eps(t) \bigr|^2_H	\,dt\leq c\,\epsilon \,|u_0|_H^2.
\end{equation}

Next, if we define $X^\e(t)$ as in Step 1, we have
\[(-A)^\beta X^\e(t)=\frac 1\e\int_0^t (-A)^\beta e^{(t-s)A/\epsilon}\left(Q \nabla U(\bar{v}^\e(s))+B^{1/2}_\epsilon \varphi^\e(s)\right)\,ds.\]
This, using Hypothesis \ref{H2} and $\nabla U(0)=0$, gives 
\[\begin{array}{l}
\ds{|(-A)^\beta X^\e(t)|_H\leq \frac c{\epsilon}\int_0^t ((t-s)/\epsilon)^{-\beta} e^{-\lambda(t-s)/\e}\left(|\bar{v}^\e(s)|_H+|B^{1/2}_\epsilon \varphi^\e(s)|_H\right)\,ds,}	
\end{array}\]
so that Young's inequality, together with \eqref{moment-H}, imply
\begin{equation}\label{as7}\begin{array}{l}\ds{\mathbb{E}\int_0^1|(-A)^\beta X^\e(t)|^2_H\,dt\leq \frac c{\epsilon^2}\left(\int_0^1 (s/\epsilon)^{-\beta} e^{-\lambda s/\e}\,ds\right)^2\int_0^1\mathbb{E}\,\left(|\bar{v}^\e(s)|^2_H+|B^{1/2}_\epsilon \varphi^\e(s)|^2_H\right)\,ds\leq c}	
\end{array}
\end{equation}
Therefore, if we combine \eqref{as7} together with \eqref{as4} and \eqref{as6}, we obtain \eqref{moment-fractional}.
\end{proof}

\subsection{Tightness of the occupation measures}
Recall that, we work under Hypotheses~\ref{H1} and~\ref{H2}.  Let $\bar v^\eps$ be the
unique mild solution of the controlled equation~\eqref{veps-controlled},
driven by a control $\varphi^\eps \in \mathcal{A}_{b,M}$. Recall from Lemma~\ref{moment-bound}
the uniform  bound
\begin{equation}\label{apriori-frac}
 \sup_{\eps \in (0,1)} \E \int_0^1 |\bar v^\eps(t)|_H^2 \, dt+ \sup_{\eps \in (0,1)} \E \int_0^1
    \bigl|(-A)^\beta \bar v^\eps(t)\bigr|_H^2 \, dt  <  \infty,
\end{equation}
for some $\beta \in (0,\alpha)$.

Throughout this subsection, $\delta > 0$ is fixed. The product space
$H \times V_\delta$, equipped with the strong topology of each factor,
is a separable metric space.

\begin{Definition}\label{occ-meas}
For every $\eps \in (0,1)$, define the random element
$\bar \Lambda^\eps \in \mathcal{P}(H \times V_\delta)$ by
\begin{equation}\label{Lambda-def}
  \bar \Lambda^\eps(C)  \coloneqq 
  \int_0^1 \mathbf{1}_C\bigl( \bar v^\eps(s),\, \varphi^\eps(s) \bigr) \, ds,
  \qquad C \in \mathcal{B}(H \times V_\delta).
\end{equation}
\end{Definition}

Equivalently, $\bar \Lambda^\eps$ is the joint occupation measure on $[0,1]$
of the pair $(\bar v^\eps, \varphi^\eps)$, where the second
coordinate is viewed as a $V_\delta$-valued process via the inclusion
$H \hookrightarrow V_\delta$.
Disintegrating along the first coordinate
yields
\[
  \bar \Lambda^\eps(dy\  dz)  =  \hat \Lambda^\eps(dy)\, \lambda^\eps(y, dz),
\]
where $\hat \Lambda^\eps \in \mathcal{P}(H)$ is the marginal law of the state,
\[
  \hat \Lambda^\eps(A)  =  \int_0^1 \mathbf{1}_A(\bar v^\eps(s))\,ds,
  \qquad A \in \mathcal{B}(H),
\]
and $\lambda^\eps(y,\cdot)$ is a regular conditional probability on $V_\delta$
given the state. The marginal $\hat \Lambda^\eps$ is nothing but the
controlled empirical measure $\bar \nu^\eps$ already introduced in
Section~\ref{varformula}.

\begin{Proposition}\label{Q-tightness}
Let Hypotheses~\ref{H1} and~\ref{H2} hold, and let
$(\varphi^\eps )_{\eps \in (0,1)} $ be a
family of controls contained in $ \mathcal{A}_{b,M}$, for some $M>0$. Let
$(\bar \Lambda^\eps)_{\eps \in (0,1)}$ be the corresponding family of
$\mathcal{P}(H \times V_\delta)$-valued random variables defined by
\eqref{Lambda-def} in Definition \ref{occ-meas}. 

Then, the family $(\bar \Lambda^\eps )_{\eps \in (0,1)}$ is tight in
$\mathcal{P}(H \times V_\delta)$, and  if we fix  $\beta \in (0,\alpha)$ as in
\eqref{apriori-frac}, we have
{\begin{equation}\label{Q-fractional-moment}
  \sup_{\eps \in (0,1)} \E \int_{H \times V_\delta}
     \bigl(  |y|_H^2 + |(-A)^{\beta} y|_H^2 + |z|_H^2 \bigr) \, \bar \Lambda^\eps(dy\, dz)
    <  \infty,
\end{equation}}
with the convention that the integrand is $+\infty$ on
{$\{(y,z) \in H\times V_{\delta}: y \notin D((-A)^{\beta}) \mbox{ or } z \notin H\}$.}
\end{Proposition}

\begin{proof}
 Let $\bar \lambda^\eps \in \mathcal{P}(H \times V_\delta)$
denote the mean measure of $\bar \Lambda^\eps$
\[
  \bar \lambda^\eps(C)  \coloneqq  \E\, \bar \Lambda^\eps(C) 
  = \int_0^1 \P\bigl(\bigl(\bar v^\eps(s), \varphi^\eps(s)\bigr) \in C\bigr) \,ds,
  \qquad C \in \mathcal{B}(H \times V_\delta).
\]
By Fubini's theorem,
{
\begin{align*}
 \E \int_{H \times V_\delta} \bigl( |y|_H^2+|(-A)^\beta y|_H^2 + |z|_H^2 \bigr) \, \bar \Lambda^\eps(dy\, dz) &=
\E \int_0^1 | \bar v^\eps(s)|_H^2 \, ds \\
&\quad + \E \int_0^1 |(-A)^\beta \bar v^\eps(s)|_H^2 \, ds
        +  \E \int_0^1 |\varphi^\eps(s)|_H^2 \, ds.
\end{align*}
The sum of first two terms on the right is finite,} uniformly in $\eps \in\,(0,1)$, by \eqref{apriori-frac}.
For the third term, since $\varphi^\eps \in \mathcal{A}_{b,M}$, we have
\begin{equation}\label{Beps-control}
  \E \int_0^1 |\varphi^\eps(s)|_H^2 \, ds
   \leq  M.
\end{equation}
This proves 
\eqref{Q-fractional-moment}.

Since $\bar \Lambda^\eps$ is a random probability measure, by
\cite[Theorem~2.11]{BudhirajaDupuis2019} in order to prove the tightness of the family $(\bar \Lambda^\eps)_{\eps \in (0,1)}$ in
$\mathcal{P}(H \times V_\delta)$ it suffices to show that the
deterministic family of mean measures $(\bar \lambda^\eps)_{\epsilon \in\,(0,1)}$ is relatively
compact in $\mathcal{P}(H \times V_\delta)$. In turn, it is enough to
prove tightness of each of the two marginals separately.

For the first marginal $\Pi_1 \bar \lambda^\eps \in \mathcal{P}(H)$, we set
\[
  K_R  \coloneqq  \bigl\{ y \in D((-A)^{\beta})  : 
     |(-A)^{\beta} y|_H  \leq  R \bigr\}, \qquad R \in (0,\infty).
\]
Since $A$ has compact resolvent, the embedding
$D((-A)^{\beta}) \hookrightarrow H$ is compact, and therefore $K_R$ is
a compact subset of $H$, for every $R$. By Markov's inequality together
with~\eqref{Q-fractional-moment},
\[
  \Pi_1 \bar \lambda^\eps\bigl(K_R^c\bigr)
   \leq  \frac{1}{R^2}\, \E \int_0^1 |(-A)^{\beta} \bar v^\eps(s)|_H^2 \, ds
   \leq  \frac{c}{R^2},
\]
for some $c>0$ independent of $\eps$. Sending $R \to \infty$
gives tightness of $(\Pi_1 \bar{\lambda}^\eps)_{\e \in\,(0,1)}$ in $\mathcal{P}(H)$.

For the second marginal $\Pi_2 \bar \lambda^\eps \in \mathcal{P}(V_\delta)$, we set
\[
  B_R  \coloneqq  \{ z \in H  :  |z|_H  \leq  R \},
  \qquad R \in (0,\infty).
\]
$B_R$ is a compact
subset of $V_\delta$, and  bound \eqref{Beps-control} gives
\[
  \Pi_2 \bar \lambda^\eps \bigl(B_R^c\bigr)
   \leq  \frac{1}{R^2}\, \E \int_0^1 |\varphi^\eps(s)|_H^2 \, ds
   \leq  \frac{M}{R^2},
\]
so that $(\Pi_2 \bar \lambda^\eps)_{\e \in\,(0,1)}$ is tight in $\mathcal{P}(V_\delta)$.
Tightness of both marginals implies tightness of the joint family
$(\bar \lambda^\eps)_{\e \in\,(0,1)}$ in $\mathcal{P}(H \times V_\delta)$, completing the proof
of our lemma.
\end{proof}

\section{Laplace principle upper bound}
\label{upperbound}

The aim of this section is to establish the Laplace upper bound for the
family of empirical measures $(\mu^\eps )_{\eps \in (0,1)}$ defined
in~\eqref{as20}, namely the inequality
\begin{equation}\label{LDP-upper-aim}
  \liminf_{\eps \to 0}\,- \eps\, s^2(\eps)\, \log \mathbb{E}\, \exp \Big (
        - \frac{1}{\eps\, s^2(\eps)}F(\mu^\eps)
      \Big)
   \geq 
  \inf_{\gamma \in \mathcal{P}(H)}
    \bigl(\, F(\gamma) \,+\, I(\gamma) \,\bigr),
\end{equation}
for every $F \in C_b(\mathcal{P}(H))$, where $I : \mathcal{P}(H) \to [0,\infty]$
is the rate function defined in~\eqref{action}.

We fix $\delta > 0$. By the definition of infimum, for every $\eps \in (0,1)$ we can
choose a control $\tilde\varphi^\eps \in \mathcal{A}_b$ that is $\delta$-optimal in
the variational representation~\eqref{varrep-scaled}, in the sense that
\begin{equation}\label{delta-optimal}
  -\eps\, s^2(\eps) \log \mathbb{E} \exp\!\Big( -\frac{1}{\epsilon s^2(\epsilon)}F(\nu^\epsilon) \Big)
   \ge  \mathbb{E}\!\left(
    \frac{1}{2} \int_0^1 \bigl| \tilde\varphi^\eps(s) \bigr|_H^2 \, ds
     +  F\bigl(\tilde\nu^\eps\bigr)
  \right)  -  \delta,
\end{equation}
where \[\tilde\nu^\eps = \int_0^1 \delta_{\tilde v^\eps(t)}\,dt,\]
and $\tilde{v}^\eps$ solves the controlled equation~\eqref{veps-controlled}, with $\varphi$ replaced
by $\tilde\varphi^\eps$.

 The next lemma, whose proof we omit, as it  follows from a localization argument identical in structure to the one used in 
\cite[Theorem~3.17]{BudhirajaDupuis2019},
shows that, at the price of further error
$\delta$ in~\eqref{delta-optimal}, one can replace $\tilde\varphi^\eps$ by a
control belonging to $\mathcal{A}_{b,M}$ for some $M \in (0,\infty)$ that is
\emph{independent} of $\eps \in\,(0,1)$.

\begin{Lemma}\label{localization}
There exist $M \in (0,\infty)$, depending only on $\|F\|_\infty$ and $\delta$,
and a family $(\varphi^\eps )_{\eps \in (0,1)} \subset \mathcal{A}_{b,M}$, such
that, for every $\eps \in (0,1)$,
\begin{equation}\label{basic-inequality}
  -\eps\, s^2(\eps) \log \mathbb{E} \exp\!\Big( -\frac{1}{\epsilon s^2(\epsilon)}F(\nu^\epsilon) \Big)
   \ge  \mathbb{E}\!\left(
    \frac{1}{2} \int_0^1 \bigl| \varphi^\eps(s) \bigr|_H^2 \, ds
     +  F\bigl(\bar\nu^\eps\bigr)
  \right)  -  2 \delta,
\end{equation}
where $\bar\nu^\eps = \int_0^1 \delta_{\bar v^\eps(t)}\,dt$ and $\bar v^\eps$ is
the solution of the controlled equation~\eqref{veps-controlled} corresponding
to $\varphi = \varphi^\eps$.
\end{Lemma}

For the remainder of this subsection, $M \in (0,\infty)$ and the family
$(\varphi^\eps)_{\e \in\,(0,1)} \subset \mathcal{A}_{b,M}$ will be fixed as in
Lemma~\ref{localization}. The strategy for proving the Laplace upper
bound~\eqref{laplace-upper} now consists of two steps: first, identify the
weak limit points of the controlled processes and the associated occupation
measures, as $\eps\downarrow 0$; second, use the inequality~\eqref{basic-inequality}
to bound the right-hand side of~\eqref{LDP-upper-aim} from below in terms of
the rate function $I$ evaluated at these weak limits. The orthogonality identity
which makes the rate function emerge from the variational cost is at the heart
of the second step.

\subsection{The orthogonality identity}
\label{as30}

In this subsection we establish the infinite-dimensional analogue of
the orthogonality identity~\cite[Lemma~4.4]{BudhirajaZoubouloglou2024},
which extracts from the weak limit of the occupation measures
$\bar \Lambda^\eps$ the precise form of the rate function. As in the
finite-dimensional case, the identity is obtained by applying It\^o's
formula to a family of test functions and passing to the limit as
$\eps \downarrow 0$. The novelty in our setting is that It\^o's formula
cannot be applied directly to test functions on $H$ involving
unbounded operators such as $A$ or $Q^{-1}$; we circumvent this by
testing against functions that depend on finitely many modes of the
eigenbasis of $A$ and only letting the dimension grow at the end.

Recall the eigenbasis $(e_k )_{k \in \N}$ of $-A$ with corresponding
eigenvalues $0 < \alpha_1 \leq \alpha_2 \leq \cdots$, $\alpha_k \to \infty$,
introduced in Remark \ref{QAcommute}. For every $n \in \N$, $\pi_n$ denotes the orthogonal projection onto the first $n$ modes. By Hypothesis~\ref{H1},
$\pi_n$ commutes with $A$ on $D(A)$, and with $Q$, $Q^{1/2}$, and with
$Q^{-1}$ and $Q^{-1/2}$ on $\pi_n(H)$. In particular, $A\pi_n$ extends
to a bounded operator on $H$ with $\|A\pi_n\|_{\mathcal{L}(H)} \leq \alpha_n$.

\begin{Remark}\label{two-level-as}
{\em Throughout this section and Sections~\ref{stab-lemma}
and~\ref{compactness}, we encounter statements that involve two
nested levels of negligibility, an outer level in $\omega \in \Omega$
with respect to $\mathbb{P}$ (or $\tilde{\mathbb{P}}$ on a Skorokhod
extension space), and an inner level in $y \in H$ with respect to a
random measure $\mu(\omega) \in \mathcal{P}(H)$.  A typical instance
is ``for $\mathbb{P}$-a.e.\ $\omega$, for $\mu(\omega)$-a.e.\ $y$,
the identity $f(\omega,y) = 0$ holds.''

We adopt the following convention. When we write
\[
  f(\omega, y) = 0 \qquad \mu\text{-a.s.},
\]
we mean that there exists a set $\Omega_0 \subset \Omega$ of full
$\mathbb{P}$-measure such that for every $\omega \in \Omega_0$,
$f(\omega, y) = 0$ for $\mu(\omega)$-a.e.\ $y \in H$.  In other
words, the outer and inner null sets are handled simultaneously,
and we suppress explicit reference to $\omega$ in $\mu(\omega)$
whenever the two-level meaning is clear from the context.}
\end{Remark}

\medskip

{ From Proposition \ref{Q-tightness}, $(\bar \Lambda^\eps)_{\eps \in (0,1)}$ is a tight family of
$\mathcal{P}(H \times V_\delta)$-valued random variables and so along a subsequence (which we denote again as $(\eps)$), $\bar \Lambda^\eps \Rightarrow \bar \Lambda$.
Note that this latter  is }
\emph{convergence in distribution} of $\mathcal{P}(H \times
V_\delta)$-valued random variables: the laws of $\bar\Lambda^\eps$
on the Polish space $(\mathcal{P}(H \times V_\delta),
d_{\mathrm{BL}})$ converge weakly to the law of $\bar\Lambda$.
In particular, $\bar\Lambda$ is a \emph{random} element of
$\mathcal{P}(H\times V_\delta)$.

Since $\mathcal{P}(H\times V_\delta)$ is Polish, the
Skorokhod representation theorem provides a probability space
$(\tilde\Omega, \tilde{\mathcal{F}}, \tilde{\mathbb{P}})$ and
$\mathcal{P}(H \times V_\delta)$-valued random variables
$\tilde\Lambda^\eps$, $\tilde\Lambda$ on it such that
\begin{equation}\label{skoro-Lambda}
  \tilde\Lambda^\eps \;\overset{d}{=}\; \bar\Lambda^\eps,
  \qquad
  \tilde\Lambda \;\overset{d}{=}\; \bar\Lambda,
  \qquad
  \lim_{\epsilon\to 0}\tilde\Lambda^\eps = \tilde\Lambda, \quad 
  \quad \tilde{\mathbb{P}}\text{-a.s.\ in }
  \mathcal{P}(H \times V_\delta).
\end{equation}
All arguments below are carried out on
$(\tilde\Omega, \tilde{\mathcal{F}}, \tilde{\mathbb{P}})$.
Since all conclusions are statements about the distribution of
$\bar\Lambda$, at the end they transfer back to the original space, via the identity of
$\tilde\Lambda$ with $\bar\Lambda$ in distribution.

We disintegrate $\tilde\Lambda$ as
\[
  \tilde\Lambda(dy\, d\varphi)
   =  \hat\Lambda(dy)\, \lambda(y, d\varphi),
\]
where $\hat\Lambda \in \mathcal{P}(H)$ is the first marginal and
$\lambda(y,\cdot)$ is a regular conditional probability on $V_\delta$
given the state, defined pathwise for every $\tilde\omega$.
{ The uniform moment bound in \eqref{Q-fractional-moment}, together with Fatou's lemma and using equality in distribution of $\bar \Lambda$ and $\tilde \Lambda$,} gives
\[
  \tilde{\mathbb{E}}\int_{H \times V_\delta} |\varphi|_H^2\,
    \tilde\Lambda(dy\,d\varphi) \;<\; \infty.
\]
By Fubini, for $\tilde{\mathbb{P}}$-a.e.\ $\tilde\omega$, for
$\hat\Lambda(\tilde\omega)$-a.e.\ $y \in H$,
\[\int_{V_\delta}|\varphi|_H^2\,\lambda(y,d\varphi) < \infty.\]
Using the convention of Remark~\ref{two-level-as}, we introduce the conditional
mean
\begin{equation}\label{def-u}
  u(y)  \;\coloneqq\;  \int_{V_\delta} \varphi\, \lambda(y, d\varphi)
  \;\in\; H,\ \ \ \ \ \ \hat\Lambda-\text{a.s.}
\end{equation}

For $y \in D(A)$ with $A y \in \mathrm{Range}(Q^{1/2})$, define
\[
  \Psi(y)  \coloneqq   Q^{-1/2} A y \,+\, Q^{1/2} \nabla U(y).
\]
In particular,  the rate function~\eqref{action} reads
\[
  I(\gamma)
   = \frac{1}{2}\int_H |\Psi(y)|_H^2\, \gamma(dy),
\]
with the
 convention that $|\Psi(y)|_H = +\infty$ when $y \notin D(A)$ or $A y \notin \mathrm{Range}(Q^{1/2})$.

 For every $n \in \N$, and with $y$ as above, we also set, 
\[
  \Psi_n(y) \coloneqq 
    Q^{-1/2} A y \,+\, Q^{1/2} \nabla U(\pi_n y),
\]
so that
\begin{equation*}
  \pi_n \Psi_n(y) = Q^{-1/2} A\,\pi_n y + Q^{1/2}\pi_n \nabla U(\pi_n y).
\end{equation*}
Since $\nabla U$ is Lipschitz with $\nabla U(0)=0$ and $\pi_n y \to y$ in $H$, for every
$y \in H$,  on {$\{\,\Psi \in L^2(\hat\Lambda;H)\} \cap \{y \in L^2(\hat\Lambda;H)\}$},
\begin{equation}\label{eq:dc1}
\lim_{n\to \infty}  \pi_n \Psi_n = \Psi,\ \ \ \ \ \ \ \lim_{n\to \infty}\pi_n \Psi =\Psi
  \qquad \text{in } L^2(\hat\Lambda; H).
\end{equation}
{ Indeed, from the Lipschitz property of $\nabla U$, $\Psi_n(y) \to \Psi(y)$, whenever $|\Psi(y)|_H<\infty$. 
In particular, on $\{\,\Psi \in L^2(\hat\Lambda;H)\,\}$, $\Psi_n(y) \to \Psi(y)$ for $\hat \Lambda$ a.e. $y$, and so 
$\pi_n\Psi_n(y) \to \Psi(y)$, for $\hat \Lambda$ a.e. $y$. Also,
$$|\pi_n\Psi_n(y)|_H \leq |\Psi_n(y)|_H \leq |Q^{1/2}\nabla U(\pi_ny)- Q^{1/2}\nabla U(y)|_H + |\Psi(y)|_H \leq c|y|_H+ |\Psi(y)|_H.
$$
Thus dominated convergence gives the first convergence in \eqref{eq:dc1}. The second convergence is argued similarly.}

\begin{Lemma}\label{orthogonality}
Let Hypotheses~\ref{H1} and~\ref{H2} hold. Then, $\tilde{\mathbb{P}}$-a.s. on the event $\{I(\hat \Lambda) < \infty\}$, the following
orthogonality identity holds
\begin{equation}\label{orth-quadratic}
  \int_H \bigl\langle \pi_n \Psi_n(y),\, \pi_n \Psi(y)\bigr\rangle_H \, \hat \Lambda(dy)
   =
  -\int_H \bigl\langle \pi_n \Psi_n(y),\, \pi_n u(y)\bigr\rangle_H \, \hat \Lambda(dy),
  \qquad n \in \N.
\end{equation}
\end{Lemma}

\begin{proof}
We work on the event $\{I(\hat\Lambda) < \infty\}$, where
$\Psi \in L^2(\hat\Lambda;H)$, with
\[
  \int_H |\Psi(y)|_H^2\, \hat\Lambda(dy)
   = 2\, I(\hat\Lambda) < \infty.
\]
{Also, without loss of generality, we assume that $\int_H |y|_H^2\, \hat\Lambda(dy)<\infty$ on this event, since from Proposition \ref{Q-tightness} and Fatou's lemma, this finiteness holds $\tilde{\mathbb{P}}$ a.s.}

Fix $n \in \N$ and consider the test function
\begin{equation}
  \eta_n(y) \coloneqq \frac12 \langle \pi_n y,\, Q^{-1} A\,\pi_n y\rangle_H
     + U(\pi_n y), \qquad y \in H.
\end{equation}
By Hypothesis~\ref{H1} and Remark~\ref{QAcommute}, the operator $Q^{-1} A$ is
self-adjoint and leaves $\pi_n(H)$ invariant, so $\eta_n$ is continuously
Fr\'echet differentiable with
\begin{equation}\label{ac112}
  \nabla \eta_n(y) = Q^{-1} A\,\pi_n y + \pi_n \nabla U(\pi_n y) 
   \in \pi_n(H).
\end{equation}
Since $\nabla\eta_n$ is Lipschitz on the finite-dimensional
subspace $\pi_n(H)\cong\R^n$, by Rademacher's theorem $D^2\eta_n$
exists a.e.\ on $\pi_n(H)$ { (with respect to Lebesgue measure)}, with
\[
  D^2 \eta_n(y) = \pi_n\bigl( Q^{-1} A + D^2 U(\pi_n y)\bigr) \pi_n,
  \qquad\text{a.e.\ on }\pi_n(H).
\]
In particular, $\nabla \eta_n$ has at most linear growth, while $D^2 \eta_n(y)$ has range
in $\pi_n(H)$, hence rank at most $n$, and is bounded a.e. uniformly in $y$, with
\[\|D^2\eta_n(y)\|_{\mathcal L(H)} \leq \|A\pi_n\|_{\mathcal L(H)}\,
\|Q^{-1}\pi_n\|_{\mathcal L(H)} + [\nabla U]_{\mathrm{Lip}} =: c_n, \; { \mbox{ a.e. on } \pi_n(H).}\]

For every $R>0$, we fix a smooth function $\phi_R:\R\to\R$
 satisfying
\[\phi_R(t)=t,\ \ |t|\leq R,\ \ \ \ \phi_R(t)=(R+1)\,\mathrm{sgn}(t),\ \ |t|\geq R+1,\]
with
\[0\leq \phi_R'(t)\leq 1,\ \ \ |\phi_R''(t)|\leq c,\ \ \ \ t \in\,\mathbb{R}.\]

Define the truncated test function
\[
  \eta_n^R(y) \;\coloneqq\; \phi_R\bigl(\eta_n(y)\bigr),
  \qquad y\in H.
\]
Since $\eta_n$ depends on $y$ only through $\pi_n y$, the same holds
for $\eta_n^R$.  The chain rule gives
\[
  \nabla\eta_n^R(y)
  = \phi_R'\bigl(\eta_n(y)\bigr)\,\nabla\eta_n(y)
\]
and, a.e.\ on $\pi_n(H)$,
\[
  D^2\eta_n^R(y)
  = \phi_R''\bigl(\eta_n(y)\bigr)\,
    \nabla\eta_n(y)\otimes\nabla\eta_n(y)
  + \phi_R'\bigl(\eta_n(y)\bigr)\,D^2\eta_n(y).
\]
Notice that the function $\eta_n^R$, depending on $y$ only through
$\pi_n y \in \pi_n(H)\cong\R^n$, is $C^{1,1}$ on $\R^n$
since $\nabla U$ is Lipschitz; the It\^o formula for $C^{1,1}$
functions on $\R^n$ holds by a standard mollification argument.

Notice that $\nabla\eta_n^R$ is bounded on $H$ and $D^2\eta_n^R$ is bounded a.e., both  supported in
$A_n(R):=\{|\eta_n|\leq R+1\}$. Moreover, we have
\begin{equation}
\label{A_R-bound}	
|\pi_n y|_H\leq c_n(R),\ \ \ \ \ \ y \in\,A_{n}(R).
\end{equation}
To see
this, note that by Hypothesis~\ref{H1}.1,
\[\langle \pi_n y,Q^{-1}A\pi_n y\rangle_H
= -\sum_{k=1}^n (\alpha_k/q_k)|y_k|^2
\leq -(\lambda/\|Q\|_{\mathcal{L}(H)})\,|\pi_n y|_H^2,\]
so, since $\nabla U(0)=0$ and $\nabla U$ is Lipschitz,
\[
  \eta_n(y)
  \;\leq\; U(0) + \frac{1}{2}\bigl([\nabla U]_{\mathrm{Lip}}
    - \lambda/\|Q\|_{\mathcal{L}(H)}\bigr)\,|\pi_n y|_H^2
  \;=\; U(0) - \frac{\omega}{2\|Q\|_{\mathcal{L}(H)}}\,|\pi_n y|_H^2,
\]
and since, from Hypothesis \ref{H2}, $\omega > 0$, this gives \eqref{A_R-bound}.

Applying It\^o's formula to $\eta_n^R(\bar v^\eps)$, multiplying by
$\eps$, and using
$\nabla\eta_n^R(\bar v^\eps)\in\pi_n(H)$ together with
$\pi_n A = A\pi_n$ and  $\pi_n Q = Q\pi_n$, we obtain
\begin{equation}\label{eq:itofor}
\begin{array}{l}
\ds{\int_0^1 \bigl\langle
    \phi_R'\bigl(\eta_n(\bar v^\eps(s))\bigr)\,\nabla\eta_n(\bar v^\eps(s)),\,
    A\pi_n\bar v^\eps(s) + \pi_n Q\,\nabla U(\bar v^\eps(s))
    + \pi_n B_\eps^{1/2}\varphi^\eps(s)
  \bigr\rangle_H\, ds}\\[10pt]
\ds{\quad\quad = \eps\bigl(\eta_n^R(\bar v^\eps(1))
    - \eta_n^R(u_0)\bigr)
  - \frac{s^2(\eps)}{2}\int_0^1
    \mathrm{tr}\bigl(D^2\eta_n^R(\bar v^\eps(s))\,B_\eps\bigr)\,ds
  - \sqrt{\eps}\,s(\eps)\,M_R^\eps,}
\end{array}
\end{equation}
where \[M_R^\eps \coloneqq \int_0^1 \langle\nabla\eta_n^R(\bar
v^\eps(s)),\,dw^\eps(s)\rangle_H.\]
  Since $\eta_n^R$ is bounded, the
boundary term satisfies
\[\E|\eps(\eta_n^R(\bar v^\eps(1))-\eta_n^R(u_0))|\leq
2(R+1)\eps.\]
For the trace term, since $D^2\eta_n^R$ has
range in $\pi_n(H)$ a.e., thanks to ~\eqref{CB} {(and since the law of $\pi_n \bar v^{\eps}(s)$ is absolutely continuous with respect to the Lebesgue measure on $\R^n$), for all $s \in [0,1]$,
\[s^2(\eps)|\mathrm{tr}(D^2\eta_n^R(\bar v^{\eps}(s))\,B_\eps)|
\leq c_n(R)\,n\,c_B^2 s^2(\eps)=O(s^2(\eps)), \; \mbox{a.s.}\]
}  For the martingale term, since $|\nabla\eta_n^R|_H$
is bounded,
\[
  \eps\,s^2(\eps)\,\E\langle M_R^\eps\rangle
  \leq \eps\,s^2(\eps)\,c_B^2\, c_n(R),
\]
so $\sqrt{\eps}\,s(\eps)\,M_R^\eps = o_{L^2(\Omega)}(1)$.  The  above estimates allow us to
conclude that the left-hand side { of \eqref{eq:itofor}} vanishes in
$L^1(\Omega,\mathbb{P})$, as $\eps\downarrow 0$, for each fixed
$R>0$ and $n \in\,\mathbb{N}$.

\smallskip

We now rewrite this left-hand side through the occupation
measure~\eqref{Lambda-def}, denoting
\[
  g_{n,\e}(y,\varphi) \;\coloneqq\;
  \bigl\langle \nabla\eta_n(y),\,
    A\pi_n y + \pi_n Q\,\nabla U(y) + \pi_n B_\eps^{1/2}\varphi
  \bigr\rangle_H,
\]
so that
\[
  \lim_{\epsilon\to 0}\,\int_{H\times V_\delta}
  \phi_R'\bigl(\eta_n(y)\bigr)\,g_{n,\epsilon}(y,\varphi)\,
  \bar\Lambda^\eps(dy\,d\varphi)=0,
  \quad \quad \text{in }L^1(\Omega,\mathbb{P}).
\]
Now, since  $\operatorname{supp}\phi_R'\subseteq [-(R+1), R+1]$ we have  $|\nabla\eta_n(y)|_H\leq c_n(R)$ on $\operatorname{supp}\phi_R'$, so that
 we get
\[
\begin{array}{l}
\ds{  \E\left|\int_{H\times V_\delta}
    \phi_R'\bigl(\eta_n(y)\bigr)\,
    \bigl\langle \nabla\eta_n(y),\,
      \pi_n(B_\eps^{1/2}-Q^{1/2})\varphi
    \bigr\rangle_H\,\bar\Lambda^\eps(dy \,d\varphi)\right|}\\[10pt]
\ds{\quad\quad\quad\quad
  \leq c_n(R)\,\max_{1\leq k\leq n}
|(B_\eps^{1/2}-Q^{1/2})e_k|_H\,
    \Big(\E\!\int_{H\times V_\delta}|\varphi|_H^2\,\bar\Lambda^\eps(dy \,d\varphi)\Big)^{1/2}.}
\end{array}\]
Thus, if we decompose \[\pi_n B_\eps^{1/2}\varphi = \pi_n Q^{1/2}\varphi
+ \pi_n(B_\eps^{1/2}-Q^{1/2})\varphi,\] 
and { denote by $g_n$ the expression defining $g_{n,\epsilon}$, with $Q^{1/2}$ in place of
$B_\eps^{1/2}$},
on recalling that
$B_\eps^{1/2}$ converges strongly to $Q^{1/2}$, in view of \eqref{Beps-control} we have
\begin{equation}\label{gn-trunc-vanish}
  \lim_{\eps\to 0}\,\E\Big|
  \int_{H\times V_\delta}
  \phi_R'\bigl(\eta_n(y)\bigr)\,g_n(y,\varphi)\,
  \bar\Lambda^\eps(dy\,d\varphi)\Big| = 0.
\end{equation}
Transferring~\eqref{gn-trunc-vanish} to the Skorokhod space via
$\tilde\Lambda^\eps\overset{d}{=}\bar\Lambda^\eps$ and extracting a
subsequence (not relabeled), we may assume
$\tilde{\mathbb{P}}$-a.s.\ vanishing.  

The integrand
$\phi_R'(\eta_n(y))\,g_n(y,\varphi)$ is continuous on
$H\times V_\delta$ (cf. Remark \ref{role-of-Ubeta}(2)) and grows at most linearly in
{$(|y|_H,|\varphi|_{H})$}.  Indeed, $\phi_R'(\eta_n(y))$ vanishes
unless {$ y \in A_n(R) = \{|\eta_n(y)|\leq R+1\}$}, which forces $|\pi_n y|_H\leq c_n(R)$ due to \eqref{A_R-bound}.
On $A_n(R)$,
\begin{equation}\label{eq:bdui}\begin{array}{l}
\ds{  \bigl|\phi_R'(\eta_n(y))\,g_n(y,\varphi)\bigr|
  \;\leq\; |\nabla\eta_n(y)|_H\cdot\bigl|
    A\pi_n y + \pi_n Q\,\nabla U(y) + \pi_n Q^{1/2}\varphi\bigr|_H\leq\; c_n(R)\,\bigl({ 1+ |y|_H +|\varphi|_H}\bigr),}
\end{array}
\end{equation}
where in the last term we used $\pi_n Q^{1/2} = Q^{1/2}\pi_n$ { and the Lipschitz property of $\nabla U$}.

We claim that if we denote $h(y,\varphi):=\phi_R'(\eta_n(y))\,g_n(y,\varphi)$, then
\begin{equation}\label{h-L1-convergence}
  \lim_{\eps\to 0}
  \int_{H\times V_\delta} h(y,\varphi)\,d\tilde\Lambda^\eps(dy\,d\varphi)
  = \int_{H\times V_\delta} h(y,\varphi)\,d\tilde\Lambda(dy\,d\varphi),\ \ \ \ \ \text{in}\ \ L^1(\tilde{\Omega}, \tilde{\mathbb{P}}).
\end{equation}
To see this, for each $M > 0$, define the bounded continuous truncation
$h^M \coloneqq (-M)\vee h \wedge M$. Then
\[\begin{array}{l}
\ds{  \tilde{\E}\left|
  \int h\,d\tilde\Lambda^\eps
  - \int h\,d\tilde\Lambda\right|
  \;\leq\;
  \tilde{\E}\left|
  \int h^M d\tilde\Lambda^\eps
  - \int h^M d\tilde\Lambda\right|
  + \tilde{\E}\int |h-h^M|\,d\tilde\Lambda^\eps
  + \tilde{\E}\int |h-h^M|\,d\tilde\Lambda}\\[12pt]
  \ds{\quad \quad \quad \quad \quad \quad \quad \quad \quad \quad \quad \quad =:I_1(\eps,M)+I_2(\eps,M)+I_3(M).}\end{array}
\]
 
Since $h^M$ is bounded continuous and
$\tilde\Lambda^\eps$ weakly converges to $\tilde\Lambda$, 
$\tilde{\mathbb{P}}$-a.s., we have
\[\lim_{\e\to 0}\int h^M\,d\tilde\Lambda^\eps =
\int h^M\,d\tilde\Lambda,\ \ \ \ \ \ \tilde{\mathbb{P}}-\text{a.s.}\] As
$|\int h^M\,d\tilde\Lambda^\eps|\leq M$ for every $\eps$, dominated
convergence on $(\tilde\Omega,\tilde{\mathbb{P}})$ gives
$I_1(\eps,M)\to 0$, as $\eps\to 0$, for each fixed $M$.
 
Next, since $|h-h^M|\leq |h|\,\mathbf{1}_{\{|h|>M\}}$
and $\tilde\Lambda^\eps\overset{d}{=}\bar\Lambda^\eps$,
\[
  I_2(\eps,M)
  \;\leq\; \E\int_{H\times V_\delta}
  |h|\,\mathbf{1}_{\{|h|>M\}}\,d\bar\Lambda^\eps\leq \frac 1M\,\E\int_{H\times V_\delta}
  |h|^2d\bar\Lambda^\eps.
\]
Therefore, in view of Proposition~\ref{Q-tightness} { and the bound in \eqref{eq:bdui}}, we get
\[
 \sup_{\e \in\,(0,1)} I_2(\eps,M)
  \;\leq\;  \frac{c}{M}.
\] 
Finally, by Fatou's lemma
{ \begin{align*}
I_3(M) \leq \liminf_{\eps\to 0} I_2(\eps,M) \leq \sup_{\e \in\,(0,1)} I_2(\eps,M)
  \;\leq\;  \frac{c}{M} \to 0 \mbox{ as  } M \to \infty.
\end{align*}}

 
\smallskip
 
Combining the three estimates, for any $\delta>0$, first choose $M$
large enough that $I_2+I_3<2\delta/3$ (uniformly in $\eps$ for $I_2$),
then choose $\eps$ small enough that $I_1<\delta/3$.  This
establishes~\eqref{h-L1-convergence}.  Extracting a further
subsequence (not relabeled), we obtain
$\int h\,d\tilde\Lambda^\eps\to
\int h\,d\tilde\Lambda$, $\tilde{\mathbb{P}}$-a.s.
Combined with the $\tilde{\mathbb{P}}$-a.s.\ vanishing of
$\int h\,d\tilde\Lambda^\eps$
(from~\eqref{gn-trunc-vanish}), we conclude that
for every fixed $R>0$ and $n \in\,\mathbb{N}$,
\begin{equation}\label{trunc-limit-identity}
  \int_{H\times V_\delta}
  \phi_R'\bigl(\eta_n(y)\bigr)\,g_n(y,\varphi)\,
  \tilde\Lambda(dy\,d\varphi)
  \;=\; 0,
  \qquad\tilde{\mathbb{P}}\text{-a.s.}
\end{equation}

\smallskip

{We will now send $R\to \infty$ in the above display.} Since $\phi_R'\to 1$ pointwise and
$|\phi_R'(\eta_n)\,g_n|\leq |g_n|$, it suffices to show that
$g_n\in L^1(\tilde\Lambda)$, $\tilde{\mathbb{P}}$-a.s., so that
dominated convergence applies.  By Cauchy-Schwarz and \eqref{ac112},
{ \[
  \int_{H\times V_\delta}|g_n|\,d\tilde\Lambda
  \;\leq\; c_n\!\left(\int|y|_H^2\,d\tilde\Lambda\right)^{\!1/2}
  \!\left(\int\bigl(|y|_H^2 +
  |\varphi|_{H}^2\bigr)\,
    d\tilde\Lambda\right)^{\!1/2}\!.
\]
}
Both factors are finite $\tilde{\mathbb{P}}$-a.s. because from Proposition~\ref{Q-tightness} we get
\[
  \int_{H\times V_\delta}
  \bigl(|y|_H^2+|\varphi|_{H}^2\bigr)\,d\tilde\Lambda
  \;\leq\;\liminf_{\eps\to 0}\int_{H\times V_\delta}
  \bigl(|y|_H^2+|\varphi|_{H}^2\bigr)\,d\tilde\Lambda^\eps
  \;<\;\infty,
  \qquad\tilde{\mathbb{P}}\text{-a.s.}
\]
Dominated convergence applied in \eqref{trunc-limit-identity} therefore yields
\[
  0= \lim_{R\to\infty}
  \int_{H\times V_\delta}\phi_R'\bigl(\eta_n(y)\bigr)\,g_n(y,\varphi)\,
  \tilde\Lambda(dy\,d\varphi)
  \;=\;
  \int_{H\times V_\delta}g_n(y,\varphi)\,
  \tilde\Lambda(dy\,d\varphi),
  \qquad\tilde{\mathbb{P}}\text{-a.s.}
\]


Next, disintegrating $\tilde\Lambda(dy\,d\varphi) =
\hat\Lambda(dy)\,\lambda(y,d\varphi)$, the terms
$A\pi_n y$ and $\pi_n Q\,\nabla U(y)$ depend only on $y$ and factor out
of the $\lambda(y,\cdot)$-integral, while
\[\int_{V_\delta}\pi_n Q^{1/2}\varphi\,\lambda(y,d\varphi)
= \pi_n Q^{1/2}u(y)\] by the definition~\eqref{def-u} of $u(y)$.
This yields
\begin{equation}\label{orth-substituted}
  \int_H \bigl\langle \nabla\eta_n(y),\,
     A\,\pi_n y + \pi_n Q\, \nabla U(y) + \pi_n Q^{1/2} u(y)\bigr\rangle_H\,
     \hat\Lambda(dy) = 0,
     \qquad \tilde{\mathbb{P}}\text{-a.s.}
\end{equation}

\smallskip

Finally, using the commutation relations and
$A\,\pi_n y + \pi_n Q\, \nabla U(y) = Q^{1/2}\,\pi_n \Psi(y)$, we get
\[
  \bigl\langle \nabla\eta_n(y),\, Q^{1/2}\,\pi_n \Psi(y)\bigr\rangle_H
   = \bigl\langle Q^{1/2}\nabla\eta_n(y),\, \pi_n \Psi(y)\bigr\rangle_H
   = \bigl\langle \pi_n \Psi_n(y),\, \pi_n \Psi(y)\bigr\rangle_H,
\]
and
\[
  \bigl\langle \pi_n \nabla\eta_n(y),\, \pi_n Q^{1/2} u(y)\bigr\rangle_H
   = \bigl\langle Q^{1/2}\nabla\eta_n(y),\, \pi_n u(y)\bigr\rangle_H
   = \bigl\langle \pi_n \Psi_n(y),\, \pi_n u(y)\bigr\rangle_H.
\]
Substituting these into~\eqref{orth-substituted}
yields~\eqref{orth-quadratic}.
\end{proof}

\begin{Corollary}\label{cor:cost-bound}
Under the assumptions of Lemma~\ref{orthogonality}, the random
limit point $\tilde\Lambda$ satisfies, $\tilde{\mathbb{P}}$-almost surely,
\begin{equation}\label{cost-lower-bound}
  \int_{H \times V_\delta} |\varphi|_H^2\, \tilde\Lambda(dy\, d\varphi)
   \geq  2\, I(\hat \Lambda),
\end{equation}
where $\hat \Lambda$ is the first marginal of $\tilde\Lambda$ and $I$
is the rate function defined in~\eqref{action}.
\end{Corollary}
 
\begin{proof}
If $I(\hat \Lambda) = +\infty$, there is nothing to prove. We work on
the event $\{I(\hat \Lambda) < \infty\}$ throughout.
 By Jensen's inequality applied to $\lambda(y,\cdot)$, for every $n
\in \N$, $\hat\Lambda$-a.s.,
\[
  \int_{V_\delta} |\pi_n \varphi|_H^2\, \lambda(y, d\varphi)
   \geq  |\pi_n u(y)|_H^2.
\]
Integrating against $\hat \Lambda$ and using $|\pi_n\varphi|_H \leq
|\varphi|_H$,
\begin{equation}\label{jensen}
  \int_{H \times V_\delta} |\varphi|_H^2\, \tilde\Lambda(dy\, d\varphi)
   \geq  \int_H |\pi_n u(y)|_H^2\, \hat \Lambda(dy).
\end{equation}

Now, completing the square against $\pi_n \Psi_n$ and using the orthogonality
identity~\eqref{orth-quadratic} to evaluate the cross term,
\[\begin{array}{l}
\ds{	
  0 \leq \int_H |\pi_n u(y) + \pi_n \Psi_n(y)|_H^2\, \hat\Lambda(dy)}\\[10pt]
  \ds{\quad \quad\quad 
   = \int_H |\pi_n u(y)|_H^2\, \hat\Lambda(dy)
     - 2\int_H \langle \pi_n \Psi_n(y),\, \pi_n \Psi(y)\rangle_H\, \hat\Lambda(dy)
     + \int_H |\pi_n \Psi_n(y)|_H^2\, \hat\Lambda(dy),}
\end{array}
\]
so that, combining with~\eqref{jensen},
\[\begin{array}{l}
\ds{	
  \int_{H \times V_\delta} |\varphi|_H^2\, \tilde\Lambda(dy\, d\varphi)
   \geq \int_H |\pi_n u(y)|_H^2\, \hat\Lambda(dy)}\\[10pt]
  \ds{\quad \quad \quad \quad  \geq 2\int_H \langle \pi_n \Psi_n(y),\, \pi_n \Psi(y)\rangle_H\, \hat\Lambda(dy)
        - \int_H |\pi_n \Psi_n(y)|_H^2\, \hat\Lambda(dy).}
\end{array}
\]
Letting $n \to \infty$ and {using 
\eqref{eq:dc1},}
the right-hand side converges to
\[2\int_H |\Psi|_H^2\,\hat\Lambda - \int_H |\Psi|_H^2\,\hat\Lambda
= \int_H |\Psi(y)|_H^2\, \hat\Lambda(dy) = 2\, I(\hat\Lambda),\] and then
\[
  \int_{H \times V_\delta} |\varphi|_H^2\, \tilde\Lambda(dy\, d\varphi)
   \geq 2\, I(\hat\Lambda),
\]
which establishes~\eqref{cost-lower-bound} $\tilde{\mathbb{P}}$-a.s.
\end{proof}

\subsection{Proof of the LDP upper bound}\label{LDP-upper-bound}

We now combine the moment estimates, tightness, and orthogonality
identity established above to conclude the LDP upper bound,
namely~\eqref{LDP-upper-aim}.

Fix $F \in C_b(\mathcal{P}(H))$ and $\delta > 0$. From the variational
representation  of Theorem~\ref{varrep-thm} and the localization argument
of Lemma~\ref{localization}, there exist $M \in (0, \infty)$ and
controls $\varphi^\eps \in \mathcal{A}_{b,M}$ such that
\begin{equation}\label{upper-start}
   -\eps\, s^2(\eps)\, \log \mathbb{E} \exp\Big (-\frac{1}{\eps\, s^2(\eps)}F(\mu^\eps) \Big)
    \geq
   \mathbb{E}\!\left(
       F(\bar \mu^\eps) \,+\, \frac{1}{2} \int_0^1 |\varphi^\eps(s)|_H^2\, ds
   \right)
   \,-\, 2 \delta,
\end{equation}
for every $\eps \in (0,1)$, where $\bar \mu^\eps$ is the empirical
measure of the controlled process $\bar v^\eps$ driven by $\varphi^\eps$
via~\eqref{veps-controlled}. Rewriting the cost in terms of the
occupation measure $\bar \Lambda^\eps$,
\[
  \frac{1}{2} \int_0^1 |\varphi^\eps(s)|_H^2\, ds
   =  \frac{1}{2} \int_{H \times V_\delta} |\varphi|_H^2\,
        \bar \Lambda^\eps(dy\, d\varphi),
\]
so that~\eqref{upper-start} becomes
\begin{equation}\label{upper-as-occ}
  -\eps\, s^2(\eps)\, \log \mathbb{E} \exp\Big (-\frac{1}{\eps\, s^2(\eps)}F(\mu^\eps) \Big)
    \geq
   \mathbb{E}\!\left(
       F\bigl( [\bar \Lambda^\eps]_1 \bigr)
       \,+\, \frac{1}{2} \int_{H \times V_\delta} |\varphi|_H^2\,
         \bar \Lambda^\eps(dy\, d\varphi)
   \right)
   \,-\, 2 \delta.
\end{equation}

{Recall from Proposition \ref{Q-tightness}, $(\bar \Lambda^\eps)_{\eps \in (0,1)}$ is a tight family of
$\mathcal{P}(H \times V_\delta)$-valued random variables and so along a subsequence (which we denote again as $(\eps)$), $\bar \Lambda^\eps \Rightarrow \bar \Lambda$.
By a standard subsequential argument, it suffices to establish \eqref{LDP-upper-aim} along this subsequence. Also, recall the Skorokhod space $(\tilde\Omega, \tilde{\mathcal{F}}, \tilde{\mathbb{P}})$ introduced below Remark \ref{two-level-as}, in particular,}
since $\tilde\Lambda^\eps \overset{d}{=}
\bar\Lambda^\eps$, the right-hand side of~\eqref{upper-as-occ}
equals
\begin{equation}\label{upper-skorokhod}
   \tilde{\mathbb{E}}\!\left(
       F\bigl( [\tilde\Lambda^\eps]_1 \bigr)
       \,+\, \frac{1}{2} \int_{H \times V_\delta} |\varphi|_H^2\,
         \tilde\Lambda^\eps(dy\, d\varphi)
   \right)
   \,-\, 2 \delta.
\end{equation}
All subsequent estimates are carried out on the above Skorokhod space.

For the first term, since $F \in C_b(\mathcal{P}(H))$ and
$[\tilde\Lambda^\eps]_1 \to \hat\Lambda$ $\tilde{\mathbb{P}}$-a.s.\
in $\mathcal{P}(H)$, by the $\tilde{\mathbb{P}}$-a.s.\ weak convergence
$\tilde\Lambda^\eps \to \tilde\Lambda$ and the continuity of the
projection onto the first marginal, dominated convergence gives
\[
   \lim_{\eps \to 0}
     \tilde{\mathbb{E}}\,F([\tilde\Lambda^\eps]_1)
    =  \tilde{\mathbb{E}}\, F(\hat \Lambda).
\]

For the second term, by the $\tilde{\mathbb{P}}$-a.s.\ weak convergence
$\tilde\Lambda^\eps \to \tilde\Lambda$ in $\mathcal{P}(H \times
V_\delta)$ and the lower semicontinuity of $\varphi \mapsto
|\varphi|_H^2$ on $V_\delta$ { (which implies the lower semicontinuity of the map $\theta \mapsto  \int_{H \times V_\delta} |\varphi|_H^2\, \theta(dy\, d\varphi)$)}, Fatou's lemma gives
\[
   \liminf_{\eps \to 0}\, \tilde{\mathbb{E}}
      \int_{H \times V_\delta} |\varphi|_H^2\, \tilde\Lambda^\eps(dy\, d\varphi)
    \geq
   \tilde{\mathbb{E}} \int_{H \times V_\delta} |\varphi|_H^2\,
     \tilde\Lambda(dy\, d\varphi).
\]
By Corollary~\ref{cor:cost-bound}, $\tilde{\mathbb{P}}$-almost surely,
\[
   \int_{H \times V_\delta} |\varphi|_H^2\, \tilde\Lambda(dy\, d\varphi)
    \geq  2\, I(\hat \Lambda).
\]
Combining,
\[
  \frac{1}{2} \liminf_{\eps \to 0}\,\tilde{\mathbb{E}}\,
        \int_{H \times V_\delta} |\varphi|_H^2\,
         \tilde\Lambda^\eps(dy\, d\varphi)
    \geq
   \tilde{\mathbb{E}}\bigl[ I(\hat \Lambda) \bigr].
\]

Putting the two pieces together and recalling that
 $\tilde\Lambda^\eps \overset{d}{=}
\bar\Lambda^\eps$, we obtain
\[
\begin{aligned}
  \liminf_{\eps \to 0}
    -\eps\, s^2(\eps)&\, \log \mathbb{E} \exp\Big(
     -\frac{1}{\eps\, s^2(\eps)} F(\mu^\eps)\Big)
   \geq  \tilde{\mathbb{E}}\bigl( F(\hat \Lambda) + I(\hat \Lambda) \bigr)
      \,-\, 2 \delta \\[10pt]
      &\geq  \inf_{\gamma \in \mathcal{P}(H)} \bigl( F(\gamma) + I(\gamma) \bigr)
      \,-\, 2 \delta.
\end{aligned}
\]
Since $\delta > 0$ is arbitrary, this gives
\[
   \liminf_{\eps \to 0}
    -\eps\, s^2(\eps)\, \log \mathbb{E} \exp\Big(
     -\frac{1}{\eps\, s^2(\eps)} F(\mu^\eps)\Big)
    \geq  \inf_{\gamma \in \mathcal{P}(H)} \bigl( F(\gamma) + I(\gamma) \bigr),
\]
which is the LDP upper bound~\eqref{LDP-upper-aim}.

{

\section{The stabilization SDE around a fixed point}\label{stab-lemma}

Fix $x_* \in H$. Following the strategy of Lemma~4.5
of~\cite{BudhirajaZoubouloglou2024}, we define the \emph{shifted drift}
at $x_*$ as the mapping $\mathcal{V}_{x_*} : H \to H$ given by
\begin{equation}\label{shifted-drift-def}
  \mathcal{V}_{x_*}(z)  \coloneqq   Q\bigl( \nabla U(x_* + z) - \nabla U(x_*) \bigr),
  \qquad z \in H.
\end{equation}
Due to Hypothesis~\ref{H1}.1,
the Lipschitz property of $\nabla U$ in Hypothesis~\ref{H2}, and the
strict inequality $\omega = \lambda -[\nabla U]_{\text{\tiny Lip}}  \|Q\|_{\mathcal{L}(H)}>0$, we have
\begin{equation}\label{shifted-dissipativity}
  \langle Az+\mathcal{V}_{x_*}(z),\, z \rangle_H
   \leq  \bigl( -\lambda + \|Q\|_{\mathcal{L}(H)}\, [\nabla U]_{\text{\tiny Lip}} \bigr)\,
    |z|_H^2=-\omega \,\vert z\vert_H^2,
  \qquad z \in D(A).
\end{equation}
Notice that
the constant $\omega$ is the effective dissipation rate of the shifted
dynamics and is independent of $x_*$.

For each $\eps \in (0,1)$ and each pair $(x_*, y_\eps) \in H \times H$,
let $y_{x_\star}^\eps = y_{x_\star}^\eps(\,\cdot\,; x_*, y_\eps)$ denote the
unique mild solution of the stochastic evolution equation
\begin{equation}\label{stab-sde}
  d y_{x_\star}^\eps(t)  =  \frac{1}{\eps}\,\left(Ay_{x_\star}^\eps(t)+  \mathcal{V}_{x_*}\bigl(y_{x_\star}^\eps(t)\bigr)\right)\,dt
     +  \frac{s(\eps)}{\sqrt{\eps}}\, dw^\eps(t),
  \qquad y_{x_\star}^\eps(0) = y_\eps,
\end{equation}
on $[0,1]$. Existence and uniqueness of a mild solution in
$L^2(\Omega; C([0,1];H))$ follow from the Lipschitz property of
$\mathcal{V}_{x_*}$ and the assumption $B_\eps^{1/2}: H \to H$
bounded, by~\cite[Theorem 7.4]{DPZ}.

Now, if $x_* \in D(A)$ and $(A x_* + Q \nabla U(x_*))$ is in the range of $B^{1/2}_{\eps}$, and if
$\bar v^\eps$ solves the controlled equation~\eqref{veps-controlled}
on an interval $[s_0, s_1]$, with
\[B_\eps^{1/2} \varphi^\eps(t) \equiv
-(A x_* + Q \nabla U(x_*)),\ \ \ \ \ \ t \in [s_0, s_1],\]
 then the shifted process
\[\bar v^\eps(s_0 + t) - x_*,\ \ \ \ \ t \in\,[0,s_1-s_0],\] is a mild solution
of~\eqref{stab-sde}, with initial condition
$y_\eps = \bar v^\eps(s_0) - x_*$.
This fact, together with Lemma \ref{stabilization} below will be used to construct the hold-phase control described in Section \ref{lb-control-construction}.
\smallskip

{
The goal of this section is to establish the stabilization estimate of Lemma~\ref{stabilization}, which asserts that, as $\eps\to0$, the occupation measure of the shifted process $y^\eps_{x_*}$ collapses onto $\delta_0$, uniformly over the averaging horizon $t \in [\kappa,1]$.}

\begin{Lemma}\label{stab-bounds}
There exists a constant $\kappa_1 \in (0,\infty)$, depending only on
$\lambda$, $\|Q\|_{\mathcal{L}(H)}$, and $[\nabla U]_{\text{\tiny Lip}}$,
such that for every $x_* \in H$, every collection
$ (y_\eps )_{\eps \in (0,1)} \subset H$, and every $\eps \in (0,1)$,
\begin{equation}\label{stab-moment}
  \sup_{0 \leq t \leq 1} \mathbb{E}\, \bigl| y_{x_\star}^\eps(t) \bigr|_H^2
   \leq  \kappa_1\, \bigl( 1 + |y_\eps|_H^2 \bigr),
\end{equation}
and
\begin{equation}\label{stab-integrated}
  \mathbb{E}\int_0^1 \bigl| y_{x_\star}^\eps(t) \bigr|_H^2\, dt
   \leq  \kappa_1\, \bigl( 1 + \eps\, |y_\eps|_H^2 \bigr).
\end{equation}
Moreover, for every $\beta \in (0,1/2)$ there is a constant
$\kappa_{1,\beta}\in(0,\infty)$, depending in addition on $\beta$, such
that
\begin{equation}\label{stab-frac}
  \mathbb{E}\int_0^1 \bigl| (-A)^\beta y_{x_\star}^\eps(t) \bigr|_H^2\, dt
   \leq  \kappa_{1,\beta}\, \bigl( 1 + \eps\, |y_\eps|_H^2 \bigr).
\end{equation}
\end{Lemma}

\begin{proof}
We write the mild solution of~\eqref{stab-sde} as
\begin{equation}\label{stab-mild}
  y_{x_\star}^\eps(t)
  =\Phi^\epsilon(t)
+Z^\eps(t),
\end{equation}
where
\begin{equation}\label{eq:451ab}\Phi^\e(t):= e^{tA/\eps}\, y_\eps
    +  \frac{1}{\eps}\int_0^t e^{(t-s)A/\eps}\, Q\,
        \bigl(\nabla U(x_* + y_{x_\star}^\eps(s)) - \nabla U(x_*)\bigr)\, ds\end{equation}
        and
\begin{equation}\label{eq:stoccon}
  Z^\eps(t)  \coloneqq  \frac{s(\eps)}{\sqrt{\eps}}\int_0^t e^{(t-s)A/\eps}\, dw^\eps(s).
\end{equation}
We have
\[
  \frac{d\Phi^\eps(t)}{dt}
   =  \frac{1}{\eps}\, A\,\Phi^\eps(t)
     +  \frac{1}{\eps}\, Q\bigl(\nabla U(x_* + y_{x_\star}^\eps(t)) - \nabla U(x_*)\bigr),\ \ \ \ \ \ \Phi^\e(0)=y_\e.
\]
Thus, using Hypotheses \ref{H1} and \ref{H2},
\begin{align*}
\frac{d}{dt}\, |\Phi^\eps(t)|_H^2 &= \frac{2}{\eps} \langle \Phi^\eps(t), A\Phi^\eps(t) + Q\bigl(\nabla U(x_* + y_{x_\star}^\eps(t)) - \nabla U(x_*)\bigr)\rangle_H\\
&\leq -\frac{2}{\eps}\left(\lambda - [\nabla U]_{\text{\tiny Lip}} \|Q\|_{\mathcal{L}(H)}\right) |\Phi^\eps(t)|_H^2 \\
&\quad +\frac{2}{\eps} \langle \Phi^\eps(t),  Q\bigl(\nabla U(x_* + y_{x_\star}^\eps(t)) - \nabla U(x_*+\Phi^\eps(t))\bigr)\rangle_H\\
&\leq -\frac{2\omega}{\eps}|\Phi^\eps(t)|_H^2 +\frac{2c}{\eps}|\Phi^\eps(t)|_H|Z^{\eps}(t)|_H
\leq -\frac{\omega}{\eps}|\Phi^\eps(t)|_H^2 + \frac{c^2}{\omega\eps}|Z^{\eps}(t)|_H^2,
\end{align*}
where the last inequality uses Young's inequality.
Thus
\[
  \frac{d}{dt}\, |\Phi^\eps(t)|_H^2  +  \frac{\omega}{\eps}\, |\Phi^\eps(t)|_H^2
   \leq  \frac{c}{\eps}\, |Z^\eps(t)|_H^2,
\]
for some constant $c$ depending only on $\|Q\|_{\mathcal{L}(H)}$ and
$[\nabla U]_{\text{\tiny Lip}}$. This implies
\[
  |\Phi^\eps(t)|_H^2
   \leq  e^{-\omega t/\epsilon}|y_\e|_H^2+\frac{c}{\eps}\int_0^t e^{-\omega(t-s)/\eps}
    \,|Z^\eps(s)|_H^2\, ds.
\]
Hence, due to \eqref{stab-mild} and \eqref{as4}, with $\beta=0$, we get
 both \eqref{stab-moment} and \eqref{stab-integrated}.

To upgrade this first-moment control to $(-A)^\beta$-regularity, the
argument is similar to the proof of \eqref{moment-fractional}, so we only
give a sketch.
Write $\Phi^{\eps}(t) = \tilde \Phi^{\eps}_1(t) + \tilde \Phi^{\eps}_2(t)$ where $\tilde \Phi^{\eps}_1(t)$ (resp. $\tilde \Phi^{\eps}_2(t)$) is the first (resp. second) term on the right side of \eqref{eq:451ab}.
Then as for the proof of \eqref{as6}, and using the Lipschitz property of $\nabla U$, we see that, for $\beta \in (0, 1/2)$,
$$
\int_0^1\bigl| (-A)^\beta \tilde\Phi_1^\eps(t) \bigr|^2_H	\,dt\leq c\,\epsilon \,|y_{\eps}|_H^2.$$
Similarly, as for the proof of \eqref{as7}, we see that
\begin{equation}\label{as7new}\begin{array}{l}\ds{\mathbb{E}
\int_0^1|(-A)^\beta \tilde \Phi^\e_2(t)|^2_H\,dt\leq \frac c{\epsilon^2}\left(\int_0^1 (s/\epsilon)^{-\beta}
e^{-\lambda s/\e}\,ds\right)^2
\int_0^1\mathbb{E}\,|y_{x_\star}^\eps(t)|^2_H\,ds
\leq c\bigl( 1 + \eps\, |y_\eps|_H^2 \bigr),
}
\end{array}
\end{equation}
where the last inequality follows from \eqref{stab-integrated}.
Combining the last two bounds with the uniform bound in \eqref{as4}, we have from \eqref{stab-mild} that \eqref{stab-frac} holds.
\end{proof}

{
The bounds just obtained will guarantee the tightness of the occupation measures of $y^\eps_{x_*}$ and control their limiting behavior below. Before turning to that probabilistic step, we isolate its deterministic counterpart. The next lemma shows that a measure which is infinitesimally invariant under the shifted deterministic flow can only be the Dirac mass at the origin.}

\begin{Lemma}\label{flow-invariance}
Fix $x_*\in H$, let $\mathcal V_{x_*}$ be the shifted drift defined in
\eqref{shifted-drift-def}, and let $\Phi_t(z)=z(t;z)$, $t\geq 0$, denote the
mild solution flow of
\[
  \frac{dz}{dt}(t)=A z(t)+\mathcal{V}_{x_\star}(z(t)),\qquad z(0)=z \in\,H.
\]
Suppose that $\mu\in\mathcal P(H)$ has finite first moment and satisfies
\begin{equation}\label{gen-identity}
  \int_H \langle\nabla\varphi(\pi_m y),\,
  A\pi_m y+\pi_m\mathcal{V}_{x_*}(y)\rangle_H\,\mu(dy)=0,
  \qquad m\in\N,\ \ \varphi\in C^2_b(\pi_m(H)).
\end{equation}
Then $\mu=\delta_0$.
\end{Lemma}

\begin{proof}
Set $\mu^{(m)}:=(\pi_m)_\#\mu\in\mathcal P(\pi_m(H))$. Adding and subtracting
$\pi_m\mathcal V_{x_*}(\pi_m y)$ in \eqref{gen-identity}, we obtain
\begin{equation}\label{stab-limit-identity}
\begin{array}{l}\ds{\int_H \langle \nabla\varphi(y),\, A y+\pi_m\mathcal{V}_{x_*}(y)\rangle_H\,
    \mu^{(m)}(dy) }\\[12pt]\ds{\quad \quad \quad  = \int_H \langle \nabla\varphi(\pi_m y),\, \pi_m(\mathcal{V}_{x_*}(\pi_m y)-\mathcal{V}_{x_*}(y))\rangle_H\,
    \mu(dy) =:R_m(\varphi).}\end{array}
\end{equation}
Notice that, due to the Lipschitz-continuity of $\mathcal{V}_{x_\star}$
\[|R_m(\varphi)|\leq c\,|\nabla \varphi|_\infty \int_H|\pi_m y-y|\,\mu(dy),\]
so that, since $\mu$ has finite first moment, for every $\kappa>0$
\begin{equation}\label{as10}
\lim_{m\to\infty}\sup\,\{	|R_m(\varphi)|\,:\,|\nabla \varphi|_\infty\leq \kappa\}=0.
\end{equation}

As in \eqref{shifted-dissipativity} (using Hypotheses \ref{H1}.1 and
\ref{H2}), the shifted drift satisfies the incremental dissipativity estimate
\[
        \langle A h+\mathcal V_{x_*}(z+h)-\mathcal V_{x_*}(z),h\rangle_H
        \leq -\omega |h|_H^2,
        \qquad z\in H,\ h\in D(A).
\]

Fix $m\in\mathbb N$ and set
\[
        E_m:=\pi_m(H),\qquad
        F_m(z):=Az+\pi_m\mathcal V_{x_*}(z),\qquad z\in E_m .
\]
 Since $E_m$ is finite-dimensional and invariant under $A$, the vector field
$F_m:E_m\to E_m$ is globally Lipschitz with Lipschitz constant $L_m$.
From the above dissipativity estimate, we also have the following one-sided Lipschitz estimate that is uniform in
$m$: for all $z,z'\in E_m$,
\[
\begin{aligned}
        &\langle F_m(z)-F_m(z'),z-z'\rangle_H
        =
        \langle A(z-z'),z-z'\rangle_H +
        \langle
        \pi_m(\mathcal V_{x_*}(z)-\mathcal V_{x_*}(z')),
        z-z'\rangle_H  \\[8pt]
        &\quad \quad \quad =
        \langle A(z-z'),z-z'\rangle_H
        +
        \langle
        \mathcal V_{x_*}(z)-\mathcal V_{x_*}(z'),
        z-z'\rangle_H  \leq -\omega |z-z'|_H^2 .
\end{aligned}
\]
Let $\Phi_t^m(z)=z^{(m)}(t;z)$ denote the unique solution of
\[
        \dot z(t)=F_m(z(t)),\qquad z(0)=z\in E_m,
\]
which follows from the Lipschitz property of $F_m$.

We will now consider a suitable smooth approximation of  $F_m$ in the finite-dimensional space $E_m$.  Identify
$E_m$ with $\mathbb R^m$, let $\rho$ be a standard smooth mollifier
supported in the unit ball, and for $n \in \mathbb{N}$, define
\[
        \rho_n(x):=n^m\rho(nx),\qquad
        F_{m,n}:=\rho_n*F_m .
\]
Then $F_{m,n}\in C^\infty(E_m;E_m)$ and from the Lipschitz property,
\begin{equation}\label{eq:lipbd-ac}
        \sup_{z\in E_m}|F_{m,n}(z)-F_m(z)|_H
        \leq \frac{L_m}{n}, \;\;        |F_{m,n}(z)-F_{m,n}(z')|
         \leq L_m |z-z'|_H ,
       \qquad z,z'\in E_m .
\end{equation}
We also have the following uniform in $m,n$, one-sided Lipschitz estimate: for $z,z'\in E_m$, $n \in \mathbb{N}$,
\begin{equation}\label{eq:1216ac}
\begin{aligned}
        \langle F_{m,n}(z)-F_{m,n}(z'),z-z'\rangle_H
        &=
        \int_{E_m}\rho_n(r)
        \langle F_m(z-r)-F_m(z'-r),z-z'\rangle_H\,dr  \\
        &\leq
        -\omega |z-z'|_H^2.
\end{aligned}
\end{equation}
Let $\Phi_t^{m,n}(z)$ be the  flow generated by the smooth vector field $F_{m,n}$:
\[
        \dot z(t)=F_{m,n}(z(t)),\qquad z(0)=z\in E_m .
\]
For $\psi\in C_b^2(H)$ define
\[
        u_{m,n}(s,z):=\psi(\Phi_s^{m,n}(z)).
\]
Then
$u_{m,n}(s,\cdot)\in C_b^2(E_m)$.

We will now provide a uniform, in $m,n,s$, gradient bound on $u_{m,n}(s,\cdot) $.  The
derivative of the smooth flow $\Phi_s^{m,n}$ solves the equation
\[
        \frac{d}{ds}D\Phi_s^{m,n}(z)h
        =
        DF_{m,n}(\Phi_s^{m,n}(z))D\Phi_s^{m,n}(z)h,
        \qquad
        D\Phi_0^{m,n}(z)h=h .
\]
Also, from \eqref{eq:1216ac}, we see that
\[
        \langle DF_{m,n}(x)k,k\rangle_H
        \leq -\omega |k|_H^2,
        \qquad x,k\in E_m .
\]
Therefore
\[
\begin{aligned}
        \frac{d}{ds}|D\Phi_s^{m,n}(z)h|_H^2
        &=
        2\left\langle
        DF_{m,n}(\Phi_s^{m,n}(z))D\Phi_s^{m,n}(z)h,
        D\Phi_s^{m,n}(z)h
        \right\rangle_H  \\[8pt]
        &\leq
        -2\omega |D\Phi_s^{m,n}(z)h|_H^2 .
\end{aligned}
\]
It then follows that
\[
        \|D\Phi_s^{m,n}(z)\|_{\mathcal L(E_m)}
        \leq e^{-\omega s}\leq 1,
\]
for all $m,n,z$ and $s\geq0$.  Hence
\[
        |\nabla u_{m,n}(s,z)|_H
        =
        \left|D\Phi_s^{m,n}(z)^*
        \nabla\psi(\Phi_s^{m,n}(z))\right|_H
        \leq |\nabla\psi|_\infty .
\]
Thus
\begin{equation}\label{eq:unifgradbd}
        \sup_{m,n}\sup_{s\geq0}
        |\nabla u_{m,n}(s,\cdot)|_\infty
        \leq |\nabla\psi|_\infty .
\end{equation}

Next, from the flow property of the smooth
ODE $\dot z=F_{m,n}(z)$,
we see that
\[
        D\Phi_s^{m,n}(z)F_{m,n}(z)
        =
        F_{m,n}(\Phi_s^{m,n}(z)).
\]
Therefore
\[
\begin{aligned}
        \frac d{ds}u_{m,n}(s,z)
        &=
        \left\langle
        \nabla\psi(\Phi_s^{m,n}(z)),
        F_{m,n}(\Phi_s^{m,n}(z))
        \right\rangle_H  =
        \left\langle
        \nabla\psi(\Phi_s^{m,n}(z)),
        D\Phi_s^{m,n}(z)F_{m,n}(z)
        \right\rangle_H  \\[8pt]
        &=
        \left\langle
        D\Phi_s^{m,n}(z)^*
        \nabla\psi(\Phi_s^{m,n}(z)),
        F_{m,n}(z)
        \right\rangle_H  =
        \left\langle
        \nabla u_{m,n}(s,z),
        F_{m,n}(z)
        \right\rangle_H .
\end{aligned}
\]
Integrating in $s$ gives
\[
        u_{m,n}(t,z)-u_{m,n}(0,z)
        =
        \int_0^t
        \left\langle
        \nabla u_{m,n}(s,z),
        F_{m,n}(z)
        \right\rangle_H\,ds .
\]

Integrating this identity against $\mu^{(m)}$ and adding and subtracting
$F_m$ gives
\[
\begin{aligned}
& \int_{E_m}\psi(\Phi_t^{m,n}(z))\,\mu^{(m)}(dz)
  -\int_{E_m}\psi(z)\,\mu^{(m)}(dz)  \\
&\quad =
  \int_0^t
  \int_{E_m}
  \left\langle
  \nabla u_{m,n}(s,z),F_m(z)
  \right\rangle_H
  \,\mu^{(m)}(dz)\,ds  \\
&\qquad+
  \int_0^t
  \int_{E_m}
  \left\langle
  \nabla u_{m,n}(s,z),F_{m,n}(z)-F_m(z)
  \right\rangle_H
  \,\mu^{(m)}(dz)\,ds .
\end{aligned}
\]
Since $u_{m,n}(s,\cdot)\in C_b^2(E_m)$, we have from the identity
\eqref{stab-limit-identity} that
\[
        \int_{E_m}
        \left\langle
        \nabla u_{m,n}(s,z),F_m(z)
        \right\rangle_H
        \,\mu^{(m)}(dz)
        =
        R_m(u_{m,n}(s,\cdot)).
\]
Thus
\[
\begin{aligned}
& \int_{E_m}\psi(\Phi_t^{m,n}(z))\,\mu^{(m)}(dz)
  -\int_{E_m}\psi(z)\,\mu^{(m)}(dz) =
  \int_0^t R_m(u_{m,n}(s,\cdot))\,ds
  +
  \mathcal E_{m,n}(t),
\end{aligned}
\]
where
\[
        \mathcal E_{m,n}(t)
        :=
        \int_0^t
        \int_{E_m}
        \left\langle
        \nabla u_{m,n}(s,z),F_{m,n}(z)-F_m(z)
        \right\rangle_H
        \,\mu^{(m)}(dz)\,ds .
\]
For fixed $m$, from the bounds in \eqref{eq:unifgradbd} and  \eqref{eq:lipbd-ac}, as $n\to \infty$,
\[
\begin{aligned}
        |\mathcal E_{m,n}(t)|
        &\leq
        t|\nabla\psi|_\infty
        \sup_{z\in E_m}|F_{m,n}(z)-F_m(z)|_H \leq
        t|\nabla\psi|_\infty \frac{L_m}{n}
        \to 0 .
\end{aligned}
\]
Thus, since from \eqref{eq:lipbd-ac}
\[
\begin{aligned}
        |\Phi_t^{m,n}(z)-\Phi_t^m(z)|_H
        &\leq
        \int_0^t
        |F_{m,n}(\Phi_s^{m,n}(z))-F_m(\Phi_s^m(z))|_H\,ds  \\
        &\leq
        L_m\int_0^t|\Phi_s^{m,n}(z)-\Phi_s^m(z)|_H\,ds
        +
        t\sup_{y\in E_m}|F_{m,n}(y)-F_m(y)|_H,
\end{aligned}
\]
we have from Gronwall's lemma and using \eqref{eq:lipbd-ac} again, that
\[
    \lim_{n\to\infty}    \sup_{0\leq s\leq t}
        |\Phi_s^{m,n}(z)-\Phi_s^m(z)|_H
        =0.
\]
Therefore, by bounded convergence,
\[
    \lim_{n\to\infty}   \int_{E_m}\psi(\Phi_t^{m,n}(z))\,\mu^{(m)}(dz)
        =
        \int_{E_m}\psi(\Phi_t^m(z))\,\mu^{(m)}(dz).
\]

Combining the preceding estimates and using the uniform gradient bound \eqref{eq:unifgradbd},
we obtain, for every fixed $t\geq0$,
\begin{equation}\label{as1003}
\begin{aligned}
&\limsup_{m\to \infty} \left|
  \int_{E_m}\psi(\Phi_t^m(z))\,\mu^{(m)}(dz)
  -
  \int_{E_m}\psi(z)\,\mu^{(m)}(dz)
  \right| \\[8pt]
&\qquad \qquad\qquad\leq
        \limsup_{m\to \infty} t\sup\Bigl\{
        |R_m(\varphi)|:
        |\nabla\varphi|_\infty\leq |\nabla\psi|_\infty
        \Bigr\} =0,
\end{aligned}
\end{equation}
where the last equality is from \eqref{as10}.

Note that, for every fixed $z\in H$ and every fixed $t\geq0$,
\[
     \lim_{m\to\infty}   \Phi_t^m(\pi_m z)= \Phi_t(z)
        \qquad\text{in }H .
\]
Since $\mu^{(m)}=(\pi_m)_\#\mu$,
\[
        \int_{E_m}\psi(\Phi_t^m(z))\,\mu^{(m)}(dz)
        =
        \int_H \psi(\Phi_t^m(\pi_m z))\,\mu(dz),
\]
so that, by the pointwise convergence just proved and bounded convergence,
\[
        \lim_{m\to\infty}
        \int_H \psi(\Phi_t^m(\pi_m z))\,\mu(dz)
        =
        \int_H \psi(\Phi_t(z))\,\mu(dz).
\]
Also,
\[
        \int_{E_m}\psi(z)\,\mu^{(m)}(dz)
        =
        \int_H \psi(\pi_m z)\,\mu(dz)
        \longrightarrow
        \int_H \psi(z)\,\mu(dz).
\]
Thus, combining with \eqref{as1003},
\begin{equation}\label{as13}
        \int_H \psi(z(t;z))\,\mu(dz) = \int_H \psi(\Phi_t(z))\,\mu(dz)
        =
        \int_H \psi(z)\,\mu(dz),\qquad \psi\in C^2_b(H),\ t\geq 0.
\end{equation}

Finally, the dissipativity condition~\eqref{shifted-dissipativity} gives
\begin{equation}\label{stab-projected-contraction}
  |z(t;z)|_H  \leq  e^{-\omega t}\, |z|_H,
  \qquad t \geq 0,\ \ \ \  z \in\,H.
\end{equation}
Hence, letting $t \to \infty$ in~\eqref{as13} and using
dominated convergence, we conclude that
\[
  \int_{H} \psi(z)\, \mu(dz) = \psi(0),
  \qquad \psi \in C^2_b(H),
\]
which means $\mu = \delta_0$.
\end{proof}
{
We now return to the stochastic problem. Using the a priori bounds of Lemma~\ref{stab-bounds}, the following lemma shows that the occupation measures of $y^\eps_{x_*}$ are tight and that, by an It\^o-formula argument, every subsequential limit satisfies exactly the infinitesimal invariance identity~\eqref{gen-identity} required in Lemma~\ref{flow-invariance}.}

\begin{Lemma}\label{stab-stationarity}
Fix $\kappa\in(0,1]$ and assume that $\sup_\eps \eps\,|y_\eps|_H^2<\infty$.
Let $\eps_n\downarrow 0$, let $t_n\in[\kappa,1]$, and consider the random
occupation measures
\begin{equation}\label{stab-pin}
  \oc^n  \coloneqq  \frac{1}{t_n}\int_0^{t_n} \delta_{y_{x_\star}^{\eps_n}(s)}\, ds
   \in  \mathcal{P}(H).
\end{equation}
Then $(\oc^n)_{n\in\N}$ is tight, and along any subsequence for which
$\oc^n\Rightarrow\oc$, the limit $\oc$ has finite first moment and
satisfies, $\mathbb P$-a.s.,
\begin{equation}\label{gen-identity-limit}
  \int_H \langle\nabla\varphi(\pi_m y),\,
  A\pi_m y+\pi_m\mathcal{V}_{x_*}(y)\rangle_H\,\oc(dy)=0,
  \qquad m\in\N,\ \ \varphi\in C^2_b(\pi_m(H)).
\end{equation}
\end{Lemma}

\begin{proof}
Since $t_n \geq \kappa$, the
integrated bound~\eqref{stab-integrated} of Lemma \ref{stab-bounds} gives the uniform estimate
\begin{equation}\label{eq:unifestab}
  \sup_n \mathbb{E}\int_H |y|_H^2\, \oc^n(dy)
   =  \sup_n \frac{1}{t_n}\int_0^{t_n} \mathbb{E}\,|y_{x_\star}^{\eps_n}(s)|_H^2\, ds
   \leq  \frac{1}{\kappa}\,\sup_n \kappa_1\bigl(1 + \eps_n |y_{\eps_n}|_H^2\bigr)
   <  \infty,
\end{equation}
where we have used the hypothesis that $\sup_\eps \eps|y_\eps|_H^2 < \infty$.
Similarly, the fractional bound~\eqref{stab-frac} gives, for $\beta\in(0,1/2)$,
\[\sup_n \mathbb{E}\int_H |(-A)^\beta y|_H^2\, \oc^n(dy)
< \infty,\]
and the tightness of $(\oc^n)_{n \in\,\mathbb{N}}$ as a family of $\mathcal{P}(H)$-valued
random variables follows. Fix a subsequence (not relabeled) along which
$\oc^n \Rightarrow \oc$.

\smallskip

We introduce the class of test functions
\[
  \eta(y)  =  \varphi(\pi_m y),
  \qquad m \in \N,\quad \varphi \in C^2_b(\pi_m(H)).
\]
For such $\eta$ one has
\[
  \nabla\eta(y) = \pi_m \nabla\varphi(\pi_m y) \in \pi_m(H) \subset D(A),
  \qquad
  D^2\eta(y) = \pi_m D^2\varphi(\pi_m y)\pi_m.
\]
Applying It\^{o}'s formula to $\eta( y_{x_\star}^{\eps_n})$, multiplying by
$\eps_n$, and using $\pi_m A = A\pi_m$ and $\pi_m Q = Q\pi_m$,
\begin{equation}\label{stab-ito-mode}
\begin{aligned}
  &\int_0^{t_n}\!\!\bigl\langle \nabla\varphi(\pi_m y_{x_\star}^{\eps_n}(s)),
     A\pi_m y_{x_\star}^{\eps_n}(s)
     + \pi_m \mathcal{V}_{x_*}(y_{x_\star}^{\eps_n}(s)) \bigr\rangle_H\, ds \\
  &\quad = \eps_n\bigl(\varphi(\pi_my_{x_\star}^{\eps_n}(t_n))
              - \varphi(\pi_m y_{\eps_n})\bigr)
      - \frac{1}{2} s^2(\eps_n)\!\int_0^{t_n}\!
        \mathrm{tr}\,\bigl(\pi_m B_{\eps_n}\pi_m
          D^2\varphi(\pi_my_{x_\star}^{\eps_n}(s))\bigr)\, ds \\
  &\quad\quad
     - \sqrt{\eps_n}\, s(\eps_n)\!\int_0^{t_n}\!\!
       \bigl\langle \pi_m \nabla\varphi(\pi_my_{x_\star}^{\eps_n}(s)),\,
         dw^{\eps_n}(s)\bigr\rangle_H.
\end{aligned}
\end{equation}
The right-hand side is $o(1)$ in $L^1(\Omega)$. Indeed, the boundary term is
bounded by $2\eps_n\|\varphi\|_\infty \downarrow 0$, the trace term is
bounded in absolute value by
\[\frac{1}{2} s^2(\eps_n) t_n m
c_B^2\|D^2\varphi\|_\infty \to 0,\ \ \ \ \ \text{as}\ n\to\infty,\]
 and the It\^{o} integral has
$L^2(\Omega)$-norm bounded by
\[\sqrt{\eps_n}\,s(\eps_n)\,c_B\,
\|\nabla\varphi\|_\infty\, t_n^{1/2} \downarrow 0,\ \ \ \ \ \text{as}\ n\to\infty.\]

The left-hand side of~\eqref{stab-ito-mode} equals
\[
  t_n \int_H \bigl\langle \nabla\varphi(\pi_m y),\,
    A\pi_m y + \pi_m\mathcal{V}_{x_*}(y)\bigr\rangle_H\,\oc^n(dy),\ \ \ \ \ \mathbb{P}-\text{a.s.}
\]
Thus, setting \[g_{m,\varphi}(y) \coloneqq \langle\nabla\varphi(\pi_m y),
A\pi_m y + \pi_m\mathcal{V}_{x_*}(y)\rangle_H,\ \ \ \ \ \ y \in\,H,\]
\eqref{stab-ito-mode} gives \[t_n\int_H g_{m,\varphi}(y)\,\oc^n(dy) = R_n,
\ \ \ \ \ \ \mathbb{P}-\text{a.s.}\] where $R_n \to 0$ in $L^1(\Omega,\mathbb{P})$.
Since $t_n \geq \kappa > 0$,
\begin{equation}\label{L1-vanish}
  \lim_{n\to\infty}\int_H g_{m,\varphi}(y)\,\oc^n(dy)= 0,
  \qquad \text{in } L^1(\Omega,\mathbb{P}).
\end{equation}
The passage from this $L^1$ zero-limit and the convergence in
distribution $\oc^n \Rightarrow \oc$ to the $\mathbb{P}$-almost-sure
identity
\begin{equation}\label{stab-identity-as}
  \int_H g_{m,\varphi}(y)\,\oc(dy) = 0
  \qquad \mathbb{P}\text{-a.s.}
\end{equation}
follows by the same argument as in Remark~\ref{two-level-as}
and the Skorokhod representation argument in
Subsection~\ref{as30}. First we apply the Skorokhod representation
theorem to obtain $\tilde\oc^n$, $\tilde\oc$ on a new space with the
same laws and $\tilde\oc^n \to \tilde\oc$, $\tilde{\mathbb{P}}$-a.s. Next we
transfer the $L^1$ vanishing~\eqref{L1-vanish} to the new space and
extract a $\tilde{\mathbb{P}}$-a.s.\ convergent subsequence. Then we combine
with the $\tilde{\mathbb{P}}$-a.s.\ weak convergence $\tilde\oc^n \to
\tilde\oc$ and the linear-growth bound on $g_{m,\varphi}$, together with \eqref{eq:unifestab}, to conclude
\begin{equation}\label{as31}\int g_{m,\varphi}(y)\,d\tilde\oc(y) = 0,\ \ \ \ \ \ \tilde{\mathbb{P}}-\text{a.s.}\end{equation} This gives the $\tilde{\mathbb{P}}-\text{a.s.}$ convergence for a single test function. Finally, we need to  handle
all $(m,\varphi)$ simultaneously
and transfer back to the original probability space, by using the fact that  $\tilde\oc= \oc$ in distribution.

  Since the family of all  $(m,\varphi)$ such that \eqref{as31} holds is
uncountable, we cannot take an uncountable union of null sets.  For each $m \in \mathbb{N}$, the space $C^2_b(\pi_m(H))$ is separable
in the $C^1$-norm on bounded sets (since $\pi_m(H)$ is
finite-dimensional), so we may choose a countable dense subset
$\mathcal{C}_m \subset C^2_b(\pi_m(H))$.  Set
$\mathcal{C} \coloneqq \bigcup_{m \geq 1} (\{m\}\times \mathcal{C}_m)$, which is
countable.  Thanks to \eqref{as31}, for each $(m,\varphi) \in \mathcal{C}$ there
is a null set $N_{m,\varphi}$.  The set
\[
  \tilde\Omega_0  \coloneqq
  \tilde\Omega \setminus \bigcup_{(m,\varphi)\,\in\,\mathcal{C}} N_{m,\varphi}
\]
has full $\tilde{\mathbb{P}}$-measure,
and on $\tilde\Omega_0$ we have
\[
  \int_H g_{m,\varphi}(y)\,\tilde\oc(\omega)(dy) = 0
  \qquad
  (m,\varphi) \in \mathcal{C},\
   \ \ \  \omega \in \tilde\Omega_0.
\]

Next, we fix $\omega \in \tilde\Omega_0$, $m \in \mathbb{N}$, and an arbitrary
$\varphi \in C^2_b(\pi_m(H))$, and we choose a sequence $\varphi_k \in
\mathcal{C}_m$, with $\|\nabla\varphi_k - \nabla\varphi\|_\infty :=  \sup_{y\in H} |\nabla\varphi_k(y) - \nabla\varphi(y)|_H\to
0$.  Then for every $y \in H$,
\[
  |g_{m,\varphi_k}(y) - g_{m,\varphi}(y)|
   \leq
  \|\nabla\varphi_k - \nabla\varphi\|_\infty\,
  c\bigl(1 + |y|_H\bigr).
\]
From \eqref{eq:unifestab} and Fatou's lemma $\tilde\oc(\omega)$ has finite first moment $\tilde{\mathbb{P}}$ a.s. and so we can assume without loss of generality
that on the set $\tilde\Omega_0$ this integrability property holds. The dominated
convergence theorem now gives
\[
 \lim_{k\to\infty} \int_H g_{m,\varphi_k}(y)\,\tilde\oc(\omega)(dy)
   =
  \int_H g_{m,\varphi}(y)\,\tilde\oc(\omega)(dy).
\]
The left side is $0$ for every $k$, so the right side is $0$ as well.
Thus, on the single full-measure set $\tilde\Omega_0$,
\[
  \int_H g_{m,\varphi}(y)\,\tilde\oc(\omega)(dy) = 0
  \qquad
  m \in \mathbb{N},\ \ \ \
  \varphi \in C^2_b(\pi_m(H)).
\]

Finally, we define the Borel set
\[
  \mathcal{B}  \coloneqq
  \Bigl\{\mu \in \mathcal{P}(H) :
    \int_H g_{m,\varphi}(y)\,d\mu(y) = 0,\ \ m \in \mathbb{N},\ \  \varphi \in C^2_b(\pi_m(H))\Bigr\}.
\]
To see that $\mathcal{B}$ is Borel, note that for each fixed $(m,\varphi)$ the
map $\mu \mapsto \int_H g_{m,\varphi}\,d\mu$ is Borel measurable on
$(\mathcal{P}(H), d_{\mathrm{BL}})$, so each level set
$\{\mu : \int_H g_{m,\varphi}\,d\mu = 0\}$ is Borel.  The full
condition defining $\mathcal{B}$ ranges over an uncountable family of
$\varphi$'s, but as we have just seen above it is equivalent to requiring
$\int_H g_{m,\varphi}\,d\mu = 0$ only for $(m,\varphi) \in
\mathcal{C}$, so $\mathcal{B}$ is in fact a \emph{countable} intersection of
Borel sets, hence Borel.

In this way, we have established that $\tilde\oc(\omega) \in \mathcal{B}$, for every $\omega
\in \tilde\Omega_0$, where $\tilde{\mathbb{P}}(\tilde\Omega_0) = 1$.
Hence $\tilde{\mathbb{P}}(\tilde\oc \in \mathcal{B}) = 1$, and this implies $\mathbb{P}(\oc \in
\mathcal{B}) = 1$. This is precisely~\eqref{gen-identity-limit}, and the finite first
moment of $\oc$ follows from~\eqref{eq:unifestab} and Fatou's lemma.
\end{proof}

{We are now in a position to prove the main result of the section. Combining the two previous lemmas yields the stabilization estimate. Lemma~\ref{stab-stationarity} identifies every subsequential limit of the occupation measures as a measure satisfying~\eqref{gen-identity-limit}, while Lemma~\ref{flow-invariance} forces any such measure to equal $\delta_0$. Hence the occupation measures themselves converge to $\delta_0$.}

\begin{Lemma}\label{stabilization}
Assume that $\sup_\eps \eps\, |y_\eps|_H^2 < \infty$. Then for every
$\kappa \in (0,1]$
\begin{equation}\label{stab-occupation}
 \lim_{\e\to 0}\  \sup_{t \in [\kappa, 1]}\, \mathbb{E}\, d_{\mathrm{BL}}\!\left(
    \frac{1}{t} \int_0^t \delta_{y_{x_\star}^\eps(s)}\, ds,  \delta_0
  \right)=0.
\end{equation}
\end{Lemma}

\begin{proof}
We fix $\kappa \in (0,1]$ and argue by contradiction, assuming that
\eqref{stab-occupation} fails. Then there
exist $\gamma > 0$, a sequence $\eps_n \downarrow  0$, and times
$t_n \in [\kappa, 1]$ such that, with the random occupation measures
$\oc^n$ defined in \eqref{stab-pin},
\[\inf_{n \in\,\mathbb{N}}\mathbb{E}\, d_{\mathrm{BL}}(\oc^n, \delta_0) \geq \gamma.\]
By Lemma~\ref{stab-stationarity}, the family $(\oc^n)$ is tight; passing to
a subsequence (not relabeled) we may assume $\oc^n \Rightarrow \oc$ for
some $\mathcal{P}(H)$-valued random variable $\oc$, which $\mathbb{P}$-a.s.\ has
finite first moment and satisfies~\eqref{gen-identity-limit}. Applying
Lemma~\ref{flow-invariance} pathwise, with $\mu = \oc(\omega)$ for
$\mathbb{P}$-a.e.\ $\omega$, we conclude that $\oc = \delta_0$,
$\mathbb{P}$-a.s.

This contradicts $\inf_n\mathbb{E}\, d_{\mathrm{BL}}(\oc^n, \delta_0) \geq \gamma$,
since $\oc^n \Rightarrow \oc = \delta_0$ and $d_{\mathrm{BL}}(\cdot,\delta_0)$
is bounded and continuous on $\mathcal{P}(H)$, so
$\mathbb{E}\,d_{\mathrm{BL}}(\oc^n,\delta_0) \to 0$. This
proves~\eqref{stab-occupation} and completes the proof.
\end{proof}}

\section{Laplace principle lower bound}
\label{lowerbound}

In this section we prove the  Laplace lower bound 
\begin{equation}\label{laplace-lower-restate}
  \limsup_{\eps \to 0}\, -\eps\, s^2(\eps) \, \log \mathbb{E}
    \exp\!\Big( -\frac{1}{\epsilon s^2(\epsilon)}F(\nu^\epsilon) \Big)
   \leq  \inf_{\gamma \in \mathcal{P}(H)}
    \bigl( F(\gamma)  +  I(\gamma) \bigr),
\end{equation}
for
every $F \in C_b(\mathcal{P}(H))$.

\subsection{Reduction to a finitely supported $\delta$-optimum}
\label{lb-reduction}

We fix $F \in C_b(\mathcal{P}(H))$ and $\delta > 0$. By the definition of
infimum, we choose $\gamma^* \in \mathcal{P}(H)$ with
\begin{equation}\label{gamma-star-deltaopt}
  F(\gamma^*) + I(\gamma^*)
   \leq  \inf_{\gamma \in \mathcal{P}(H)} \bigl( F(\gamma) + I(\gamma) \bigr) + \delta.
\end{equation}
Since $I(\gamma^*) < \infty$, the measure $\gamma^*$ is concentrated on
\[
 { \mathcal{D}  \coloneqq  \bigl\{ x \in D(A)\ :\ A x \in \mathrm{Range}(Q^{1/2}) \bigr\}.}
\]

Following the discrete approximation strategy used in
~\cite{BudhirajaZoubouloglou2024}, an application of the
Glivenko-Cantelli theorem together with the strong law of large numbers
applied to an i.i.d.\ sequence $\xi_1, \xi_2, \ldots$ distributed as
$\gamma^*$ yields, almost surely along the realization,
\[
 \lim_{n\to \infty} \frac{1}{n} \sum_{i=1}^n \delta_{\xi_i}  = \gamma^*,\ \ \ \ \ 
  \mbox{ weakly in } \mathcal{P}(H),
\]
{
Thus there is a measurable $\Omega_0 \subset \Omega$ with $\mathbb{P}(\Omega_0)=1$ such that for all $\omega \in \Omega_0$, \[\lim_{n\to\infty}\frac{1}{n} \sum_{i=1}^n \delta_{\xi_i(\omega)}  = \gamma^*\] in the weak convergence topology.
Since, $\gamma^*(\mathcal{D})=1$, we can assume without loss of generality that for all $\omega \in \Omega_0$, $\xi(\omega) \in \mathcal{D}$. Indeed, fix any element $x^* \in \mathcal{D}$ and define \[\tilde \xi_i(\omega) = \xi_i(\omega) 1_{\{\xi_i(\omega) \in \mathcal{D}\}} + x^* 1_{\{\xi_i(\omega) \in \mathcal{D}^c\}}.\] It is then easy to see that  $\frac{1}{n} \sum_{i=1}^n \delta_{\tilde \xi_i(\omega)}  \to \gamma^*$. 
Furthermore, by intersection with another set with full measure, we may assume without loss of generality that for all $\omega \in \Omega_0$, }
\[
\lim_{n\to \infty}   \frac{1}{n} \sum_{i=1}^n \bigl| Q^{-1/2}A\xi_i(\omega) + Q^{1/2}\nabla U(\xi_i(\omega)) \bigr|_H^2=  \int_H \bigl| Q^{-1/2} A x + Q^{1/2} \nabla U(x) \bigr|_H^2\, \gamma^*(dx)
   =  2 I(\gamma^*).
\]

Now, we fix a realization {$\tilde\omega \in \Omega_0$ so that for this $\tilde \omega$} both convergences above hold,
denote $x_i \coloneqq \xi_i(\tilde\omega) \in \mathcal{D}$, and choose
$k \in \mathbb{N}$ large enough so that the  measure
\begin{equation}\label{tilde-mu-def}
  \tilde\mu  \coloneqq  \frac{1}{k} \sum_{i=1}^k \delta_{x_i}
   =  \sum_{i=1}^k p_i\, \delta_{x_i},
  \qquad p_i \coloneqq \frac{1}{k}, \quad i = 1, \ldots, k,
\end{equation}
after possibly relabelling the $x_i$ to be distinct and
adjusting the weights $p_i$ accordingly, satisfies 
\begin{equation}\label{tilde-mu-deltaopt}
  F(\tilde \mu) + \frac{1}{2}\int_H \bigl| Q^{-1/2}A x + Q^{1/2}\nabla U(x) \bigr|_H^2\, \tilde\mu(dx)
  \leq  F(\gamma^*) + I(\gamma^*) + 2\delta
   \leq 
  \inf_{\gamma \in \mathcal{P}(H)} \bigl( F(\gamma) + I(\gamma) \bigr) + 3 \delta.
\end{equation}
In the sequel we work with the partition of $[0,1]$ associated with
$\tilde\mu$. Namely, we will set
\[P_0 \coloneqq 0,\ \ \ \ \ 
  P_i  \coloneqq  \sum_{j=1}^i p_j, \qquad i = 1, \ldots, k.
\]

\subsection{Construction of the controlled process}
\label{lb-control-construction}
We now construct, for each $\eps \in (0,1)$, an admissible control
$\varphi^\eps \in \mathcal{A}$ such that the associated controlled
process $\bar v^\eps$, solving~\eqref{veps-controlled}, approximates,
in occupation measure, the discrete target $\tilde\mu$ introduced in
Subsection~\ref{lb-reduction} and satisfying \eqref{tilde-mu-deltaopt}. 

Fix {two points $y_0 \in H$ and $y_1 \in \mathcal{D}$} and let $\eta \coloneqq y_1 - e^{A} y_0$.
For each {$\eps \in (0,1)$} we define the deterministic path
$f^\eps_{y_0, y_1} : {[0, \eps)} \to H$ by its expansion in the
eigenbasis $ (e_k )_{k \in\,\mathbb{N}}$,
\begin{equation}\label{travel-control-def}
  f^\eps_{y_0, y_1}(r)  \coloneqq  \sum_{k=1}^\infty c^\eps_k(y_0, y_1)\,
    e^{-\alpha_k(\eps - r)/\eps}\, e_k,
  \qquad r \in [0, \eps),
\end{equation}
with coefficients
\begin{equation}\label{travel-control-coeffs}
  c^\eps_k(y_0, y_1)  \coloneqq  \frac{2 \alpha_k}{1 - e^{-2 \alpha_k}}\,
    \eta_k,
  \qquad
  \eta_k  \coloneqq  \langle y_1 - e^{A} y_0,\, e_k \rangle_H.
\end{equation}
In particular, it is easy to check that
\begin{equation}\label{travel-design}
  \frac{1}{\eps} \int_{0}^{\eps} e^{(\eps - r) A / \eps}\,
    f^\eps_{y_0, y_1}(r)\, dr  =  y_1 - e^{A} y_0,
\end{equation}
and a short computation yields
\begin{equation}\label{travel-cost}
  \int_0^\eps \bigl| f^\eps_{y_0, y_1}(r) \bigr|_H^2\, dr
   =  \sum_{k=1}^\infty |\eta_k|^2 \, \frac{2 \alpha_k\, \eps}{1 - e^{-2 \alpha_k}}
   \leq  {c}\eps\, \bigl| (-A)^{1/2}(y_1 - e^{A} y_0) \bigr|_H^2.
\end{equation}

Now, we define, for $r \in [0,\eps)$ 
\begin{equation}\label{g-def}
  g^\eps_{y_0, y_1}(r)  \coloneqq  Q^{-1/2} f^\eps_{y_0, y_1}(r)
    =  \sum_{k=1}^\infty c^\eps_k(y_0, y_1)\, q_k^{-1/2}\,
   e^{-\alpha_k(\eps - r)/\eps}\, e_k.
\end{equation}
Notice that  $g^\eps_{y_0, y_1} \in\,L^2(0,\eps; H)$, as{}
\begin{equation}\label{travel-cost-CM}
  \int_0^\eps \bigl| g^\eps_{y_0, y_1}(r) \bigr|_H^2\, dr
   \leq  {c}\eps\, \bigl| Q^{-1/2}(-A)^{1/2}(y_1 - e^{A} y_0) \bigr|_H^2,
\end{equation}
which is finite for $y_0, y_1 \in \mathcal{D}$ {by Hypothesis \ref{H1}(2)}.

Construction \ref{cons-A} below gives the control $\varphi^{\eps}$ and the associated controlled process $\bar v^{\eps}$ that will be used in the 
proof of the lower bound.

\begin{construction}
\label{cons-A}
{\em Fix $\eps>0$ small enough that $\eps<\min_{1\leq j\leq k}p_j$.
We define the pair $(\bar v^\eps,\varphi^\eps)$ recursively on the
intervals
\[
  T_i^\eps:=(P_i,P_i+\eps],
  \qquad
  H_i^\eps:=(P_i+\eps,P_{i+1}],
  \qquad i=0,\ldots,k-1 .
\]
Set $\bar v^\eps(0)=u_0$.  Suppose that $\bar v^\eps$ has already been
constructed up to time $P_i$, and let
\[
  Y_i^\eps:=\bar v^\eps(P_i).
\]

{\bf The travel-phase control.} On the travel interval $T_i^\eps$, we first define $\bar v^\eps$ as the
unique mild solution of
\begin{equation}\label{eq:tser}
  d\bar v^\eps(t)
  =
  \frac1\eps
  \Bigl[
      A\bar v^\eps(t)
      +
      f^\eps_{Y_i^\eps,x_{i+1}}(t-P_i)
      +
      R_i^\eps(t,\bar v^\eps(t))
  \Bigr]dt
  +
  \frac{s(\eps)}{\sqrt\eps}\,dw^\eps(t),  \ \ \ \ \ 
  \bar v^\eps(P_i)=Y_i^\eps ,
\end{equation}
where
\[
  R_i^\eps(t,z)
  :=
  (B_\eps^{1/2}-Q^{1/2})
  \left(
      g^\eps_{Y_i^\eps,x_{i+1}}(t-P_i)
      -
      Q^{1/2}\nabla U(z)
  \right).
\]
After this state process has been constructed on $T_i^\eps$, identify the
control on $T_i^\eps$ by the feedback formula
\begin{equation}\label{travel-feedback-recursive}
  \varphi^\eps(t)
  :=
  g^\eps_{Y_i^\eps,x_{i+1}}(t-P_i)
  -
  Q^{1/2}\nabla U(\bar v^\eps(t)),
  \qquad t\in T_i^\eps .
\end{equation}
Then
\[
  B_\eps^{1/2}\varphi^\eps(t)
  =
  f^\eps_{Y_i^\eps,x_{i+1}}(t-P_i)
  -
  Q\nabla U(\bar v^\eps(t))
  +
  R_i^\eps(t,\bar v^\eps(t)),
\]
and hence \eqref{eq:tser} is precisely the
controlled equation \eqref{veps-controlled} with the above choice of control $\varphi^{\eps}$ on $T_i^\eps$. This is discussed further below this construction 
(see \eqref{tcpc}).

Next define
\begin{equation}\label{eq:hpc1}
  h_{i+1}:=
  Q^{-1/2}Ax_{i+1}+Q^{1/2}\nabla U(x_{i+1}) .
\end{equation}
{\bf The hold-phase control.} On the interval $H_i^\eps$, set
\begin{equation}\label{eq:hfr}
  \varphi^\eps(t):=-h_{i+1},
  \qquad t\in H_i^\eps ,
\end{equation}
and define $\bar v^\eps$ as the unique mild solution of
\begin{equation}\label{hold-state-equation-recursive}
  d\bar v^\eps(t)
  =
  \frac1\eps
  \Bigl[
      A\bar v^\eps(t)
      +
      Q\nabla U(\bar v^\eps(t))
      -
      B_\eps^{1/2}h_{i+1}
  \Bigr]dt
  +
  \frac{s(\eps)}{\sqrt\eps}\,dw^\eps(t), \ \ \ \ \   \bar v^\eps(P_i+\eps)=\bar v^\eps((P_i+\eps)-).
\end{equation}
Equivalently,
\begin{equation}\label{eq:conshold}
  B_\eps^{1/2}\varphi^\eps(t)
  =
  -\bigl(Ax_{i+1}+Q\nabla U(x_{i+1})\bigr)
  +
  \mathfrak r^\eps_{i+1},
  \qquad
  \mathfrak r^\eps_{i+1}:=
  (Q^{1/2}-B_\eps^{1/2})h_{i+1}.
\end{equation}

Proceeding in this way for $i=0,\ldots,k-1$ defines
$(\bar v^\eps,\varphi^\eps)$ on all of $[0,1]$. }  \hfill \qed
\end{construction}



$\,$\\

The equation in \eqref{eq:tser} is indeed well posed, since from \eqref{travel-cost} and \eqref{g-def}, conditional on
$\mathcal F_{P_i}$, the functions
\[
        s\mapsto f^\eps_{Y_i^\eps,x_{i+1}}(s-P_i),
        \qquad
        s\mapsto g^\eps_{Y_i^\eps,x_{i+1}}(s-P_i)
\]
belong to $L^2(P_i,P_i+\eps;H)$ a.s.  Moreover, since $\nabla U$ is globally
Lipschitz,
\begin{equation}\label{eq:rilip}
\begin{aligned}
  |R_i^\eps(t,z)-R_i^\eps(t,z')|_H
  &\leq
  \|B_\eps^{1/2}-Q^{1/2}\|_{\mathcal L(H)}
  \|Q^{1/2}\|_{\mathcal L(H)}
  [\nabla U]_{\mathrm{Lip}}
  |z-z'|_H .
\end{aligned}
\end{equation}
Thus $z\mapsto R_i^\eps(t,z)$ is globally Lipschitz, uniformly in $t$.
It also has at most linear growth.  Hence standard results (see e.g. 
\cite[Theorem 7.4]{DPZ}) give a unique adapted mild solution
of \eqref{eq:tser} on $[P_i,P_i+\eps]$.
The mild form of \eqref{eq:tser} is
\begin{equation}\label{travel-state-mild}
\begin{aligned}
  \bar v^\eps(t)
  &=
  e^{(t-P_i)A/\eps}Y_i^\eps +
  \frac1\eps
  \int_{P_i}^t
  e^{(t-s)A/\eps}
  \Bigl[
      f^\eps_{Y_i^\eps,x_{i+1}}(s-P_i)
      +
      R_i^\eps(s,\bar v^\eps(s))
  \Bigr]\,ds     \\[8pt]
  &\quad+
  \frac{s(\eps)}{\sqrt\eps}
  \int_{P_i}^t e^{(t-s)A/\eps}\,dw^\eps(s),
  \qquad  t\in[P_i,P_i+\eps].
\end{aligned}
\end{equation}

Also, with $\varphi^{\eps}$ defined on $T_i^{\eps}$ as in \eqref{travel-feedback-recursive},  
\begin{equation} \label{tcpc}
\begin{aligned}
  &B_\eps^{1/2}\varphi^\eps(t)
  =
  B_\eps^{1/2}
  \left[
      g^\eps_{Y_i^\eps,x_{i+1}}(t-P_i)
      -
      Q^{1/2}\nabla U(\bar v^\eps(t))
  \right]                                      \\[8pt]
  &\quad =
  Q^{1/2}g^\eps_{Y_i^\eps,x_{i+1}}(t-P_i)
  -
  Q\nabla U(\bar v^\eps(t))                    +
  (B_\eps^{1/2}-Q^{1/2})
  \left[
      g^\eps_{Y_i^\eps,x_{i+1}}(t-P_i)
      -
      Q^{1/2}\nabla U(\bar v^\eps(t))
  \right]\\[8pt]
  &\quad \quad \quad \quad =
  f^\eps_{Y_i^\eps,x_{i+1}}(t-P_i)
  -
  Q\nabla U(\bar v^\eps(t))
  +
  R_i^\eps(t,\bar v^\eps(t)).
\end{aligned}
\end{equation}

Substituting \eqref{tcpc} in \eqref{eq:tser}
\[
  d\bar v^\eps(t)
  =
  \frac1\eps
  \bigl(
      A\bar v^\eps(t)
      +
      Q\nabla U(\bar v^\eps(t))
  \bigr)dt
  +
  \frac1\eps B_\eps^{1/2}\varphi^\eps(t)\,dt
  +
  \frac{s(\eps)}{\sqrt\eps}\,dw^\eps(t).
\]
Thus $\bar v^{\eps}$ defines the controlled process given by \eqref{veps-controlled} with the associated control $\varphi^{\eps}$ (given in feedback form) defined by 
\eqref{travel-feedback-recursive}.


Regarding the hold-phase control, note that
since $x_{i+1} \in \mathcal{D}$, the vector $h_{i+1}$ defined in \eqref{eq:hpc1}
belongs to $H$
and thus the control $\varphi^\epsilon$ in \eqref{eq:hfr} is well defined and
for every $\epsilon \in\,(0,1)$ and $t \in H_i^{\eps}$,
\begin{equation}\label{hold-CM-norm}
  | \varphi^\eps(t) |_H^2 = \bigl| Q^{-1/2} A x_{i+1} + Q^{1/2} \nabla U(x_{i+1}) \bigr|_H^2,
  \qquad t \in\,(P_i+\epsilon,P_{i+1}].
\end{equation}
Moreover, from the strong convergence of $B_{\eps}$ to $Q$, the residual $\mathfrak{r}^\eps_{i+1}$ defined in \eqref{eq:conshold}
satisfies $|\mathfrak{r}^\eps_{i+1}|_H \to 0$, as $\eps\downarrow 0$. 


Since $Y_i^\eps$ is
$\mathcal F_{P_i}$-measurable and the maps in
\eqref{travel-feedback-recursive} and \eqref{eq:hfr} are
adapted feedback functions of the already constructed state, $\varphi^\eps$ is
progressively measurable.  Moreover, by the Lipschitz property of
$\nabla U$, the estimates \eqref{travel-cost-CM} and \eqref{hold-CM-norm},
and the moment estimate proved below in Lemma \ref{lb-moment}, $\varphi^\eps\in \mathcal{A}$.
  Thus $\varphi^\eps$ is an admissible control and $\bar v^\eps$
solves \eqref{veps-controlled} on the whole interval $[0,1]$.

\begin{Remark}\label{design-strategy}
{\em
The design of the control may be summarized as follows.

\smallskip

On a hold interval $(P_i + \eps, P_{i+1}]$, the
Cameron-Martin perturbation $B_\eps^{1/2}\varphi^\eps(t)$ exactly cancels
(up to a vanishing residual) the deterministic drift
$A x_{i+1} + Q\, \nabla U(x_{i+1})$. The shifted process
$z^\eps(t) \coloneqq \bar v^\eps(P_i + \eps + t) - x_{i+1}$ then satisfies,
up to the residual $\mathfrak{r}^\eps_{i+1}$,
\[
  d z^\eps  =  \frac{1}{\eps}\,\left(A z^\epsilon(t)+ \mathcal{V}_{x_{i+1}}(z^\eps)\right)\, dt
     +  \frac{s(\eps)}{\sqrt{\eps}}\, dw^\eps,
\]
so that the stabilization Lemma~\ref{stabilization} applies with
$x_* = x_{i+1}$ and $y_\eps = \bar v^\eps(P_i + \eps) - x_{i+1}$. The
cost of this control on $(P_i + \eps, P_{i+1}]$ is, by~\eqref{hold-CM-norm},
exactly
\[
  \frac{1}{2}\int_{P_i + \eps}^{P_{i+1}} |\varphi^\eps(t)|_H^2\, dt
   =  \frac{1}{2}\, (p_{i+1} - \eps)\, \bigl| Q^{-1/2}A x_{i+1} + Q^{1/2} \nabla U(x_{i+1}) \bigr|_H^2,
\]
which converges, as $\eps\downarrow 0$, to the contribution
$\frac{1}{2} p_{i+1}\,|Q^{-1/2}A x_{i+1} + Q^{1/2} \nabla U(x_{i+1})|_H^2$
of $x_{i+1}$ to $I(\tilde\mu)$.
\smallskip

On a travel interval $(P_i, P_i + \eps]$, the
Cameron-Martin perturbation has two pieces. The first piece,
$f^\eps_{\bar v^\eps(P_i), x_{i+1}}$, drives $\bar v^\eps$ from its
initial state to $x_{i+1}$ in time $\eps$ by following the explicit
deterministic path according to the drift~\eqref{travel-control-def};
and its squared $L^2$-norm is $O(\eps)$ by~\eqref{travel-cost}. The
second piece, $-Q\, \nabla U(\bar v^\eps(t))$, cancels the nonlinear drift
during the rapid travel, so that the effective dynamics on the
travel interval reduce to the linear case (up to a negligible term). Its squared $L^2$-norm over a travel interval is also
$O(\eps)$ by the moment bounds from Lemma \ref{lb-moment} below.
\smallskip

Finally, since  there are $k$ travel intervals,
each contributing $O(\eps)$,  the total travel cost is $O(k\,\eps) = O(\eps)$, hence negligible for the overall cost of the control.

}
\end{Remark}

\subsection{Convergence of the controlled empirical measures}
\label{lb-convergence}

The next three lemmas establish the moment bounds and the convergence of
$\bar v^\eps$ to the desired states on the travel and hold phases.

\begin{Lemma}\label{lb-moment}
With $\bar v^\eps$ as constructed above,
\[
  \sup_{\eps \in (0,1)}\, \sup_{0 \leq t \leq 1} \mathbb{E}\,|\bar v^\eps(t)|_H^2
   <  \infty.
\]
\end{Lemma}

{
\begin{proof}
As in Lemma~4.6 of~\cite{BudhirajaZoubouloglou2024}, we show by induction
on $i=0,1,\ldots,k$ that
\begin{equation}\label{lb-moment-induction}
  M_i \;\coloneqq\; \sup_{\eps\in(0,1)}\ \sup_{0\leq t\leq P_i}
     \mathbb{E}\,\bigl|\bar v^\eps(t)\bigr|_H^2 \;<\;\infty .
\end{equation}
For $i=0$ this holds since $P_0=0$ and $\bar v^\eps(0)=u_0$ is fixed. Assume
\eqref{lb-moment-induction} for some $i\in\{0,\ldots,k-1\}$. In particular
$\sup_\eps\mathbb{E}\,|Y_i^\eps|_H^2\leq M_i<\infty$, where we recall  $Y_i^\eps=\bar v^\eps(P_i)$.

\smallskip
On the interval $T_i^\eps=(P_i,P_i+\eps]$, from the mild
form~\eqref{travel-state-mild} and the contractivity
$\|e^{rA/\eps}\|_{\mathcal L(H)}\leq 1$, we get
\[ 
  |\bar v^\eps(t)|_H
  \leq |Y_i^\eps|_H
    + \frac1\eps \int_{P_i}^{t}\! |f^\eps_{Y_i^\eps,x_{i+1}}(s-P_i)|_H\,ds
    + \frac1\eps \int_{P_i}^{t}\! |R_i^\eps(s,\bar v^\eps(s))|_H\,ds
    + |Z^\eps(t)|_H,
\]
where $Z^\eps$ is the stochastic convolution in~\eqref{eq:stoccon}.
For the forcing term, Cauchy-Schwarz and the cost bound~\eqref{travel-cost} give
\[
  \frac1\eps\!\int_{P_i}^{t}\!|f^\eps_{Y_i^\eps,x_{i+1}}(s-P_i)|_H\,ds
  \leq \eps^{-1/2}\Bigl(\,\int_0^\eps |f^\eps_{Y_i^\eps,x_{i+1}}(r)|_H^2\,dr\Bigr)^{1/2}
  \leq c\,\bigl|(-A)^{1/2}\bigl(x_{i+1}-e^{A}Y_i^\eps\bigr)\bigr|_H .
\]
Since $x_{i+1}\in\mathcal D\subset D((-A)^{1/2})$ and analytic smoothing yields
\[\|(-A)^{1/2}e^{A}\|_{\mathcal L(H)}=\sup_k\alpha_k^{1/2}e^{-\alpha_k}<\infty,\]
this is bounded by $c\,(1+|Y_i^\eps|_H)$.

For the residual $R_i^\eps$, the Lipschitz property  in \eqref{eq:rilip}, together with Cauchy-Schwarz inequality gives,
%
\[
\begin{aligned}
  \frac1{\eps^2}\,\mathbb E\Bigl(\,\int_{P_i}^{t}|R_i^\eps(s,\bar v^\eps(s))|_H\,ds\Bigr)^{2}
	&\leq \frac{c}{\eps^2} \,\mathbb E\Bigl(\,\int_{P_i}^{t} |\bar v^\eps(s)|_H + |R_i^\eps(s,0)|_H\,ds\Bigr)^{2}\\
	&\leq c \left( (1+ \mathbb E|Y^\eps_i|_H^2) + \frac{1}{\eps} \int_{P_i}^{t} \mathbb E|\bar v^\eps(s)|_H^2\,ds\right),
	\end{aligned}
\]
where the last inequality uses the definition of $R_i^{\eps}(s,0)$ and the estimate
$$
\mathbb E\left(\sup_{0\leq s \leq 1}| g^\eps_{Y_i^\eps,x_{i+1}}(s)|^2\right) \leq c\, \mathbb E(1 + |Y^\eps_i|_H^2).$$
Finally, by the stochastic-convolution
estimate in \eqref{as4} (with $\beta=0$)
\[
 \sup_{\epsilon\in\,(0,1)}\sup_{t \in\,[0,1]} \mathbb E|Z^\eps(t)|_H^2
  \leq c.
\]

 Collecting the four terms and applying Gr\"onwall over
$[P_i,P_i+\eps]$  yields
\[
  \sup_{P_i\leq t\leq P_i+\eps}\mathbb E|\bar v^\eps(t)|_H^2
  \;\leq\; c\bigl(1+\mathbb E|Y_i^\eps|_H^2\bigr)
  \;\leq\; c\,(1+M_i)\;<\;\infty .
\]

\smallskip
On the interval $H_i^\eps=(P_i+\eps,P_{i+1}]$, by construction the shifted
process \[z^\eps(t)\coloneqq\bar v^\eps(P_i+\eps+t)-x_{i+1},\ \ \ \ \ 0\leq t\leq p_{i+1}-\eps,\] is the mild solution of the stabilization
SDE~\eqref{stab-sde}, with $x_*=x_{i+1}$ and initial datum
$y_\eps=\bar v^\eps(P_i+\eps)-x_{i+1}$, driven by the increments of $w^\eps$
after $P_i+\eps$, which are independent of $\mathcal F_{P_i+\eps}$. Conditioning
on $\mathcal F_{P_i+\eps}$ and applying the a-priori bound~\eqref{stab-moment} of
Lemma~\ref{stab-bounds}, whose constant $\kappa_1$ is independent of $x_*$,
\[
  \sup_{0\leq t\leq 1}\mathbb E|z^\eps(t)|_H^2
  \leq \kappa_1\bigl(1+\mathbb E|\bar v^\eps(P_i+\eps)-x_{i+1}|_H^2\bigr).
\]
Hence, using the travel-phase bound at $t=P_i+\eps$,
\[
  \sup_{P_i+\eps\leq t\leq P_{i+1}}\mathbb E|\bar v^\eps(t)|_H^2
  \leq 2|x_{i+1}|_H^2
    + 2\kappa_1\bigl(1+\mathbb E|\bar v^\eps(P_i+\eps)-x_{i+1}|_H^2\bigr) \leq \blue{c(1+M_i)}
  \;<\;\infty .
\]
Combining the two phases gives $M_{i+1}<\infty$, completing the induction.
Taking $i=k$, so that  $P_k=1$, yields the claim.
\end{proof}}

\begin{Lemma}\label{lb-endpoint}
For each $i = 0, 1, \ldots, k - 1$,
\begin{equation}\label{stab-initial}
 \lim_{\e\to 0} \eps\,\mathbb{E}\,|\bar v^\eps(P_i + \eps) - x_{i+1}|_H^2
  =0.
\end{equation}
\end{Lemma}

Before giving the proof of Lemma \ref{lb-endpoint}, we discuss one preliminary fact about
the travel-phase equation.

The travel control~\eqref{travel-control-def} was designed precisely so
{that the sum of the first two terms on the right side of \eqref{travel-state-mild}, evaluated at
$t = P_i + \eps$, equals $x_{i+1}$, up to a negligible error term}. Indeed, the identity 
\begin{equation}\label{linear-endpoint-identity}
  e^{A}\, \bar v^\eps(P_i)  +  \frac{1}{\eps} \int_{P_i}^{P_i + \eps} e^{(P_i + \eps - s) A/\eps}\,
    f^\eps_{\bar v^\eps(P_i), x_{i+1}}(s - P_i)\, ds  =  x_{i+1}.
\end{equation}
can be verified by expanding both sides in the
eigenbasis $ (e_k )_{k \in\,\mathbb{N}}$ and using~\eqref{travel-control-coeffs}.
To see this, set $r = s - P_i$
and note that the integral of the $k$-th component of the integrand in the eigenbasis expansion  becomes
\[
  \frac{1}{\eps}
  c^\eps_k(\bar v^\eps(P_i), x_{i+1})
    \int_0^\eps e^{-\alpha_k(\eps - r)/\eps}\, e^{-\alpha_k(\eps - r)/\eps}\, dr
    =  c^\eps_k(\bar v^\eps(P_i), x_{i+1}) 
     \frac{1 - e^{-2\alpha_k}}{2\alpha_k},
\]
which, by~\eqref{travel-control-coeffs}, equals exactly
$\eta_k = \langle x_{i+1} - e^{A}\bar v^\eps(P_i),\, e_k\rangle_H$, proving
the claimed identity in \eqref{linear-endpoint-identity} in the $k$-th coordinate for every $k$.

\begin{proof}[Proof of Lemma~\ref{lb-endpoint}]

In what follows, we write, for $t \in T_i^{\eps}$,
\[r^\eps_i(t)  {:=
R_i^{\eps}(t, \bar v^{\eps}(t))}=(B_\eps^{1/2}-Q^{1/2})\left(g^\eps_{\bar v^\eps(P_i),x_{i+1}}(t-P_i)
- Q^{1/2}\nabla U(\bar v^\eps(t))\right).\]
Due to the equiboundedness of $\|B^{1/2}_{\epsilon}-Q^{1/2}\|_{\mathcal{L}(H)}$ (see \eqref{CB}), thanks to ~\eqref{travel-cost-CM}  we have
\begin{align*}
 & \int_{P_i}^{P_i+\eps}\mathbb{E}\,|r_i^\eps(t)|_H^2\,dt
  \;\leq\;
  c\int_{P_i}^{P_i+\eps}\mathbb{E}\,
    |g^\eps_{\bar v^\eps(P_i),x_{i+1}}(t-P_i)|_H^2\,dt +\;
  c\int_{P_i}^{P_i+\eps}\mathbb{E}\,
    |Q^{1/2}\nabla U(\bar v^\eps(t))|_H^2\,dt\\[10pt]
 \leq &\,c\,\eps\,\mathbb{E}\,
|Q^{-1/2}(-A)^{1/2}(x_{i+1}-e^A\bar v^\eps(P_i))|_H^2 +c\,\eps\,\left(\|Q^{1/2}\|^2_{\mathcal{L}(H)}[\nabla U]^2_{\mathrm{Lip}}
\sup_{t\in [0,1]}\mathbb{E}\,|\bar v^\eps(t)|_H^2 \right).
\end{align*}
Using Lemma \ref{lb-moment} and Hypothesis \ref{H1}(2),
this implies
\begin{equation}\label{residual-vanish}
  \lim_{\eps\to 0}
  \int_{P_i}^{P_i+\eps}\mathbb{E}\,|r^\eps_i(t)|_H^2\,dt \;=\; 0.
\end{equation}

By~\eqref{travel-state-mild}, evaluated at $t = P_i + \eps$, and by
~\eqref{linear-endpoint-identity},
\begin{equation}\label{endpoint-decomposition}
  \bar v^\eps(P_i + \eps) - x_{i+1}
   =  \mathcal{R}^\eps_i  +  \mathcal{N}^\eps_i,
\end{equation}
where
\begin{equation}\label{R-residual-int}
  \mathcal{R}^\eps_i  \coloneqq  \frac{1}{\eps}\int_{P_i}^{P_i + \eps}
    e^{(P_i + \eps - s) A/\eps}\, r^\eps_i(s)\, ds,
  \qquad
  \mathcal{N}^\eps_i  \coloneqq  \frac{s(\eps)}{\sqrt{\eps}} \int_{P_i}^{P_i + \eps} e^{(P_i + \eps - s) A/\eps}\, dw^\eps(s).
\end{equation}

By It\^{o} isometry and a change of variables, we have
\[
  \mathbb{E}\,|\mathcal{N}^\eps_i|_H^2
  \;=\; s^2(\eps)\int_0^1
    \mathrm{tr}\bigl(e^{uA}\,B_\eps\,e^{uA}\bigr)\,du.
\]
Then, due to  condition~\eqref{as1} in
Hypothesis~\ref{H1}.4, we have
\[
  \eps\,\mathbb{E}\,|\mathcal{N}^\eps_i|_H^2
  \;\leq\; c\,\eps  \;\longrightarrow\; 0,\ \ \ \ \text{as}\ \eps\downarrow 0.
\]
 
Concerning the  residual $\mathcal{R}^\eps_i$, we have\[
  |\mathcal{R}^\eps_i|_H
  \;\leq\; \frac{1}{\sqrt{2\lambda\,\eps}}
    \Bigl(\int_{P_i}^{P_i+\eps}|r^\eps_i(s)|_H^2\,ds\Bigr)^{\!1/2},
\]
hence
\[
  \eps\,\mathbb{E}\,|\mathcal{R}^\eps_i|_H^2
  \;\leq\; \frac{1}{2\lambda}\,
    \mathbb{E}\int_{P_i}^{P_i+\eps}|r^\eps_i(s)|_H^2\,ds
  \;\longrightarrow\; 0,\ \ \ \ \text{as}\ \eps\downarrow 0,
\]
where the last step uses~\eqref{residual-vanish}.

By combining the last two estimates with~\eqref{endpoint-decomposition}, we obtain \eqref{stab-initial}.

\end{proof}

\begin{Lemma}\label{lb-hold-occupation}
For each $i = 0, 1, \ldots, k - 1$
\[\lim_{\e\to 0}\,
  \mathbb{E}\, d_{\mathrm{BL}}\!\left(
    \frac{1}{p_{i+1} - \eps}\int_{P_i + \eps}^{P_{i+1}} \delta_{\bar v^\eps(s)}\, ds, 
    \delta_{x_{i+1}}
  \right)  =  0.
\]
\end{Lemma}

\begin{proof}
 By translation invariance of
$d_{\mathrm{BL}}$, under the substitution $x \mapsto x - x_{i+1}$
applied to test functions $f \in \mathrm{Lip}_{b,1}(H)$, we get
\begin{equation}\label{hold-occ-translation}
\begin{array}{l}  \ds{\mathbb{E}\, d_{\mathrm{BL}}\!\left(\frac{1}{p_{i+1} - \eps}\int_{P_i + \eps}^{P_{i+1}} \delta_{\bar v^\eps(s)}\, ds,  \delta_{x_{i+1}}\right)   =  \mathbb{E}\, d_{\mathrm{BL}}\!\left(\frac{1}{p_{i+1} - \eps}\int_0^{p_{i+1} - \eps} \delta_{z^\eps(s)}\, ds,  \delta_0\right),}
\end{array}
\end{equation}
where $z^\eps(s) \coloneqq \bar v^\eps(P_i + \eps + s) - x_{i+1}$. By
the design of the hold-phase
control, the process $z^\eps$
solves
\begin{equation}\label{zeps-equation}
  d z^\eps(t)  =\frac 1\e \left(Az^\e(t) + \mathcal{V}_{x_{i+1}}(z^\eps(t))\right)\, dt
     +  \frac{1}{\eps}\, \mathfrak{r}^\eps_{i+1}\, dt
     +  \frac{s(\eps)}{\sqrt{\eps}}\, dw^\eps(t),
\end{equation}
with $z^\eps(0) = \bar v^\eps(P_i + \eps) - x_{i+1}$, where
$\mathcal{V}_{x_{i+1}}$ is the shifted drift~\eqref{shifted-drift-def}
at $x_* = x_{i+1}$ and $\mathfrak{r}^\eps_{i+1}$ is the residual
defined in ~\eqref{eq:conshold}.

If we compare~\eqref{zeps-equation} with the stabilization
SDE~\eqref{stab-sde}, for $x_* = x_{i+1}$
and the same initial condition $y_\eps = \bar v^\eps(P_i + \eps) - x_{i+1}$, we see that  they differ only by the deterministic
forcing $\eps^{-1} \mathfrak{r}^\eps_{i+1}$. Hence,   the dissipativity~\eqref{shifted-dissipativity}
and Gronwall's inequality yield
\begin{equation}\label{zeps-vs-stab}
  \sup_{t \in [0, p_{i+1} - \eps]} \mathbb{E}\, |z^\eps(t) - y_{x_{i+1}}^\eps(t)|_H^2
   \leq  \frac{|\mathfrak{r}^\eps_{i+1}|_H^2}{\omega^2}\, \bigl( 1 - e^{-\omega(p_{i+1}-\eps)/\eps} \bigr)
   \leq  \frac{|\mathfrak{r}^\eps_{i+1}|_H^2}{\omega^2}.
\end{equation}
Since $|\mathfrak{r}^\eps_{i+1}|_H \to 0$, this shows that $z^\eps$ and
$y_{x_{i+1}}^\eps$ are asymptotically equal in $L^2(\Omega,\mathbb{P}; H)$,
uniformly on $[0, p_{i+1} - \eps]$.

From Lemma~\ref{lb-endpoint} we have 
$\sup_{\epsilon \in\,(0,1)} \eps\, \mathbb{E}|y_\eps|_H^2 < \infty$. By
Lemma~\ref{stabilization} applied to $y_{x_{i+1}}^\eps$
(with the random initial condition $y_\eps$ first conditioned out and
then averaged via dominated convergence, {and an application of Markov's inequality}),
\[
  \lim_{\e\to 0}\mathbb{E}\, d_{\mathrm{BL}}\!\left(\frac{1}{p_{i+1} - \eps}\int_0^{p_{i+1} - \eps} \delta_{y_{x_{i+1}}^\eps(s)}\, ds,  \delta_0\right)  =  0.
\]
Combining this with~\eqref{zeps-vs-stab} and using the properties of $d_{\mathrm{BL}}$, together with the identity in   \eqref{hold-occ-translation}, yields the
claimed convergence in the lemma.
\end{proof}

\subsection{Cost estimate and occupation measure convergence}
\label{lb-cost-sec}

We now estimate the total Cameron-Martin cost and characterize the limit of the
empirical measure $\bar\nu^\eps$.

\begin{Proposition}\label{lb-cost}
With $\varphi^\eps$ as in Construction \ref{cons-A},
\[
  \limsup_{\eps \to 0}\, \frac{1}{2}\, \mathbb{E}\int_0^1 |\varphi^\eps(t)|_H^2\, dt
   =  \frac{1}{2}\, \int_H \bigl| Q^{-1/2}A x + Q^{1/2} \nabla U(x) \bigr|_H^2\, \tilde\mu(dx).
\]
\end{Proposition}

\begin{proof}
We decompose the total cost into travel and hold contributions
\[
  \frac{1}{2}\, \mathbb{E}\int_0^1 |\varphi^\eps(t)|_H^2\, dt
   =  \frac{1}{2}\sum_{i=0}^{k-1} \mathbb{E}\int_{P_i}^{P_i + \eps} |\varphi^\eps(t)|_H^2\, dt
    +  \frac{1}{2}\sum_{i=0}^{k-1} \mathbb{E}\int_{P_i + \eps}^{P_{i+1}} |\varphi^\eps(t)|_H^2\, dt.
\]
For the travel contributions, by~\eqref{travel-feedback-recursive}
\[
  \int_{P_i}^{P_i + \eps} |\varphi^\eps(t)|_H^2\, dt
    \leq  2 \int_0^\eps \bigl| g^\eps_{\bar v^\eps(P_i), x_{i+1}}(r) \bigr|_H^2\, dr
     +  2 \int_{P_i}^{P_i + \eps} \bigl| Q^{1/2}\, \nabla U(\bar v^\eps(t)) \bigr|_H^2\, dt,
\]
 the cost estimate~\eqref{travel-cost-CM} gives
 \[
   \mathbb{E}\int_0^\eps \bigl| g^\eps_{\bar v^\eps(P_i), x_{i+1}}(r) \bigr|_H^2\, dr
    \leq  c\eps\, \mathbb{E}\, \bigl| Q^{-1/2}(-A)^{1/2}(x_{i+1} - e^{A}\bar v^\eps(P_i)) \bigr|_H^2
    \leq  c_i\, \eps,
\]
where $c_i < \infty$.
{Indeed,
\(x_{i+1}\in\mathcal D\) implies
$
  Q^{-1/2}(-A)^{1/2}x_{i+1}\in H$.
Moreover, by Hypothesis \ref{H1}(2), for every \[|Q^{-1/2}(-A)^{1/2}e^A y|_H^2
  \leq
  C |y|_H^2,\ \ \ \ y \in\,H.\]
Thus, using the moment bound of
Lemma~\ref{lb-moment} gives
\[\mathbb E
  \bigl|
  Q^{-1/2}(-A)^{1/2}
  (x_{i+1}-e^A\bar v^\eps(P_i))
  \bigr|_H^2
  <\infty.\] Combining the above observations gives the claim $c_i<\infty$.}

By the Lipschitz property of $\nabla U$ in Hypothesis~\ref{H2}.1 and the
boundedness of $Q^{1/2}$,
\[
   \mathbb{E}\int_{P_i}^{P_i + \eps} \bigl| Q^{1/2}\, \nabla U(\bar v^\eps(t)) \bigr|_H^2\, dt
    \leq  \|Q\|_{\mathcal{L}(H)}\, \eps\, \cdot 2 \bigl( [\nabla U]_{\text{\tiny Lip}}^2\, M_1 + |\nabla U(0)|_H^2 \bigr)
    \leq  c\,\eps,
\]
where $M_1 \coloneqq \sup_{\eps \in (0,1)} \sup_{t\in [0,1]} \mathbb{E}\,|\bar v^\eps(t)|_H^2 < \infty$
by Lemma~\ref{lb-moment}. Hence
\[
  \lim_{\e\to 0}\sum_{i=0}^{k-1}\mathbb{E}\int_{P_i}^{P_i + \eps} |\varphi^\eps(t)|_H^2\, dt
   =   0.
\]

For the hold contributions, in view of ~\eqref{hold-CM-norm}
\begin{align*}
  \frac{1}{2}\sum_{i=0}^{k-1} \mathbb{E}\int_{P_i + \eps}^{P_{i+1}} |\varphi^\eps(t)|_H^2\, dt
  &= \frac{1}{2}\sum_{i=0}^{k-1} (p_{i+1} - \eps)\, \bigl| Q^{-1/2}A x_{i+1} + Q^{1/2} \nabla U(x_{i+1}) \bigr|_H^2,\end{align*}
  so that 
   \[\begin{array}{l}
\ds{	\lim_{\e\to 0} \frac{1}{2}\sum_{i=0}^{k-1} \mathbb{E}\int_{P_i + \eps}^{P_{i+1}} |\varphi^\eps(t)|_H^2\, dt= \frac{1}{2}\sum_{i=0}^{k-1} p_{i+1}\, \bigl| Q^{-1/2}A x_{i+1} + Q^{1/2} \nabla U(x_{i+1}) \bigr|_H^2 }\\[14pt]
  \ds{\quad \quad \quad \quad \quad \quad\quad \quad \quad\quad \quad \quad\quad \quad \quad\quad \quad \quad= \frac{1}{2}\int_H \bigl| Q^{-1/2}A x + Q^{1/2} \nabla U(x) \bigr|_H^2\, \tilde\mu(dx).}
\end{array}\]
Combining the travel and hold cost estimates completes the proof of the proposition.
\end{proof}

\begin{Proposition}\label{lb-empirical}
With $\bar\nu^\eps = \int_0^1 \delta_{\bar v^\eps(s)}\, ds$,
$\lim_{\e\to 0}\bar\nu^\eps  =  \tilde\mu$ in  $\mathcal{P}(H)$, in probability.
\end{Proposition}

\begin{proof}We decompose \[\bar\nu^\eps = \sum_{i=0}^{k-1}\bigl( (p_{i+1} - \eps)  \,\theta^\eps_i + \eps \, \tau^\eps_i \bigr),\]
where
\[
  \theta^\eps_i  \coloneqq  \frac{1}{p_{i+1} - \eps}\int_{P_i + \eps}^{P_{i+1}} \delta_{\bar v^\eps(s)}\, ds,
  \qquad
  \tau^\eps_i  \coloneqq  \frac{1}{\eps}\int_{P_i}^{P_i + \eps} \delta_{\bar v^\eps(s)}\, ds.
\]
By Lemma~\ref{lb-hold-occupation}, $\theta^\eps_i \to \delta_{x_{i+1}}$
in probability. Moreover,  the contribution of the $\tau^\eps_i$ is $O(\eps)$, which vanishes as $\e\downarrow 0$. This result follows.
\end{proof}

\subsection{Proof of the Laplace lower bound}\label{lb-assembly}
 By the variational
representation~\eqref{varrep-scaled}, the admissibility of
$\varphi^\eps \in \mathcal{A}$ as constructed in
Subsection~\ref{lb-control-construction}, and the identification of the
controlled process with $\bar\nu^\eps$, we obtain
\begin{align*}
  -\eps\, s^2(\eps)\, \log \mathbb{E}\, \exp\!\left( -\frac{1}{\eps s^2(\eps)}\,F(\nu^\eps) \right)
  & \leq  \mathbb{E}\!\left( F(\bar\nu^\eps) + \frac{1}{2}\int_0^1 |\varphi^\eps(t)|_H^2\, dt \right).
\end{align*}
Since $F \in C_b$, Proposition~\ref{lb-empirical} allows us to handle the limsup of the first term on the right-hand side, and  
Proposition~\ref{lb-cost} the limsup of  the second one, so that
\begin{align*}
  \limsup_{\eps \to 0}\, -\eps\, s^2(\eps)\, \log \mathbb{E}\, &\exp\!\Big( -\frac{1}{\eps s^2(\eps)}\,F(\nu^\eps) \Big)
  \leq  F(\tilde\mu) + \frac{1}{2}\int_H \bigl| Q^{-1/2}A x + Q^{1/2} \nabla U(x) \bigr|_H^2\, \tilde\mu(dx) \\[12pt]
  & =  F(\tilde\mu)  +  I(\tilde\mu)  \leq  \inf_{\gamma \in \mathcal{P}(H)} \bigl( F(\gamma) + I(\gamma) \bigr) + 3\delta,
\end{align*}
where the last inequality is~\eqref{tilde-mu-deltaopt}. Since
$\delta > 0$ is arbitrary, the proof of~\eqref{laplace-lower-restate}
is complete. \qed

\section{Compactness of level sets}
\label{compactness}

In this section we complete the proof of Theorem~\ref{main} by
showing that the functional $I : \mathcal{P}(H) \to [0, \infty]$
defined in~\eqref{action} is a rate function on $\mathcal{P}(H)$.

We recall from Section~\ref{as30} the notation
\[
\Psi(y)  \coloneqq   Q^{-1/2} A y \,+\, Q^{1/2} \nabla U(y),
\]
defined for $y \in D(A)$ with $A y \in \mathrm{Range}(Q^{1/2})$
(with the convention $|\Psi(y)|_H = +\infty$ otherwise), so that
\[
  I(\gamma)  =  \frac{1}{2}\int_H \bigl|\Psi(y)\bigr|_H^2\, \gamma(dy),
  \qquad \gamma \in \mathcal{P}(H).
\]
We also recall the dissipativity constant $\omega$ from Hypothesis \ref{H2}
 and the
common orthonormal eigenbasis $ (e_k )_{k \in \N}$ of $-A$ and $Q$
from Remark~\ref{QAcommute}.
Recall that, for every
$n \in \N$, $\pi_n : H \to \pi_n(H) = \mathrm{span}\{e_1,\ldots,e_n\}$
denotes the orthogonal projection onto the first $n$ modes.

We begin with the key pointwise coercivity estimate.

\begin{Lemma}\label{coercivity-Psi}
There is a constant
$\kappa_0 \in (0, \infty)$, depending only on
$\lambda$, $[\nabla U]_{\mathrm{Lip}}$ and $\|Q\|_{\mathcal{L}(H)}$, such that
for every $y \in D(A)$ with $A y \in \mathrm{Range}(Q^{1/2})$,
\begin{equation}\label{coerc-y}
  |y|_H  \leq  \kappa_0 \, \bigl|\Psi(y)\bigr|_H,
\end{equation}
and
\begin{equation}\label{coerc-Ay}
  |A y|_H  \leq  \kappa_0\, \bigl|\Psi(y)\bigr|_H.
\end{equation}
\end{Lemma}

\begin{proof}
Fix $y \in D(A)$ with $A y \in \mathrm{Range}(Q^{1/2})$, so that
$\Psi(y) \in H$ and 
\[A y + Q\, \nabla U(y) = Q^{1/2}\, \Psi(y).\]
Taking
the inner product with $y$ and using $\langle A y, y\rangle_H \leq
-\lambda |y|_H^2$, together with $\nabla U(0)=0$ and
the Lipschitz property of $\nabla U$,
\[
  \langle Q^{1/2}\,\Psi(y),\, y\rangle_H
   =  \langle A y, y\rangle_H + \langle Q\, \nabla U(y), y\rangle_H
   \leq  -\lambda |y|_H^2
     \,+\, \|Q\|_{\mathcal{L}(H)}\, [\nabla U]_{\mathrm{Lip}}\, |y|_H^2
   =  -\omega |y|_H^2.
\]
On the other hand, by the Cauchy-Schwarz inequality,
\[
  \bigl|\langle Q^{1/2}\,\Psi(y),\, y\rangle_H\bigr|
   =  \bigl|\langle \Psi(y),\, Q^{1/2}\, y\rangle_H\bigr|
   \leq  \|Q^{1/2}\|_{\mathcal{L}(H)}\, |\Psi(y)|_H\, |y|_H,
\]
so that
\[
  \omega\, |y|_H^2  \leq  \|Q^{1/2}\|_{\mathcal{L}(H)}\, |\Psi(y)|_H\, |y|_H,
\]
which gives~\eqref{coerc-y} with $\kappa_0 \geq \omega^{-1} \|Q^{1/2}\|_{\mathcal{L}(H)}$.

Next, from $A y = Q^{1/2}\Psi(y) - Q\, \nabla U(y)$ and the Lipschitz
property of $\nabla U$ (with $\nabla U(0)=0$),
\[
  |A y|_H
   \leq  \|Q^{1/2}\|_{\mathcal{L}(H)}\,|\Psi(y)|_H
     + \|Q\|_{\mathcal{L}(H)}\,[\nabla U]_{\mathrm{Lip}}\,|y|_H.
\]
Combining this with~\eqref{coerc-y} yields~\eqref{coerc-Ay},
upon enlarging $\kappa_0$ if necessary.
\end{proof}

\begin{Remark}\label{Psi-domain}
{\em
Estimate~\eqref{coerc-Ay} shows in particular that, on the 
domain $\mathcal{D}$ constructed in \eqref{as35}
the quantity $|\Psi(y)|_H$ controls the full graph norm of $A$. Since
$A$ has compact resolvent, the sublevel sets
$\{y \in \mathcal{D} : |\Psi(y)|_H \leq R\}$ are relatively compact in
$H$ for every $R \in (0,\infty)$. This is the geometric content
underlying the tightness step of Theorem~\ref{I-rate}
below.
}
\end{Remark}

We next address the lower semicontinuity of the integrand. Because
$\Psi$ involves the unbounded operators $A$ and $Q^{-1/2}$, the lower semicontinuity of the map
$y \mapsto |\Psi(y)|_H^2$ is not immediate. We bypass this difficulty by writing $|\Psi(y)|_H^2$
as a supremum, over $n \in \N$, of expressions involving only the
bounded operators $A \pi_n$ and $Q^{-1/2}\pi_n$ on $\pi_n(H)$.

For $n \in \N$ and $y \in H$, set
\[
  \Psi_n(y)  \coloneqq  Q^{-1/2}\, A\, \pi_n y
     \,+\, Q^{1/2}\, \pi_n \nabla U(y),
\]
which is well defined, since $\pi_n(H) \subset D(A)$ and $\pi_n$
commutes with $Q^{1/2}$ and with $Q^{-1/2}$ on $\pi_n(H)$. Note that
$\Psi_n(y) \in \pi_n(H)$, and that $\Psi_n$ is continuous on $H$ as a
composition of bounded linear operators and the continuous map $\nabla U$.

\begin{Lemma}\label{Psi-truncation}
For every $y \in H$, the map $n \mapsto |\Psi_n(y)|_H^2$ is
nondecreasing, and
\begin{equation}\label{Psi-supremum}
  \bigl|\Psi(y)\bigr|_H^2  =  \sup_{n \in \N}\, \bigl|\Psi_n(y)\bigr|_H^2
    =  \lim_{n \to \infty} \bigl|\Psi_n(y)\bigr|_H^2,
\end{equation}
where both sides are interpreted in $[0, +\infty]$.
In particular, $y \mapsto |\Psi(y)|_H^2$ is lower semicontinuous on $H$.
\end{Lemma}

\begin{proof}
Since $ (e_k )_{k \in\,\mathbb{N}}$ diagonalizes both $A$ and $Q$, for every $y \in H$ we
have, with $y_k \coloneqq \langle y, e_k\rangle_H$ and
$d_k(y) \coloneqq \langle \nabla U(y), e_k\rangle_H$,
\[
  \Psi_n(y)  =  \sum_{k=1}^n \Bigl( q_k^{-1/2}\, (-\alpha_k)\, y_k
    \,+\, q_k^{1/2}\, d_k(y) \Bigr)\, e_k,
\]
so that
\[
  |\Psi_n(y)|_H^2
   =  \sum_{k=1}^n \Bigl( q_k^{-1/2}(-\alpha_k) y_k
    + q_k^{1/2} d_k(y) \Bigr)^2.
\]
The sequence is therefore nondecreasing
in $n$, and its supremum equals $|\Psi(y)|_H^2$ by the definition of
$\Psi$ and the convention $|\Psi(y)|_H = +\infty$ outside $\mathcal{D}$.
This proves~\eqref{Psi-supremum}.

Lower semicontinuity follows: for each fixed $n$, $\Psi_n$ is
continuous on $H$, so $y \mapsto |\Psi_n(y)|_H^2$ is continuous; and a
supremum of continuous functions is lower semicontinuous.
\end{proof}

We are now ready to prove that $I$ is a rate function. Fix
$M \in (0, \infty)$ and consider the level set
\[
  \Phi_M  \coloneqq  \bigl\{ \gamma \in \mathcal{P}(H)  :  I(\gamma) \leq M \bigr\}.
\]

We can now conclude the proof of compactness of level sets.

\begin{Theorem}\label{I-rate}
Under Hypotheses~\ref{H1} and~\ref{H2}, the functional
$I : \mathcal{P}(H) \to [0,\infty]$ defined in~\eqref{action} is a
rate function on $\mathcal{P}(H)$.
\end{Theorem}

\begin{proof}
We first show that for every $M<\infty$, $\Phi_M$ is a relatively compact subset of
$\mathcal{P}(H)$. 
For $R<\infty$, let
\[
  K_R  \coloneqq  \bigl\{ y \in H  :  |\Psi(y)|_H \leq R \bigr\}.
\]
From Remark \ref{Psi-domain}, for each $R<\infty$, $K_R$ is a compact subset of $H$. Also, by Markov's inequality, as $R\to \infty$,
$$
\sup_{\gamma \in \Phi_M} \gamma(K_R^c) \leq \frac{1}{R^2} \int_{H} |\Psi(y)|_H^2 \gamma(dy) \leq \frac{M}{R^2} \to 0$$
This proves the desired relative compactness. 
To complete the proof it suffices to argue that the functional
$I : \mathcal{P}(H) \to [0,\infty]$ is lower semicontinuous.
For that, consider
 $ (\gamma_n )_{n \in \N} \subset \mathcal{P}(H)$ converging weakly
to $\gamma \in \mathcal{P}(H)$. Then by Fatou's lemma and the lower semicontinuity of $y \mapsto |\Psi(y)|_H^2$ shown in Lemma \ref{Psi-truncation}, we have
$$I(\gamma) = \frac{1}{2}\int_H |\Psi(y)|_H^2\, \gamma(dy) \leq \liminf_{n\to \infty} \frac{1}{2}\int_H |\Psi(y)|_H^2\, \gamma_n(dy) = \liminf_{n\to \infty} I(\gamma_n).$$
The desired lower semicontinuity follows and completes the proof of the theorem.

\end{proof}

\section*{Acknowledgements} 
This material is based upon work supported by the National Science
Foundation under Grant No. DMS-2424139, while the  authors were in
residence at the Simons Laufer Mathematical Sciences Institute in
Berkeley, California, during the Fall 2025 semester.
AB would also like to acknowledge support from NSF RTG grant DMS-2134107.

\end{document}